\documentclass[12pt, a4paper]{article}%\usepackage{CJK}
\usepackage{amsmath,amssymb,amsthm,amsfonts,latexsym,graphicx,subfigure}
\usepackage[all,import]{xy}
\usepackage[numbers,sort&compress]{natbib}
\usepackage{indentfirst}
\usepackage{fancyhdr}
\usepackage{bbm}
\usepackage{enumerate,tikz}
\usepackage{caption}
\usepackage{accents,cases}
\usepackage{color,tikz}
\usetikzlibrary{matrix}
\usepackage{multirow}
\usepackage[normalem]{ulem}
\useunder{\uline}{\ul}{}
\usepackage{hyperref}
\usepackage{tabularx,booktabs}
\usepackage{caption}
\usepackage{adjustbox}
\usepackage{blindtext}
\usepackage{etex}
\usepackage{amsmath}
\usepackage{amsfonts}
\usepackage{accents,cases}
\usepackage{algorithm} 
\usepackage{algorithmicx}
\usepackage{algpseudocode} 
\usepackage{graphicx} 
\usepackage{epstopdf}
\usepackage{amscd}
\usepackage{amsthm}
\usepackage{amssymb}
\usepackage{bm}
\usepackage{latexsym,float}
\usepackage{makecell}
\usepackage{multicol, multirow}
\usepackage[figuresright]{rotating}
\usepackage{longtable}
\usepackage{booktabs}
\usepackage{lscape}
\usepackage{hyperref}
\usepackage{tikz}
\usetikzlibrary{shapes,arrows}

\usepackage{graphics}
\usepackage{array}
\usepackage{enumerate}
\usepackage{color,tikz}
\usetikzlibrary{calc,intersections}

\usepackage{algorithm,algpseudocode,float}
\usepackage{lipsum}

\makeatletter
\newenvironment{breakablealgorithm}
{% \begin{breakablealgorithm}
		\begin{center}
			\refstepcounter{algorithm}% New algorithm
			\hrule height.8pt depth0pt \kern2pt% \@fs@pre for \@fs@ruled
			\renewcommand{\caption}[2][\relax]{% Make a new \caption
				{\raggedright\textbf{\ALG@name~\thealgorithm} ##2\par}%
				\ifx\relax##1\relax % #1 is \relax
				\addcontentsline{loa}{algorithm}{\protect\numberline{\thealgorithm}##2}%
				\else % #1 is not \relax
				\addcontentsline{loa}{algorithm}{\protect\numberline{\thealgorithm}##1}%
				\fi
				\kern2pt\hrule\kern2pt
			}
		}{% \end{breakablealgorithm}
		\kern2pt\hrule\relax% \@fs@post for \@fs@ruled
	\end{center}
}
\makeatother

\usetikzlibrary{arrows,chains,matrix,positioning,scopes}
\makeatletter
\tikzset{join/.code=\tikzset{after node path={%
			\ifx\tikzchainprevious\pgfutil@empty\else(\tikzchainprevious)%
			edge[every join]#1(\tikzchaincurrent)\fi}}}
\makeatother
\tikzset{>=stealth',every on chain/.append style={join},
	every join/.style={->}}
\tikzstyle{labeled}=[execute at begin node=$\scriptstyle,
execute at end node=$]
\usepackage{wasysym}
\usepackage{hyperref}
\usepackage{pdfsync}
\usepackage{comment}
\usepackage[all]{xy}

\allowdisplaybreaks[4]

\usepackage{tikz-cd}
\usetikzlibrary{matrix, calc, arrows}

\DeclareMathOperator{\coker}{coker}

\usepackage{array}
\newcommand{\PreserveBackslash}[1]{\let\temp=\\#1\let\\=\temp}
\newcolumntype{C}[1]{>{\PreserveBackslash\centering}p{#1}}
\newcolumntype{R}[1]{>{\PreserveBackslash\raggedleft}p{#1}}
\newcolumntype{L}[1]{>{\PreserveBackslash\raggedright}p{#1}}

\renewcommand{\arraystretch}{1.3}

\DeclareMathOperator*{\argmin}{\ensuremath{arg\,min}}
\DeclareMathOperator*{\argmax}{\ensuremath{arg\,max}}
\DeclareMathOperator*{\argopt}{\ensuremath{arg\,opt}}
\DeclareMathOperator*{\argopto}{\ensuremath{arg\,\overline{opt}}}
\DeclareMathOperator*{\argoptt}{\ensuremath{arg\,\widetilde{opt}}}
\DeclareMathOperator*{\median}{\ensuremath{median}}
\DeclareMathOperator*{\sgn}{\ensuremath{Sgn}}
\DeclareMathOperator*{\card}{\ensuremath{Card}}
\DeclareMathOperator*{\sign}{\ensuremath{sign}}
\DeclareMathOperator*{\cone}{\ensuremath{cone}}
\DeclareMathOperator*{\cen}{\ensuremath{center}}
\DeclareMathOperator*{\mean}{\ensuremath{mean}}
\DeclareMathOperator*{\col}{\ensuremath{col}}
\DeclareMathOperator*{\opt}{\ensuremath{opt}}

\makeatletter
\def\wbar{\accentset{{\cc@style\underline{\mskip8mu}}}}
\newcommand{\wcheck}{\raisebox{-1ex}{$\check{}$}}
\makeatother
\def\pd#1#2{\frac{\partial #1}{\partial #2}}
\renewcommand{\vec}[1]{\mbox{\boldmath \small $#1$}}
\def\d{\mathrm{d}}
\def\mi{\mathtt{i}}
\def\sech{\mathrm{sech}} %added by xj
\def\me{\mathrm{e}} %added by xj
\newcommand{\pp}[2]{\frac{\partial{#1}}{\partial{#2}}}
\newcommand{\dd}[2]{\frac{\dif{#1}}{\dif{#2}}}% added by xj
\newcommand{\bmb}[1]{\left(#1\right)}
\newcommand{\mat}[1]{\mathbf{#1}}
\def\disp{\displaystyle}
\def\dif{\mathrm{d}}
\newcommand{\eucn}[1]{\left|{#1}\right|}
\newcommand{\bra}[1]{\left|{#1}\right\rangle}
\newcommand{\ket}[1]{\left\langle{#1}\right|}
\newcommand{\ketbra}[2]{\left\langle{#1}|{#2}\right\rangle}
\newcommand{\trace}{\mathrm{Tr}}
\newcommand{\diag}{\mathrm{diag}}
\newcommand{\cut}{\mathrm{cut}}
\newcommand{\imag}{\mathrm{Im}}
\newcommand{\vol}{\mathrm{vol}}
\newcommand{\real}{\mathrm{Re}}
\newcommand{\ansatz}{\textit{ansatz}{ }}
\newcommand{\videpost}{\textit{vide post}{}}
\newcommand{\etc}{\textit{etc}{}}
\newcommand{\etal}{\textit{et al}{}}
\newcommand{\ie}{\textit{i.e.}{~}}
\newcommand{\eg}{\textit{e.g.}{~}}
\newcommand{\vs}{\textit{v.s.}{~}}
\newcommand{\randn}{\ensuremath{\mathrm{randn}}}
\newcommand{\dist}{\ensuremath{\mathrm{dist}}}
\newcommand{\score}{\ensuremath{\mathrm{score}}}
\newcommand{\labeling}{\ensuremath{\mathrm{lab}}}

\newcommand{\vsigma}{\mbox{\boldmath$\sigma$}}
\newcommand{\vgamma}{\vec{\gamma}}
\newcommand{\vGamma}{\vec{\Gamma}}
\newcommand{\vPsi}{\vec{\Psi}}
\newcommand{\vPhi}{\vec{\Phi}}
\newcommand{\vpsi}{\vec{\psi}}
\newcommand{\veta}{\vec{\eta}}
\newcommand{\vJ}{\vec{J}}
\newcommand{\vT}{\vec{T}}
\newcommand{\vK}{\vec{K}}
\newcommand{\vE}{{E}}
\newcommand{\vI}{\vec{I}}
\newcommand{\vdelta}{{\delta}}
\newcommand{\vell}{{\ell}}
\newcommand{\power}{\ensuremath{\mathcal{P}}}

\newcommand{\cL}{{\cal L}}

\newcommand{\shao}[1]{\textcolor{red}{#1}}

\numberwithin{equation}{section}

\theoremstyle{plain}

\newtheorem{defn}{Definition}

\newtheorem{prop}{Proposition}

\def\Lindelof{Lindel\"of}
\def\Holder{H\"older}

\newcommand{\blue}{\textcolor{blue}}

\begin{document}
	\bibliographystyle{plain} %plain apalike alpha unsrt
	%\title{A Note on $1$-spectral clustering}
	%\title{Dual Cheeger Constant and Inverse Power Method for Max Cut}
	\title{A Parallel Evolutionary Algorithm Framework for Graph $k$-CUT Problems}
	\author{Sihong Shao\footnotemark[1],
		\and Chuan Yang\footnotemark[2]
	}
	\renewcommand{\thefootnote}{\fnsymbol{footnote}}
	\footnotetext[1]{CAPT, LMAM and School of Mathematical Sciences, Peking University, Beijing 100871, China.
		Email: {\tt sihong@math.pku.edu.cn} (Sihong Shao).
	}
	\footnotetext[2]{School of Mathematics and Statistics, Fuzhou University, Fuzhou 350108, China, and School
of Mathematical Sciences, Peking University, Beijing 100871, China.
		Email: {\tt chuanyang@fzu.edu.cn} (Chuan Yang).
	}
	
	%\footnotetext[1]{To
		%whom correspondence should be addressed. Email:
		%\texttt{sihong@math.pku.edu.cn}}
	%\date{November 13, 2018}
	\date{}
	\maketitle
	
	\begin{abstract}
    Graph $k$-CUT problems include many important variants whose objectives combine
cut value, volume, and cardinality terms in different ways. Most existing algorithms are
designed for individual formulations, which limits their transferability across
related models. In this paper, we organize a broad family of graph partitioning
problems into two classes, MaxGCP and MinGCP, according to their optimization
orientation and balance-related structure. Based on this classification, we propose a unified
Parallel Evolutionary Algorithm Framework (PEAF). This framework combines structure-inheriting crossover operators, a hierarchical
mutation mechanism based on the Multiple Mutation Heuristic (MMH) and the Auxiliary Cut Mutation Heuristic (ACMH), and a diversity-preserving selection
strategy. Extensive experiments on G-set with $k \in \{2,3,4,5\}$ show that
PEAF-ACMH consistently outperforms Gurobi on nine representative $k$-CUT
problems. For MaxGCP, PEAF-ACMH improves several best-known solutions
for \textsc{Max}-$k$-\textsc{Cut} with $k \ge 3$, and through numerical bounds
derived from its relation to \textsc{Max}-$k$-\textsc{Cut}, verifies the high
quality of the obtained solutions for \textsc{Judicious}-$k$-\textsc{Partition}
and \textsc{AntiCheeger}-$k$-\textsc{Cut}. The results further indicate that
\textsc{Judicious}-$k$-\textsc{Partition} usually yields more balanced
partitions than \textsc{AntiCheeger}-$k$-\textsc{Cut}. For MinGCP, theoretical and computational comparisons show that
\textsc{Cheeger}-$k$-\textsc{Cut} and \textsc{Sparsest}-$k$-\textsc{Cut} produce
more balanced partitions than \textsc{Normalized}-$k$-\textsc{Cut} and
\textsc{Ratio}-$k$-\textsc{Cut}, respectively. PEAF-ACMH also obtains highly
similar partitions for \textsc{Min}-$k$-\textsc{Cut} and
\textsc{MinMax}-$k$-\textsc{Cut} within short running times, providing numerical
evidence for their structural affinity. These results demonstrate that PEAF
is both an effective unified solver and a useful tool for revealing structural
properties of graph $k$-CUT models.

		\vspace{0.3cm}
		
		\noindent\textbf{Keywords:}
        Evolutionary algorithm;
		Cheeger cut;
		Maxcut;
		Normalized cut;
		Mincut;
		Ratio cut;
		Minmax cut;
		Sparsest cut;
		Judicious partition;
		AntiCheeger cut
		
		\vspace{0.3cm}
		
		\noindent\textbf{AMS subject classifications:}
        68W50;
        90C27;
        05C85;
        90C59;
        68W10
        
		%05C85;
		%90C27;
		%58C40;
		%35P30;
		%05C50

	\end{abstract}
	
	\tableofcontents
	\listoftables

	\section{Introduction}
	\label{sec:intro}

	Graph $k$-CUT problems aim to partition the vertices of a graph into $k$ ($k \geq 2$) disjoint subsets by removing the edges among them. These problems encompass a variety of graph cut types, including \textsc{Max}-$k$-\textsc{Cut} \cite{kann1997hardness}, \textsc{Judicious}-$k$-\textsc{Partition} \cite{bollobas2002problems}, \textsc{AntiCheeger}-$k$-\textsc{Cut} \cite{xu2016graph}, \textsc{Min}-$k$-\textsc{Cut} \cite{goldschmidt1994polynomial}, \textsc{MinMax}-$k$-\textsc{Cut} \cite{chandrasekaran2023fixed}, \textsc{Normalized}-$k$-\textsc{Cut} \cite{shi2000normalized}, \textsc{Ratio}-$k$-\textsc{Cut} \cite{dhillon2007weighted}, \textsc{Cheeger}-$k$-\textsc{Cut} \cite{louis2012many}, and \textsc{Sparsest}-$k$-\textsc{Cut} \cite{arora2009expander,bresson2013multiclass}. For simplicity, we categorize these problems into two main groups: Maximal Graph Cut Problems (MaxGCP) and Minimal Graph Cut Problems (MinGCP) (see Definitions \ref{def:maxgcp} and \ref{def:mingcp}). The distinction is based on the preference for larger cuts in MaxGCP and smaller cuts in MinGCP. Both categories have garnered significant attention due to their wide-ranging applications. 	MaxGCP are particularly valuable in fields such as statistical physics \cite{barahona1988application, liers2004computing}, VLSI circuit layout design \cite{chang1987efficient, cho1998fast}, and machine learning \cite{jun2013semi}. MinGCP are extensively applied in image segmentation \cite{boykov2001interactive, estrada2005quantitative, boykov2006graph}, stereo vision \cite{buehler2002minimal, ishikawa1998occlusions}, clustering \cite{bresson2013multiclass, dhillon2004kernel, pei2020efficient}, community detection \cite{newman2013spectral, mu2019multi}, data classification \cite{merkurjev2017modified}, and data mining \cite{hochbaum2016sparse}.

	With the exception of \textsc{Min}-$k$-\textsc{Cut} \cite{goldschmidt1994polynomial} and \textsc{MinMax}-$k$-\textsc{Cut} \cite{chandrasekaran2023fixed} (for fixed $k$), the aforementioned nine graph $k$-cut problems are NP-hard \cite{kann1997hardness, shahrokhi1994complexity, shao2021continuous, shi2000normalized, hagen1992new, hochbaum2013polynomial, matula1990sparsest}. For these NP-hard variants, exact algorithms are often restricted to small-scale instances due to their exponential computational complexity. For instance, the branch-and-bound approach proposed by Krislock et al. \cite{krislock2014improved}, despite incorporating quadratic regularization and quasi-Newton methods, is limited to \textsc{Max}-$2$-\textsc{Cut} on graphs with fewer than 500 vertices. Consequently, heuristic methods are indispensable for tackling large-scale sparse graphs encountered in real-world applications. Significant efforts have been directed toward developing robust heuristics for these challenging problems. Within the MaxGCP framework, diverse strategies have been proposed for \textsc{Max}-$2$-\textsc{Cut}, ranging from tabu search \cite{arraiz2009competitive} and scatter search \cite{Marti2009} to GRASP, variable neighborhood search (VNS), and path-relinking \cite{Festa2002randomized}. Other notable approaches include simple iterative algorithms \cite{shao2009maxcut} and rank-two relaxation heuristics \cite{burer2002rank}. Similarly, for MinGCP variants such as \textsc{Ratio}-$k$-\textsc{Cut} and \textsc{Normalized}-$k$-\textsc{Cut}, Palubeckis \cite{palubeckis2022metaheuristic} introduced multistart simulated annealing and memetic algorithms. Furthermore, recent advancements for \textsc{Cheeger}-$2$-\textsc{Cut} feature breakout local search \cite{lu2019stagnation}, hybrid evolutionary algorithms \cite{lu2020hybrid}, and parallel Markov Chain Monte Carlo methods \cite{chen2025monte}.

	Research on heuristic algorithms for various graph cut problems has been disproportionate, leaving a notable gap between theoretical investigation and algorithmic development. For instance, while \textsc{Max}-$k$-\textsc{Cut} has been extensively studied, numerical results for \textsc{Judicious}-$k$-\textsc{Partition} remain scarce despite decades of combinatorial interest \cite{bollobas2002problems, bollobas2004judicious, lee2016judicious, hou2021bisections}. To bridge this gap and establish a unified computational framework, we propose an efficient parallel evolutionary algorithm framework for generalized $k$-cut problems. Our approach addresses variables defined by boundaries~\eqref{eq::boundary}, volumes~\eqref{eq::volume}, and cardinalities~\eqref{eq::card}, which are systematically categorized into MaxGCP and MinGCP (see Section~\ref{sec::notation}).

	The primary contributions of this work are summarized as follows:	
	\begin{itemize}
		\item \textbf{A High-Performance Multiple Mutation Heuristic (MMH):} We develop MMH (see Algorithm~\ref{algorithm::combinatorial} and Section~\ref{sec::mutation}) as a versatile search engine for generalized $k$-\textsc{Cut} problems. MMH integrates nine search operators organized into three hierarchies: local search ($\widetilde{O}_1, \widetilde{O}_2, \widehat{O}_1, \widehat{O}_2$), tabu search ($\widetilde{O}_3, \widetilde{O}_4, \widehat{O}_3, \widehat{O}_4$), and strong random perturbation ($O_5$) (see Sections~\ref{sec::o1o4}, \ref{sec::tabu} and \ref{sec::perturb}). The efficiency is underpinned by a well-designed bucket sorting structure that maintains two stable move-gain matrices (see Eqs.~\eqref{delta::dcut} and \eqref{delta::dpartial}), enabling rapid tracking of changes in boundaries~\eqref{eq::boundary} and cut values~\eqref{eq::cut-value}. Furthermore, we establish the computational complexities (Proposition~\ref{thm::o1comp}) and provide quality guarantees specifically for \textsc{Max}-$k$-\textsc{Cut} and \textsc{Min}-$k$-\textsc{Cut} (Proposition~\ref{thm::maxkcut}).		
		\item \textbf{An Enhanced Auxiliary Cut Mutation Heuristic (ACMH):} To improve search robustness, we propose ACMH (see Section~\ref{sec::cut-combined} and Algorithm~\ref{algorithm::acmh}), an advanced variant of MMH. ACMH adaptively rotates between the primary objective and strategically selected auxiliary cut problems. This multi-tasking mutation strategy facilitates escaping local optima by observing whether superior solutions to the primary problem emerge from the landscape of auxiliary tasks.		
		\item \textbf{A Synergized Parallel Evolutionary Framework (PEAF):} We integrate the proposed heuristics into the mutation phase in PEAF (see Algorithm~\ref{algorithm::framework}), resulting in two algorithms: PEAF-MMH and PEAF-ACMH. This framework coordinates three distinct crossover mechanisms (see Section~\ref{sec::crossover}) involving five operators: $C_1$ utilizes a scoring-based inheritance (Section~\ref{sec::c1}); $C_{2,1}$ and $C_{2,2}$ employ greedy subgraph reconstruction (Section~\ref{sec::c2c3}); and $C_{3,1}, C_{3,2}$ implement path-relinking to explore mutual paths between reference solutions (Section~\ref{sec::c4c5}). Finally, a solution pool selection mechanism (Section~\ref{sec::update} and Algorithm~\ref{algorithm::selection}). These three phases work together to inherit shared structural characteristics from parent solutions while maintaining a robust balance between population quality and genetic diversity.	
\item \textbf{Extensive Benchmarking and Structural Discovery:}
        Comprehensive experiments on G-set\footnote{Downloaded from \href{https://web.stanford.edu/~yyye/yyye/Gset/}{https://web.stanford.edu/$\sim$yyye/yyye/Gset/}} for nine $k$-cut variants 
($k \in \{2,3,4,5\}$) demonstrate that PEAF-ACMH consistently outperforms the
Gurobi solver. For MaxGCP, PEAF-ACMH improves a number of best-known
solutions for \textsc{Max}-$k$-\textsc{Cut} with $k \ge 3$ reported in
\cite{Ma2015}. More importantly, by using these best-known
\textsc{Max}-$k$-\textsc{Cut} values as reference results and carefully exploiting
the close relationships among \textsc{Max}-$k$-\textsc{Cut},
\textsc{Judicious}-$k$-\textsc{Partition}, and
\textsc{AntiCheeger}-$k$-\textsc{Cut}, we derive numerical bounds
(Eqs.~\eqref{neq:jp} and \eqref{neq:ac}) that validate the high quality of the
solutions obtained for the latter two problems
(Tables~\ref{tab:jp-quality} and \ref{tab:ac-quality}). These results further
show that \textsc{Judicious}-$k$-\textsc{Partition} usually produces more
balanced partitions than \textsc{AntiCheeger}-$k$-\textsc{Cut}
(Section~\ref{sec::result-anticheeger}). For MinGCP, our theoretical and computational analyses show that
\textsc{Cheeger}-$k$-\textsc{Cut} and \textsc{Sparsest}-$k$-\textsc{Cut} tend to
produce more balanced partitions than \textsc{Normalized}-$k$-\textsc{Cut} and
\textsc{Ratio}-$k$-\textsc{Cut}, respectively
(Sections~\ref{sec::cheeger-k-cut} and \ref{sec::sparsest-k-cut}). We also
observe that PEAF-ACMH obtains very similar partitions for
\textsc{Min}-$k$-\textsc{Cut} and \textsc{MinMax}-$k$-\textsc{Cut}, while solving
both problems within very short running times. This provides numerical evidence
for the structural affinity between the two problems. It is also consistent with
the known complexity landscape: \textsc{Min}-$k$-\textsc{Cut} has long been known
to be polynomial-time solvable \cite{goldschmidt1994polynomial}, whereas the
corresponding complexity of \textsc{MinMax}-$k$-\textsc{Cut} was settled only
recently \cite{chandrasekaran2023fixed}. 
% These observations indicate that
% PEAF-ACMH is not only competitive computationally, but also useful for exposing
% and validating structural properties of combinatorial $k$-cut problems.
% 		\item \textbf{Extensive Benchmarking and Structural Discovery:} Comprehensive experiments on G-set\footnote{Downloaded from \href{https://web.stanford.edu/~yyye/yyye/Gset/}{https://web.stanford.edu/$\sim$yyye/yyye/Gset/}} for nine $k$-cut variants ($k \in \{2,3,4,5\}$) demonstrate that PEAF-ACMH consistently outperforms the Gurobi solver. Notably, for \textsc{Max}-$k$-\textsc{Cut} ($k \ge 3$), our results refresh a number of best-known solutions in \cite{Ma2015}. Leveraging these results, we derive numerical bounds (Eqs.~\eqref{neq:jp} and \eqref{neq:ac}) that confirm the superior performance of PEAF-ACMH for \textsc{Judicious}-$k$-\textsc{Partition} and \textsc{AntiCheeger}-$k$-\textsc{Cut} (see Tables~\ref{tab:jp-quality} and \ref{tab:ac-quality}). Furthermore, our high-quality results reveal that \textsc{Judicious}-$k$-\textsc{Partition} typically achieves higher balance than \textsc{AntiCheeger}-$k$-\textsc{Cut} (see Section \ref{sec::result-anticheeger}), and both \textsc{Cheeger}-$k$-\textsc{Cut} and \textsc{Sparsest}-$k$-\textsc{Cut} provide more balanced partitions than \textsc{Normalized}-$k$-\textsc{Cut} and \textsc{Ratio}-$k$-\textsc{Cut}, respectively (see Sections~\ref{sec::cheeger-k-cut} and \ref{sec::sparsest-k-cut}).
	\end{itemize}
	
	The rest of this paper is organized as follows. Section~\ref{sec::notation} introduces the mathematical formulations and notations for the nine graph $k$-CUT problems. Section~\ref{sec::peaf} details the proposed parallel evolutionary algorithm framework, including crossover, mutation and selection mechanisms. In Section~\ref{sec::experiment}, we conduct extensive numerical experiments on the G-set and analyze the results. Finally, Section~\ref{sec::conclusion} concludes the paper and discusses potential avenues for future research.

	\section{Graph $k$-CUT problems}
	\label{sec::notation}

	Given an undirected graph $G=(V,E)$, $V=[n] := \{1,\ldots,n\}$ and $E\subset V\times V$ are the sets of  vertices and edges, respectively. Each edge $\{u,v\}\in E$ is associated with a positive integer weight, denoted as $w_{uv}\in\mathbb{Z}^+$, and $d_u=\sum_{\{u,v\}\in E}w_{uv}$ represents the degree of the $u$-th vertex. For a nonempty subset $S\subset V$, $\vol(S)$ sums the degrees of all vertices in $S$
	and $|S|$ gives the number of vertices contained in $S$. Let $\vec S$ be a family of nonempty subsets of $V$ and $E(\vec S)$  collect all boundaries among these subsets
	\begin{equation}
		E(\vec S)=\left\{\{u,v\}\in E:\, u\in I, v\in J, \forall\, I\ne J, I\in\vec S, J\in\vec S\right\}.
	\end{equation}
	Accordingly, we define the boundary of $S$, $\partial S = E(\{S, S^c\})$, with $S^c$ being the complement of $S$ in $V$,
	and 
	\begin{equation}
		\cut(\partial S) = \sum_{\{u,v\}\in \partial S} w_{uv}.
	\end{equation}

	We call $E(\vec S)$ a $k$-CUT of $G$ if $\vec S= \{S_1,\ldots,S_k\}$ partitions $V$ into $k~(\ge 2)$ disjoint nonempty subsets,
	and the corresponding cut value is 
	\begin{equation}
		\label{eq::cut-value}
		\cut(\vec S) = \sum_{\{u,v\}\in E(\vec S)} w_{uv},
	\end{equation}
	which happens to be the objective function for both \textsc{Max}-$k$-\textsc{Cut} \eqref{eq::max-k-cut} and \textsc{Min}-$k$-\textsc{Cut} \eqref{eq::min-k-cut}. For any partition $\vec S=\{S_1,\ldots,S_k\}$, we call its subset $S_p$ a cut segment.
	Moreover, the set of all $k$-CUTs of $G$ is denoted as $\mathbb{E}$.
	For the partition $\vec S= \{S_1,\ldots,S_k\}$ of $V$, we make the following conventions 
	\begin{align}
		\cut(\partial \vec S)&=\{\cut(\partial S_1), \cut(\partial S_2),\dots, \cut(\partial S_k)\},\label{eq::boundary}\\
		\vol(\vec S)&=\{\vol(S_1),\vol(S_2),\ldots,\vol(S_k)\},\label{eq::volume}\\
		|\vec S|&=\{ |S_1|,|S_2|,\ldots,|S_k|\}.\label{eq::card}
	\end{align}
	%	is  defined as a partition of $V$ into $k~(\ge 2)$ disjoint nonempty subsets $S_1\cup S_2\cup\ldots\cup S_k=V$, denoted as
	%	\begin{equation}
		%		\{S_1,\ldots,S_k\}=\left\{\{u,v\}\in E:\, u\in S_i,\,i\in [k],\,v\in S_j,\,j\in [k]\backslash\{i\}\right\}.
		%	\end{equation}
	%A natural measure for the cut value of $\{S_1,\ldots,S_k\}$ is to sum the weights of all its edges---
	%%		\begin{equation}
		%%			\label{eq::cut-value}
		%%			\cut(\{S_1,\ldots,S_k\})=\sum_{1\leq p<q\leq k}\sum_{i\in S_p,\,j\in S_q}w_{ij},
		%%		\end{equation}
	%		which clearly serves as the objective function for both MAX $k$-CUT \eqref{eq::max-k-cut} and \textsc{Min}-$k$-\textsc{Cut} \eqref{eq::min-k-cut}. 
	The objective functions of nine graph $k$-CUT problems studied in this paper are all composites of the aforementioned three set-valued functions towards the partition $\vec S=\{S_1,S_2,\ldots,S_k\}$ of $V$ and can be rewritten into the following abstract form 
	\begin{equation}
		\label{obj}
		F\left(\cut(\partial \vec S),\vol(\vec S),|\vec S|\right),
	\end{equation}
	%		where $\vol(\vec S)$ and $|\vec S|$ may be excluded (e.g. MAX $k$-CUT), but $\cut(\partial \vec S)$ must be included, owing to the fact that $\cut(\partial \vec S)$ directly characterize how closely the partitioned vertex subsets corresponds to each other. Hereafter, we use $f(A(\vec S))$ as the objective function for simplicity, where $\cut(\partial\vec S)\subseteq A(\vec S)\subseteq \cut(\partial \vec S)\cup\vol(\vec S)\cup|\vec S|$ represents the set of actual arguments required for the partition $\vec S$.
	with the corresponding optimization problems being 
	\begin{equation}
		\label{def:k-cut}
		\opt\limits_{\vec S\in\mathbb{E}} F(\cut(\partial \vec S),\vol(\vec S),|\vec S|), \quad \opt\in\{\min,\max\}.
	\end{equation}
	To be more specific, the nine graph $k$-CUT problems can be formulated as follows.
	\begin{align}
		\text{\textsc{Max}-$k$-\textsc{Cut}}&:&\max_{\vec S\in\mathbb{E}} \frac{1}{2}\sum_{p=1}^k \cut(\partial S_p).\label{eq::max-k-cut}\\
		\text{\textsc{Judicious}-$k$-\textsc{Partition}}&:&\min_{\vec S\in\mathbb{E}} \frac{1}{2}\max_{1\leq p\leq k} \{\vol(S_p)-\cut(\partial S_p)\}.\label{eq::jp-k-cut}\\
		\text{\textsc{AntiCheeger}-$k$-\textsc{Cut}}&: &\max_{\vec S\in\mathbb{E}}\min_{1\leq p\leq k}\left\{\frac{\cut(\partial S_p)}{\vol(S_p)}\right\}.\label{eq::ah-k-cut}\\
		\text{\textsc{Min}-$k$-\textsc{Cut}}&:&\min_{\vec S\in\mathbb{E}} \frac{1}{2}\sum_{p=1}^k \cut(\partial S_p).\label{eq::min-k-cut}\\
		\text{\textsc{MinMax}-$k$-\textsc{Cut}}&:&\min_{\vec S\in\mathbb{E}} \max_{1\leq p\leq k} \{\cut(\partial S_p)\}.\label{eq::minmax-k-cut}\\
		\text{\textsc{Normalized}-$k$-\textsc{Cut}}&:&\min_{\vec S\in\mathbb{E}}\sum_{p=1}^k \frac{\cut(\partial S_p)}{\vol(S_p)}.\label{eq::normalized-k-cut}\\
		\text{\textsc{Ratio}-$k$-\textsc{Cut}}&:&\min_{\vec S\in\mathbb{E}}\sum_{p=1}^k \frac{\cut(\partial S_p)}{|S_p|}.\label{eq::ratio-k-cut}\\
		\text{\textsc{Cheeger}-$k$-\textsc{Cut}}&:&\min_{\vec S\in\mathbb{E}}\max_{1\leq p\leq k} \left\{\frac{\cut(\partial S_p)}{\min\{\vol(S_p),\vol(S_p^c)\}}\right\}.\label{eq::cheeger}\\
		\text{\textsc{Sparsest}-$k$-\textsc{Cut}}&:&\min_{\vec S\in\mathbb{E}}\sum_{p=1}^k \frac{\cut(\partial S_p)}{\min\{(k-1)|S_p|,|S_p^c|\}}.\label{eq::sparsest}
	\end{align}
	For simplicity, we abbreviate Problem~\eqref{def:k-cut} to ``$\opt\,\, F(\vec S)$'' hereafter.
	
	It can be readily observed from Eqs.~\eqref{eq::max-k-cut}-\eqref{eq::sparsest} that the nine objective functions all contain the set-valued function $\cut(\partial \vec S)$ which directly characterizes how closely the partitioned vertex subsets corresponds to each other. Using this fact, we would like to 
	categorize those nine graph cut problems into two types, for which we replace $\cut(\partial\vec S)$ with $\vec x=(x_1,x_2,\ldots,x_k)\in\mathbb{R}^k$ to extend the discrete function $F$ from $\mathbb{E}$ to $\mathbb{R}$ into a continuous function $\bar{F}$ from $\mathbb{R}^k$ to $\mathbb{R}$:
	\[
	\bar{F}(\vec x): =F(\vec x,\vol(\vec S),|\vec S|).
	\]
	Furthermore, we are able to define the partial derivative of $\bar{F}$ at the point $\vec x=(x_1,x_2,\ldots,x_k)\in\mathbb{R}^k$ with respect to the $p$-th variable $x_p$
	\begin{equation}
		\label{derivate:g}\nabla_p\, \bar{F}=\lim\limits_{h\rightarrow 0}\frac{\bar{F}(\ldots,x_{p-1},x_p+h,x_{p+1}\ldots)-\bar{F}(\ldots,x_{p-1},x_p,x_{p+1}\ldots)}{h},\quad \forall\,p\in [k].
	\end{equation}
	Finally,  the aforementioned nine problems can be categorized into two types: Maximal Graph Cut Problems (MaxGCP) and Minimal Graph Cut Problems (MinGCP). 
	\begin{defn}
		\label{def:maxgcp}
		A $k$-CUT problem belongs to MaxGCP if either $\nabla_p\, \bar{F}\leq 0$, $\forall\,\vec x\in\mathbb{R}^k$, $\forall\,p\in[k]$ holds for $\opt=\min$; or $\nabla_p\, \bar{F}\geq 0$, $\forall\,\vec x\in\mathbb{R}^k$, $\forall\,p\in[k]$ holds for $\opt=\max$.
	\end{defn}
	\begin{defn}
		\label{def:mingcp}
		A $k$-CUT problem belongs to MinGCP if either $\nabla_p\, \bar{F}\geq 0$, $\forall\,\vec x\in\mathbb{R}^k$, $\forall\,p\in[k]$ holds for $\opt=\min$; or $\nabla_p\, \bar{F}\leq 0$, $\forall\,\vec x\in\mathbb{R}^k$, $\forall\,p\in[k]$ holds for $\opt=\max$.
	\end{defn}		
	It is easy to verify that the first three cut problems in Eqs.~\eqref{eq::max-k-cut}-\eqref{eq::ah-k-cut} pertain to MaxGCP, and the remaining six in Eqs.~\eqref{eq::min-k-cut}-\eqref{eq::sparsest} belong to MinGCP. That is, no matter the balanced indicators --- volume~\eqref{eq::volume} and cardinality~\eqref{eq::card} show up or not in the objective functions, the basic pursuit of MaxGCP is to find solutions with large cut values~\eqref{eq::cut-value} and close boundary~\eqref{eq::boundary} connections, whereas MinGCP favors solutions with small cut values and loose boundary connections. Considering \textsc{Normalized}-$k$-\textsc{Cut}~\eqref{eq::normalized-k-cut} as an example, we have $\opt=\min$, $\bar{F}(\vec x)=\sum_{p=1}^k\frac{x_p}{\vol(S_p)}$ and $\nabla_p\, \bar{F}\geq 0$, $\forall\,\vec x\in\mathbb{R}^k$, $\forall\,p\in[k]$, suggesting that minimizing $\bar{F}$ requires smaller values of $\vec x$. That implies the objective function~\eqref{obj} and the boundary sizes~\eqref{eq::boundary} either increase or decrease simultaneously, meaning that reduced values of $\cut(\partial\vec S)$ are more likely to yield improved objective function values for \textsc{Normalized}-$k$-\textsc{Cut}.

	\section{Parallel evolutionary algorithm framework}
	\label{sec::peaf}
	
	%\subsection{General working procedure}
	
	This paper presents a parallel evolutionary algorithm framework PEAF for solving the graph $k$-CUT problems that satisfy $\opt\, F\in\{\text{MaxGCP},\text{MinGCP}\}$. The primary objective of this framework is to maintain a population of high-quality solutions. Initially, PEA randomly selects pairs of individuals from the population and applies multiple crossover operators to exchange and combine their excellent features. Along this line, various mutation operators are employed to explore new regions of the solution space, thereby facilitating the potential discovery of improved solutions. Both the crossover and mutation processes are executed in parallel to enhance the computational efficiency. The formation of the new population involves selecting solutions from the union of the original population and the newly generated offspring, which ensures that the new population is not only of high quality but also of great diversity. The oveall procedure is outlined in Algorithm~\ref{algorithm::framework}.

	\begin{algorithm}[hbtp]
		\caption{A parallel evolutionary algorithm framework (PEAF) for $\opt F$ on $G=(V,E)$.} \label{algorithm::framework}
		\begin{algorithmic}[1] \small
			\Require The population size $m$, the number of possible solution pairs $r$, the time limit for stopping. 					\Ensure An approximate solution $I^*$ to the graph $k$-CUT problem $\opt F$ on $G$
			\State \textbf{Initialization}: $\mathbb{I}\gets\Call{generate\_population}{G,m}$ \Comment{Section \ref{sec::initial}} 
			\While {the stopping condition is not satisfied}
			\State $\mathbb{C}\gets\Call{parallel\_random\_crossover}{\mathbb{I},r}$ 
			\Comment{Section \ref{sec::crossover}}
			\State $ \mathbb{M}\gets\Call{parallel\_mutation}{\mathbb{C}}$ 
			\Comment{Section \ref{sec::mutation}}
			\State $\mathbb{I}\gets \Call{update\_population}{\mathbb{M}\cup\mathbb{I}}$ 
			\Comment{Section \ref{sec::update}}
			\EndWhile
			\State $I^*\gets \Call{best}{\mathbb{I}}$
			\State \Return{$I^*$}
		\end{algorithmic}  
	\end{algorithm}

	The initial population $\mathbb{I}$, whose generation is detailed in Section~\ref{sec::initial}, comprises $m=|\mathbb{I}|$ solutions, and each individual is a partition of $V$. Algorithm \ref{algorithm::framework} iterates repeatedly until a specified time limit is reached and  outputs the best solution $I^*$ in the final solution pool. In each iteration, we first randomly pick $r$ pairs of reference solutions from the current population, and each pair is applied with a randomly selected crossover operator (see Section~\ref{sec::crossover}) to generate two offspring solutions that preserve the common excellent traits. These $r$ runs are executed in parallel and output the offspring solution set $\mathbb{C}$ with $|\mathbb{C}|=2r$. Each solution of $\mathbb{C}$ is input to the mutation phase (see Section \ref{sec::mutation}) to pursue deep improvement. All these $2r$ runs are also executed in parallel and output the offspring solution set $\mathbb{M}$ with $|\mathbb{M}|=2r$. The new population for the next iteration is selected from $\mathbb{M}\cup\mathbb{I}$, the merged set of the existing solution pool and $\mathbb{M}$, as detailed in Section \ref{sec::update}.
	
	%		Furthermore, each solution  $\{I_{l_{\tau}}^{\prime}\}_{\tau=1}^p\cup \{I_{r_{\tau}}^{\prime}\}_{\tau=1}^p$ is employed with nine mutation operators (see Section \ref{sec::mutation}). The flexibility of the general algorithm scheme allows for enhancements in the mutation phase through slight modifications, which have been experimentally validated to increase efficiency (see Section \ref{sec::cut-combined}). The offspring solutions $\{I_{l_{\tau}}^{*}\}_{\tau=1}^p\cup \{I_{r_{\tau}}^{*}\}_{\tau=1}^p$ generated after the mutation phase are then merged with the existing solution pool $\{I_t\}_{t=1}^m$. The new population for the next generation is selected using the updating rule described in Section \ref{sec::update}.

	\subsection{Initial population}
	\label{sec::initial}
	
	We adopt the following random initialization procedure: $k$ different vertices are randomly selected and allocated separately to $k$ subsets (thus empty subsets are avoided), then each of the remaining vertices is randomly assigned to one of those $k$ subsets. A special care should be taken sometimes for the isolated vertices in MinGCP. 
	For example, if the number of isolated vertices is no less than $k$ in \textsc{Min}-$k$-\textsc{Cut}, \textsc{MinMax}-$k$-\textsc{Cut}, \textsc{Ratio}-$k$-\textsc{Cut}, and \textsc{Sparsest}-$k$-\textsc{Cut}, it will result in a trivial solution with the objective function value being $0$. Thus, an additional step should be included to examine the number of isolated vertices.

	% for these problems.
	%
	%
	%
	%can be significantly impacted by the presence of isolated vertices. If its number is at least $k$, this will result in a trivial case with the objective function value being $0$. Thus, an additional step is included to examine the number of isolated vertices for these problems.} 

\subsection{Crossover}
\label{sec::crossover}

%The purpose of crossover is to produce offspring with exploration prospects by inheriting the favorable characteristics embedded in the reference solutions, in order to dig into the potential of the population as a whole, rather than relying solely on individuals. As the Algorithm \ref{algorithm::framework} progresses, the quality of the solution pool improves. Therefore, the shared traits of the references can be considered as indicators for obtaining high quality solutions to some extent. Consequently, the development of crossover operators focuses on identifying the shared components and addressing the differing parts.
%The shared components between reference solutions are defined as the points classified into the same categories. Prior to applying crossover operators, it is crucial to establish a one-to-one correspondence between the segmentations, which ensures the production of high-quality offspring. 

%The crossover operator produces offspring solutions by 

In order to combine shared traits from the reference (parent) solutions, one usually adopts some crossover operators to produce offspring solutions.
Before that, it is crucial to establish a one-to-one mapping among the cut segments of the reference solutions to identify the shared component, which is defined to be the induced subgraph of $G$ on the shared vertices that appear in the same labeled cut segments. Specifically, for any two reference solutions, $I^1=\{S_1^1,S_2^1,\ldots,S_k^1\}$ and $I^2=\{S_1^2,S_2^2,\ldots,S_k^2\}$, with one-to-one mapping $S_p^2\rightarrow S_p^1$, $\forall\,p\in[k]$, the shared component is denoted as $G[V_0]$ with $V_0=\cup_{p=1}^k(S_p^1\cap S_p^2)$ representing the union of the shared vertices and forms the backbone of $I^1$ and $I^2$. The shared component $G[V_0]$ is considered the combination of advantageous features (i.e.~the shared vertices and their connecting edges) from ${I}^1$ and ${I}^2$, and should be preserved in the offspring solutions, denoted by $\bar{I}^1$ and $\bar{I}^2$. This implies that the segment labels of the shared vertices remain unchanged in the offspring solutions, namely, 
\begin{equation}
	\label{eq:shared}
	\labeling(v,\bar{I}^1)=\labeling(v,\bar{I}^2)=\labeling(v, I^1)=\labeling(v,I^2),\quad \forall\, v\in V_0=\cup_{p=1}^k(S_p^1\cap S_p^2),
\end{equation}
where the segment label of the vertex $v$ for any solution $I=\{S_1,\ldots,S_k\}$ is denoted as $\labeling(v,I)$, which is equal to $p$ if $v\in S_p$. Therefore, the major role of any crossover operator is to determine the segment labels of the remaining vertices $V_0^c$ in the offspring solutions.

The first task is to establish the one-to-one mapping among the cut segments of the reference solutions. In order to enlarge the shared component as wide as possible, we tend to maximize the number of the shared vertices in a greedy manner, borrowing the idea in the greedy partition crossover operators utilized in graph segmentation \cite{Benlic2011partition,Wu2012} and graph coloring \cite{dorne1998new,galinier1999hybrid,hao2012memetic}. Here are the specific steps. Firstly, compute a matrix $\mathcal{S}\in\mathbb{N}^{k\times k}$ with each element being $\mathcal{S}_{ij}=|S_i^1\cap S_j^2|$. Secondly, match a position pair by $(i^*,j^*)\in\argmax_{ij}\{\mathcal{S}_{ij}\}$, 
ignore the $i^*$-th row and $j^*$-th column and repeat this step on the remaining part of $\mathcal{S}$ until all $k$ pairs are established. Thirdly, for each pair $(i,j)$, rearrange the $j$-th cut segment $S_j^2$ in $I^2$ to be the $i$-th segment $S_i^2$ and make the mapping $S_i^2\rightarrow S_i^1$. An example on this mapping process is given in Figure~\ref{fig::example}.

%For the graph in Figure~\ref{fig::example} and the reference solutions $I^1=\{\{a,b,c,d\},\{e,f\},$ $\{g,h,i\}\}$ plus $I^2=\{\{a,c\},\{b,f,g,h\},\{d,e,i\}\}$, we rearrange $I^2$ by $S_1^2=\{a,c\}$, $S_2^2=\{d,e,i\}$ and $S_3^2=\{b,f,g,h\}$, indicating the union of shared vertices is $\{a,c,e,g,h\}$. 

\begin{figure}[htbp]
	\centering
	\begin{tikzpicture}[scale=1.2]
		\tikzset{mypoints/.style={fill=white,draw=black,thick}}
		\def\ptsize{2.0pt}
		\def\a{3} \def\b{0.5} \def\c{1}
		%		\draw[name path=ellipse2,draw=black,fill=blue!10!white!30]
		%		(0,\a) circle[x radius = \a cm, y radius = 0.5*\a cm];
		%		\draw[name path=ellipse1,draw=black,fill=blue!10!white!30]
		%		(-\a,\b) circle[x radius = \c cm, y radius = \c cm];
		%		\draw[name path=ellipse3,draw=black,fill=blue!10!white!30]
		%		(\a,\b) circle[x radius = \c cm, y radius = \c cm];
		\coordinate[label=below:$f$] (f) at (-\a,0);
		\coordinate[label=above:$c$] (c) at (0,1.25*\a);
		\coordinate[label=above:$e$] (e) at (-\a,2*\b);
		\coordinate[label=right:$i$] (i) at (\a,2*\b);
		\coordinate[label=below:$h$] (h) at (\a,0);
		\coordinate[label=above:$b$] (b) at (0.5*\a,\a);
		\coordinate[label=above:$g$] (g) at (\a-1.5*\b,\b);
		\coordinate[label=above:$d$] (d) at (0,0.55*\a);
		\coordinate[label=above:$a$] (a) at (-0.5*\a,\a);
		%		\coordinate[label=above:$S_2$] (s2) at (-\a-0.5,\b-0.5);
		%		\coordinate[label=above:$S_3$] (s3) at (\a+0.5,\b-0.5);
		%		\coordinate[label=above:$S_1$] (s1) at (-0.5*\a,1.15*\a);
		\draw (a)--(b);
		\draw (a)--(c);
		\draw (a)--(d);
		\draw (a)--(i);
		\draw (b)--(c);
		\draw (b)--(e);
		\draw (b)--(f);
		\draw (b)--(i);
		\draw (c)--(e);
		\draw (c)--(i);
		\draw (c)--(h);
		\draw (d)--(e);
		\draw (e)--(f);
		\draw (e)--(h);
		\draw (f)--(g);
		\draw (f)--(h);
		\draw (f)--(i);
		\draw (h)--(i);
		\foreach \p in {a,b,c,d,e,f,g,h,i}
		\fill[mypoints] (\p) circle (\ptsize);
	\end{tikzpicture}
	\caption{An undirected graph with 9 vertices and 18 edges. For two reference solutions $I^1=\{\{a,b,c,d\},\{e,f\},$ $\{g,h,i\}\}$ and $I^2=\{\{a,c\},\{b,f,g,h\},\{d,e,i\}\}$, we rearrange $I^2$ by $S_1^2=\{a,c\}$, $S_2^2=\{d,e,i\}$ and $S_3^2=\{b,f,g,h\}$, indicating the union of shared vertices $V_0=\{a,c,e,g,h\}$. }
	\label{fig::example}
\end{figure}
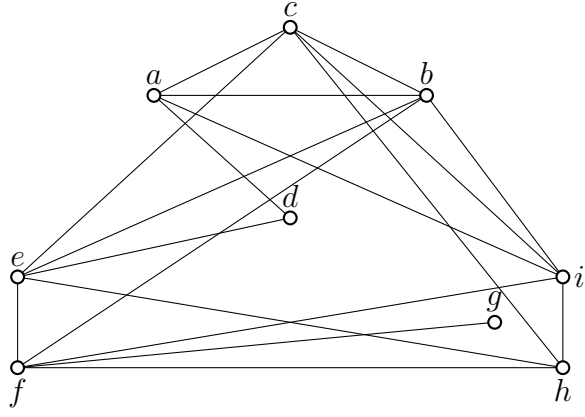

The following five crossover operators will be used in PEA to produce the offspring solutions $\bar{I}^1$, $\bar{I}^2$
from the parent solutions ${I}^1$, ${I}^2$. 
\begin{itemize}
	\item $C_1$: Implement a scoring strategy.
	\item $C_{2,1}$ and $C_{2,2}$: Utilize subgraph recovery methods and random generation techniques.
	\item $C_{3,1}$ and $C_{3,2}$: Employ the path-relinking algorithms to generate offspring solutions
\end{itemize}

\subsubsection{Crossover operator $C_1$}
\label{sec::c1}

The operator $C_1$, originally designed for \textsc{Max}-$2$-\textsc{Cut} \cite{Marti2009}, is now extended to address the general $k$-CUT problem ``$\opt\,F$", via scoring the reference solutions, $I^1$, $I^2$, and the segment labels of that with a higher score are more more likely being inherited. Assuming without loss of generality that $F(I^1)\geq F(I^2)$, 
we record the score for $I^1$ as
$$
\score^1=\left\{
\begin{aligned}
&\frac{F(I^1)}{F(I^1)+F(I^2)},&\text{if } \opt=\max,\\
&\frac{F(I^2)}{F(I^1)+F(I^2)},&\text{if } \opt=\min,
\end{aligned}
\right.
$$
and then the score for $I^2$ is $1-\score^1$. 
In consequence, each vertex $v$ in $\bar{I}^1$ is assigned the segment label $\labeling(v,\bar{I}^1)$ according to
\begin{equation}
\label{eq::label_i1}
\left\{
\begin{aligned}
	&\labeling(v,I^1),&\text{if } r\leq \score^1,\\
	&\labeling(v,I^2),&\text{if } r> \score^1,
\end{aligned}
\right.
\end{equation}
where $r$ is randomly selected from a uniform distribution $\mathcal{U}[0,1]$. On the other hand, each segment label of $\bar{I}^2$ is obtained from fixing $r=0.5$ in Eq.~\eqref{eq::label_i1}, thus $\bar{I}^2$ happens to be the better between $I^1$ and $I^2$.

\subsubsection{Crossover operators $C_{2,1}$ and $C_{2,2}$}
\label{sec::c2c3}

We extend the combination operator $CB_2$ proposed in \cite{Marti2009} to form two novel crossover operators, $C_{2,1}$ and $C_{2,2}$, tailored for general problem ``$\opt\,F$". Both start from the shared component $G[V_0]$ and proceed by greedily assigning the remaining vertices to the cut segments iteratively until recovering to the original graph $G$. For any subset $S\subseteq V$ and any partition $I=\{S_1,\ldots,S_k\}$ of $V$ on $G$, let $$I[S]=\{S\cap S_1,S\cap S_2,\ldots S\cap S_k\}$$ represent the induced partition of $S$ on the induced subgraph $G[S]$, thus we can still calculate the objective function value $F(I[S])$ on the subgraph $G[S]$. 

The generations of $\bar{I}^1$ by $C_{2,1}$ and $C_{2,2}$ are as follows.  Let $V_t$ denote the subset of vertices already assigned at the $t$-th step. The selected vertex $v^*$ and its label $\labeling(v^*,\bar{I}^1)$ are determined on $G[V_t]$ by
\begin{align}
\left(v^*,\labeling(v^*,\bar{I}^1)\right) &\in \argopt_{v\in V\backslash V_t,\,\labeling(v,\bar{I}^1)\in \mathcal{L}_2(v)}\left\{F(\bar{I}^1[V_t\cup\{v\}])\right\}, \\
\mathcal{L}_2(v) &=\left\{
\begin{aligned}
	&	\{\labeling(v,I^1),\labeling(v,I^2)\},& \text{when calling }C_{2,1},\\
	&[k], &\text{when calling }C_{2,2}.
\end{aligned}
\right. \label{eq:label_cb2}
\end{align}
Then we have $V_{t+1}\leftarrow V_t\cup\{v^*\}$ and an updated induced subgraph $G[V_{t+1}]$. Repeating this step for $|V_0^c|$ times restores the original graph $G$ and completes $\bar{I}^1$. 

The segment label $\labeling(v,\bar{I}^2)$ of each vertex $v\in V_0^c$ is determined by a random selection from $\mathcal{L}_2(v)$ (see Eq.~\eqref{eq:label_cb2}). 

Using 	\textsc{Normalized}-$3$-\textsc{Cut} on the graph in Figure \ref{fig::example} as an example, Figures~\ref{fig::c2} and \ref{fig::c3} illustrate the steps of applying $C_{2,1}$ and $C_{2,2}$ to generate $\bar{I}^1$, respectively. It is easy to verify that  $F(\bar{I}^1)<\min\{F(I^1),F(I^2)\}$ holds for both $C_{2,1}$ and $C_{2,2}$ in this example.

\begin{figure}[htbp]
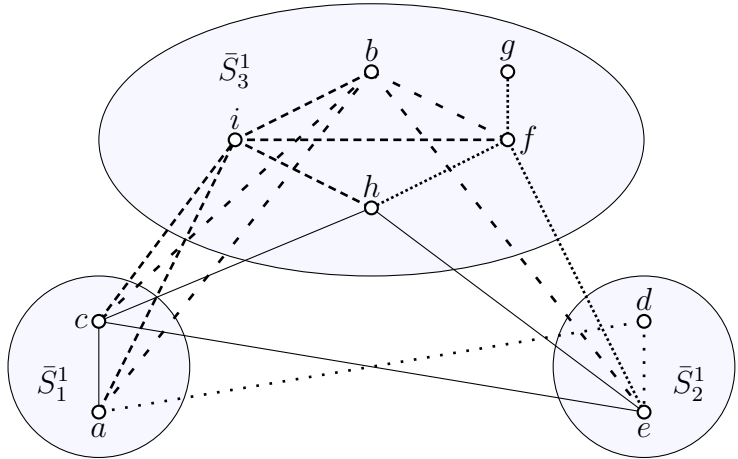

	\centering

	%		\captionsetup{singlelinecheck=off}
	\caption{	\textsc{Normalized}-$3$-\textsc{Cut}: Application of the crossover operator $C_{2,1}$ on the graph in Figure \ref{fig::example}. It consists of four steps: (1) $V_1\leftarrow V_0\cup\{f\}$ by connecting $f$ to
		$e,g,h$ in densely dotted lines and $\labeling(f,\bar{I}^1)=3$;
		(2) $V_2\leftarrow V_1\cup\{d\}$ by connecting $d$ to $a,e$ in loosely dotted lines and $\labeling(d,\bar{I}^1)=2$;
		(3) $V_3\leftarrow V_2\cup\{b\}$ by connecting $b$ to $a,c,e,f$ in loosely dashed lines and $\labeling(b,\bar{I}^1)=3$;
		(4) $V=V_4\leftarrow V_3\cup\{i\}$ by connecting $i$ to $a,b,c,f,h$ in densely dashed lines and $\labeling(i,\bar{I}^1)=3$. Finally, we have $\bar{I}^1=\{\bar{S}_1^1, \bar{S}_2^1, \bar{S}_3^1\}$.}
	\label{fig::c2}
\end{figure}

\begin{figure}[htbp]
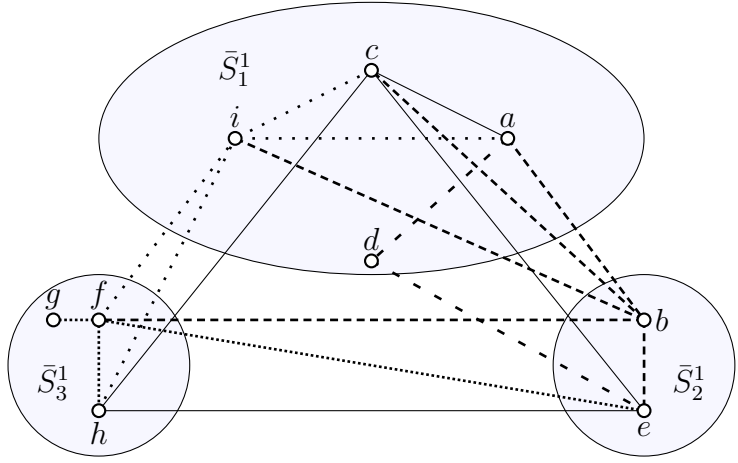

	\centering

	\caption{
		\textsc{Normalized}-$3$-\textsc{Cut}: Application of the crossover operator $C_{2,2}$ on the graph in Figure~\ref{fig::example}. It consists of four steps:
		(1) $V_1\leftarrow V_0\cup\{f\}$ by connecting $f$ to
		$e,g,h$ in densely dotted lines and $\labeling(f,\bar{I}^1)=3$;
		(2) $V_2\leftarrow V_1\cup\{i\}$ by connecting $i$ to $a,c,f,h$ in loosely dotted lines and $\labeling(i,\bar{I}^1)=1$;
		(3) $V_3\leftarrow V_2\cup\{d\}$ by connecting $d$ to $a,e$ in loosely dashed lines and $\labeling(d,\bar{I}^1)=1$;
		(4) $V=V_4\leftarrow V_3\cup\{b\}$ by connecting $b$ to $a,c,e,i$ in densely dashed lines and $\labeling(b,\bar{I}^1)=2$. Finally, we have $\bar{I}^1=\{\bar{S}_1^1, \bar{S}_2^1, \bar{S}_3^1\}$.
	}
	\label{fig::c3}
\end{figure} 

\subsubsection{Crossover operators $C_{3,1}$ and $C_{3,2}$}
\label{sec::c4c5}

The operators $C_{3,1}$ and $C_{3,2}$, extending $CB_3$ from \cite{Marti2009}, create $\bar{I}^1$ and $\bar{I}^2$ by identifying the midpoints of the mutual paths linking $I^1$ and $I^2$. Prior to detailing the process, let 
\begin{equation}
	\label{delta::obj}
	\Delta_M=F(I\circ M)-F(I)
\end{equation}
represent the change in $F$ resulting from a move $M$ on the given solution $I$ (denoted by $I\circ M$). Let the $j$-th order neighborhood of $I$ be
\begin{equation}
	\mathbb{E}^{(j)}(I) = 
	\{J:\,E(J)\in\mathbb{E},\, |\{v\in V:\,\labeling(v,J)=\labeling(v, I)\}|=n-j\},
\end{equation}
i.e., the union of all solutions that differ from $I$ by exactly $j$ vertices not belonging to the same labeled cut segments.		
The simplest move involves transferring a vertex $v$ from the current subset $S_p$ to another subset $S_q$, denoted as $S_p\rightarrow v\rightarrow S_q$, or $v\rightarrow S_q$, or $(v,q)$. All simplest moves causing changes in $F$ form the so-called ``natural move-gain" matrix for $F$:
\begin{equation}
	\label{delta::dobj}
	\mathcal{D}=(\Delta_{v\rightarrow S_q})_{v\in V,\,q\in[k]}\in\mathbb{Z}^{n\times k},
\end{equation}
which records the total possible changes in $F$ within the first order neighborhood of $I$.
For the vertices in {${V}_0^c=V\backslash(\cup_{p=1}^k(S_p^1\cap S_p^2))$}, executing the simplest move $|{V}_0^c|$ times can realize the transition $I^1\rightarrow I^2$ or $I^2\rightarrow I^1$, and the offspring solutions $\bar{I}^1$ and $\bar{I}^2$ are thus created in the halfway. This indicates that generating $\bar{I}^1$ or $\bar{I}^2$ requires $|{V}_0^c|/2$ times of simplest moves, each of which $v^*\rightarrow S_{q^*}$ at the $t$-th step is determined by
\begin{equation}
	\label{path::cb3-1}
	(v^*,\,q^*)\in\argopt_{v\in {V}_t^c,\, q\in\mathcal{L}_3(v)}\left\{ \Delta_{v\rightarrow S_q} \right\},\quad {V}_{t+1}\leftarrow {V}_t\cup \{v^*\},
\end{equation}
where
$$\mathcal{L}_3(v)=
\left\{
\begin{aligned}
	& \{\labeling(v,I^2)\}, & \text{when generating } \bar{I}^1 \text{ by }  C_{3,1},\\
	& \{\labeling(v,I^1)\}, & \text{when generating } \bar{I}^2 \text{ by }  C_{3,1},\\
	& [k]\backslash\{\labeling(v,I^1)\},& \text{when generating } \bar{I}^1 \text{ by }  C_{3,2},\\
	&[k]\backslash\{\labeling(v,I^2)\},& \text{when generating } \bar{I}^2 \text{ by }  C_{3,2}.
\end{aligned}	
\right.
$$
Both $C_{3,1}$ and $C_{3,2}$ aim to improve the quality of the offspring solutions. 

For 	\textsc{Normalized}-$3$-\textsc{Cut} on the graph in Figure~\ref{fig::example}, we have $|{V}_0^c|/2=2$, and the specific moves to create $\bar{I}^1$ an $\bar{I}^2$ by $C_{3,1}$ and $C_{3,2}$ are
$$
\begin{aligned}
	&C_{3,1}:&\bar{I}^1&=I^1\circ (f\rightarrow S_3^1) \circ (d\rightarrow S_2^1),\\
	&C_{3,1}:&\bar{I}^2&=I^2\circ (i\rightarrow S_3^2) \circ (b\rightarrow S_1^2),\\
	&C_{3,2}:&\bar{I}^1&=I^1\circ (f\rightarrow S_3^1) \circ (d\rightarrow S_2^1),\\
	&C_{3,2}:&\bar{I}^2&=I^2\circ (i\rightarrow S_1^2) \circ (b\rightarrow S_1^2).
\end{aligned}
$$
It can be readily verified that each offspring solution $\bar{I}\in\{\bar{I}^1,\bar{I}^2\}$ satisfies $F(\bar{I})<\min\{F(I^1),F(I^2)\}$ in this example.

\subsection{Mutation}
\label{sec::mutation}

For any $\opt\,F$ in MinGCP and MaxGCP, we design a multiple mutation heuristic (MMH)  (see Algorithm~\ref{algorithm::combinatorial}) to enhance the solution quality after the crossover phase,
inspired by the multiple operator heuristic (MOH) for \textsc{Max}-$k$-\textsc{Cut} \eqref{eq::max-k-cut} \cite{Ma2015}. 
MMH incorporates four local search (see $\widetilde{O}_1$, $\widehat{O}_1$, $\widetilde{O}_2$, $\widehat{O}_2$),   
four tabu search (see $\widetilde{O}_3$, $\widetilde{O}_4$, $\widehat{O}_3$, $\widehat{O}_4$) and one strong random perturbation (see $O_5$) operators, whereas MOH employs five search operators. 
A basic search operator in MOH is the single-transfer $O_1$, 
a best-of-all Variable Neighborhood Search (VNS) operator to seek the simplest move that yields the largest change (see Eq.~\eqref{delta::obj}) in the cut value function (see Eq.~\eqref{eq::cut-value}) \cite{mladenovic1997variable,hansen2002developments}.
To be more specific, 
after inputting a solution $I=\{S_1,\ldots,S_k\}$, 
$O_1$ examines the natural move-gain matrix $\mathcal{D}$ \eqref{delta::dobj} of ``$F=\frac{1}{2}\sum_{p=1}^k\cut(\partial S_p)$" to make the best simplest move $v\rightarrow S_q$ to form $I^\prime= I\circ (v\rightarrow S_q)$ that maximizes $F(I^\prime)$ within the first order neighborhood $I^\prime\in\mathbb{E}^{(1)}(I)$, and subsequently updates $\mathcal{D}$ for the next move. The majority of elements in $\mathcal{D}$ remain unchanged and only those related to the vertices in $\{v\}\cup\mathcal{N}(v)$ need to be recalculated after $O_1$ with $\mathcal{N}(v)=\{u\in V:\,\{u,v\}\in E\}$ being the neighborhood of $v$. Consequently, the number of elements requiring updates is $\mathcal{O}(kd_{\max})$, where $d_{\max}$ denotes the highest degree of the graph, enabling $O_1$ to be a rapid search operator. Nevertheless, the best-of-all search rule may not be effective for other $k$-CUT problems, especially for those with objective functions involving $\vol(\vec S)$ or $|\vec S|$.  Taking \textsc{Normalized}-$k$-\textsc{Cut} \eqref{eq::normalized-k-cut} as an example, the number of elements in $\mathcal{D}$ that necessitates updating after $O_1$ is $\mathcal{O}(k|V|)$, making the computation time-consuming \cite{hansen2012vns}. To strike a balance between the solution quality and computational cost, 
several reduced VNS methods~\cite{hansen2001variable,hansen2002developments,hansen2012vns,Ma2015} have been implemented, such as substituting the best-of-all rule with the first-of-all rule or utilizing random search strategies.
In a similar way, MMH replaces the best-of-all rule with a ``best-of-part" strategy to form $\widetilde{O}_1$ and $\widehat{O}_1$ as detailed in Section~\ref{sec::o1o4}. Each time applying $\widetilde{O}_1$ and $\widehat{O}_1$ involves updating the move-gain matrices $\widetilde{\mathcal{D}}$ and $\widehat{\mathcal{D}}$, respectively, and the number of elements requiring updates are both $\mathcal{O}(kd_{\max})$. Here are the definitions of $\widetilde{\mathcal{D}}$ and $\widehat{\mathcal{D}}$.

\textbf{The cut stable move-gain matrix $\tilde{\mathcal{D}}$} ~
Given any solution $I=\{S_1,S_2,\ldots,S_k\}$, let the change in the cut value function \eqref{eq::cut-value} be
\begin{equation}
	\widetilde{\Delta}_M=\cut(I\circ M)-\cut(I),
\end{equation} 
when applying a move $M$, and then define 
\begin{equation}
	\label{delta::dcut}
	\widetilde{\mathcal{D}}=(\widetilde{\Delta}_{v\rightarrow S_q})_{v\in V,\,q\in [k]}\in\mathbb{Z}^{n\times k}.
\end{equation} 
A direct calculations shows $\widetilde{\mathcal{D}} = (\widetilde{\mathcal{D}}_{vq})_{n\times k}$
with 
\begin{equation}
	\label{delta::value_cut}
	\widetilde{\mathcal{D}}_{vq}=\sum_{\{u,v\}\in E,\, u\in S_p}w_{uv}-\sum_{\{u,v\}\in E,\, u\in S_q}w_{uv}.
\end{equation} 
We call matrix $\tilde{\mathcal{D}}$ \emph{cut stable} after noting that $\widetilde{\mathcal{D}}_{vq}$, the change in cut value, depends only on the vertices in $\mathcal{N}(v)$ when moving the vertex $v$ from $S_p$ to $S_q$. 

%	,  alias $\widetilde{\Delta}_{v\rightarrow S_q}$,
%	. The boundary stable move-gain matrix is defined as

\textbf{The boundary stable move-gain matrix $\widehat{\mathcal{D}}$}~
Given any solution $I=\{S_1,S_2,\ldots,S_k\}$, let the change in the boundary size function $\cut(\partial S_q)$ be
\begin{equation}
	\widehat{\Delta}_M^{\{q\}}=\cut(\partial S_q^\prime)-\cut(\partial S_q),\quad I\circ M=\{S_1^\prime,\ldots,S_k^\prime\},\quad\forall q\in [k],
\end{equation}
when applying a move $M$, and then define 
\begin{equation}
	\label{delta::dpartial}
	\widehat{\mathcal{D}}=(\widehat{\Delta}_{v}^{\{q\}})_{v\in V,\,q\in [k]}\in\mathbb{Z}^{n\times k},\quad 	\widehat{\Delta}_v^{\{q\}}=\left\{
	\begin{aligned}
		&\widehat{\Delta}_{v\rightarrow S_q}^{\{q\}}, &v\notin S_q,\\
		&\widehat{\Delta}_{S_q\rightarrow v}^{\{q\}}, &v\in S_q.
	\end{aligned}
	\right.
\end{equation} 
A direct calculations shows $\widehat{\mathcal{D}} = (\widehat{\mathcal{D}}_{vq})_{n\times k}$
with 
\begin{equation}
	\label{delta::value_partial}
	\widehat{\mathcal{D}}_{vq}=\left\{
	\begin{aligned}
		&d_v-2\sum_{\{u,v\}\in E,\, u\in S_q}w_{uv}, &v\notin S_q,\\
		&2\sum_{\{u,v\}\in E,\, u\in S_q}w_{uv}-d_v, &v\in S_q.
	\end{aligned}
	\right.
\end{equation} 
We call matrix $\widehat{\mathcal{D}}$ \emph{boundary stable} after noting that $\widehat{\mathcal{D}}_{vq}$, the change in boundary size,  depends only on the vertices in $\mathcal{N}(v)$ when moving the vertex $v$ from $S_p$ to $S_q$. It should be pointed out that $\widehat{\mathcal{D}}_{vq}$ (resp.~$\widetilde{\mathcal{D}}_{vq}$) is another notation of $\widehat{\Delta}_{v}^{\{q\}}$ (resp.~$\widetilde{\Delta}_{v\rightarrow S_q}$) and we prefer to use the former for brevity and to the latter when emphasizing the move.

%\blue{It is readily derived from Eqs.~\eqref{delta::value_cut} and \eqref{delta::value_partial} that when a vertex is transferred from its  current subset to another cut segment, most elements remain unchanged except for those pertaining to the neighbors of $v$.}

%In addition, $\gamma_0$ and $1-\gamma_0$ are the possibilities of applying the packaged groups ($\widetilde{O}_1$, $\widetilde{O}_2$) and ($\widehat{O}_1$, $\widehat{O}_2$), respectively.

%with the possibilities of $\gamma_1$, $\gamma_2-\gamma_1$, $\gamma_3-\gamma_2$ and $1-\gamma_3$, respectively, 

The double-transfer local search operators $\widetilde{O}_2$ and $\widehat{O}_2$ moves an edge, i.e., two simplest moves, to improve $F$ and are triggered when $\widetilde{O}_1$ and $\widehat{O}_1$ exhaust, respectively, (see Algorithm~\ref{algorithm::combinatorial}, lines 4-18 and Section~\ref{sec::o1o4}). The computational cost of these four local search operators are lower than those of best-of-all rules (see Proposition~\ref{thm::o1comp}) with the benefit of bucket sorting.  When they fail to enhance the objective function value, 
MMH switches to the tabu search operators $\widetilde{O}_3$, $\widetilde{O}_4$, $\widehat{O}_3$, $\widehat{O}_4$ (see Algorithm \ref{algorithm::combinatorial}, lines 26-41 and Section \ref{sec::tabu}). Upon reaching the maximum perturbation  or surpassing the objective function value, MMH reverts to the local search phase. If the consecutive $\xi$ rounds of local-tabu search fail to improve the solution quality, a strong random perturbation operator $O_5$ is applied (see Algorithm \ref{algorithm::combinatorial}, line 43 and Section \ref{sec::perturb}). The termination condition for mutation can be either the time limit or the maximum number of iterations.

Following MMH, we propose a category search operator to further enhance the solution quality 
via applying MMH into an auxiliary graph $k$-CUT problem of the same category and the resulting algorithm is called an auxiliary cut mutation heuristic (ACMH). It is motivated by some early attempts using \textsc{Max}-$2$-\textsc{Cut} to generate high-quality solutions for \textsc{AntiCheeger}-$2$-\textsc{Cut} and  \textsc{Max}-\textsc{Bisection} \cite{shao2021continuous, burer2002rank}. 
In some sense, ACMH aims to explore high-quality solutions in the neighborhood of the auxiliary cut approximations.

\begin{breakablealgorithm}
	\caption{A multiple mutation heuristic (MMH) on graph $G = (V, E)$.}
	\label{algorithm::combinatorial}
	\begin{algorithmic}[1] \small  
		\Require 
		\begin{align*}
			\opt\, F(\cdot): & \text{~target problem}; \\
			I=\{S_1,\ldots,S_k\}: & \text{~initial solution}; \\
			\omega\in\mathbb{N}: & \text{~maximum number of tabu search moves}; \\
			\xi\in\mathbb{N}: & \text{~maximum number of consecutive non-improvement} \\
			& \text{~local-tabu search rounds}; \\
			\eta\in\mathbb{N}: & \text{~perturbation strength}; \\
			\gamma_0, \gamma_1, \gamma_2, \gamma_3\in(0,1): & \text{~probabilities of applying $(\widetilde{O}_1,\widetilde{O}_2)$, $(\widehat{O}_1,\widehat{O}_2)$, $\widetilde{O}_3$, $\widetilde{O}_4$, $\widehat{O}_3$, $\widehat{O}_4$} \\
			& \text{~are $\gamma_0$, $1-\gamma_0$, $\gamma_1$, $\gamma_2-\gamma_1$, $\gamma_3-\gamma_2$, $1-\gamma_3$, respectively.}
		\end{align*}
		\State $c_{\text{non-improved}} \gets 0$.
		$F_{\text{best}} \gets F(I)$.
		$I_{\text{best}} \gets I$.
		\While {not satisfying stopping criteria}
		\State Random selection:  $r_1\sim\mathcal{U}[0,1]$.\Comment{ $\mathcal{U}[0,1]$ represents the uniform distribution on $[0,1]$}
		\If {$r_1\leq \gamma_0$}
		\While {$I$ can be improved by $\widetilde{O}_1$, $\widetilde{O}_2$}
		\While {$I$ can be improved by $\widetilde{O}_1$ }
		\State $I\leftarrow I\circ \widetilde{O}_1$. Update $\widetilde{\mathcal{D}}$, $\widehat{\mathcal{D}}$. \Comment{Local search operator $\widetilde{O}_1$}
		\EndWhile
		\State $I\leftarrow I\circ \widetilde{O}_2$. Update $\widetilde{\mathcal{D}}$, $\widehat{\mathcal{D}}$.  \Comment{Local search operator $\widetilde{O}_2$}
		\EndWhile
		\Else
		\While {$I$ can be improved by $\widehat{O}_1$, $\widehat{O}_2$}
		\While {$I$ can be improved by $\widehat{O}_1$}
		\State $I\leftarrow I\circ \widehat{O}_1$. Update $\widetilde{\mathcal{D}}$, $\widehat{\mathcal{D}}$.  \Comment{Local search operator $\widehat{O}_1$}
		\EndWhile
		\State $I\leftarrow I\circ \widehat{O}_2$. Update $\widetilde{\mathcal{D}}$, $\widehat{\mathcal{D}}$.  \Comment{Local search operator $\widehat{O}_2$}
		\EndWhile
		\EndIf
		\State $F_{\text{local}} \gets F(I)$.
		\If { $F_{\text{local}}$ is better than $ F_{\text{best}}$}
		\State $F_{\text{best}} \gets F_{\text{local}}$.
		$I_{\text{best}}  \gets I$.
		$c_{\text{non-improved}} \gets 0$.
		\Else
		\State $c_{\text{non-improved}} \gets c_{\text{non-improved}} + 1$.
		\EndIf
		\State $c_{\text{tabu}} \gets 0$ and reset tabu list $L$.
		\While {$c_{\text{tabu}} \leq \omega$ and $F(I)$ is not better than $F_{\text{local}}$ } \Comment{Tabu search phase}
		\State Random selection: $\rho \sim \mathcal{U}[0,1]$.
		\If {$0\leq\rho <\gamma_1$}
		\State $I\leftarrow I\circ \widetilde{O}_3$. Update $\widetilde{\mathcal{D}}$, $\widehat{\mathcal{D}}$. Update $L$. \Comment{Tabu search operator $\widetilde{O}_3$}
		\EndIf
		\If {$\gamma_1\leq\rho <\gamma_2$}
		\State $I\leftarrow I\circ \widetilde{O}_4$. Update $\widetilde{\mathcal{D}}$, $\widehat{\mathcal{D}}$. Update $L$.
		\Comment{Tabu search operator $\widetilde{O}_4$}
		\EndIf
		\If {$\gamma_2\leq\rho <\gamma_3$}
		\State $I\leftarrow I\circ \widehat{O}_3$. Update $\widetilde{\mathcal{D}}$, $\widehat{\mathcal{D}}$. Update $L$.
		\Comment{Tabu search operator $\widehat{O}_3$}
		\EndIf
		\If {$\gamma_3\leq\rho \leq 1$}
		\State $I\leftarrow I\circ \widehat{O}_4$. Update $\widetilde{\mathcal{D}}$, $\widehat{\mathcal{D}}$. Update $L$.
		\Comment{Tabu search operator $\widehat{O}_4$}
		\EndIf
		\State $c_{\text{tabu}} = c_{\text{tabu}} + 1$.
		\EndWhile
		\If {$c_{\text{non-improved}} > \xi$}\Comment{Strong random perturbation phase}
		\State $I\leftarrow I\circ O_5$. \Comment{Random perturbation operator $O_5$ for $\eta$ times }
		\State Update $\widetilde{\mathcal{D}}$, $\widehat{\mathcal{D}}$.
		\State $c_{\text{non-improved}} \gets 0$. 
		\EndIf
		\EndWhile
		\State \Return{$I_{\text{best}}$}.
	\end{algorithmic}  
\end{breakablealgorithm}

			\subsubsection{Local search operators $\widetilde{O}_1$, $\widehat{O}_1$, $\widetilde{O}_2$, $\widehat{O}_2$}
			\label{sec::o1o4}

			The best-of-part rule utilizes the move-gain matrix to seek improved solutions within limited yet promising search areas, as opposed to the best-of-all rule which exhaustively explores all possible solutions in the neighborhoods. The proposed four local search operators fully follow the best-of-part rule. For simplicity, we denote the move-gain matrix
			$\mathcal{D} = (\mathcal{D}_{vq})_{n\times k}$ with 			
			\begin{equation}
				\mathcal{D}_{vq}=\left\{
				\begin{aligned}
					& \widetilde{\mathcal{D}}_{vq},&\text{when applying }\widetilde{O}_1\text{ or }\widetilde{O}_2,\\
					& \widehat{\mathcal{D}}_{vq},&\text{when applying }\widehat{O}_1 \text{ or }\widehat{O}_2.
				\end{aligned}
				\right.
			\end{equation}

			%										And for each double-transfer local search operator $\widetilde{O}_2$ or $\widehat{O}_2$, the combined two simplest moves $v^*\rightarrow S_{q^*}$ and $u^*\rightarrow S_{p^*}$ are selected from Eq.~\eqref{choice::o2cut} with the two vertices being endpoints of some edge $\{u^*,v^*\}\in E$.

			\textbf{The single-transfer local search operator $\widetilde{O}_1$ or $\widehat{O}_1$}  moves the vertex ${v}^*$ to $S_{{q}^*}$ that satisfies
			%will select the vertex $\widetilde{v}^*$ and transfer it into the subset $S_{\widetilde{q}^*}$ using the following procedure:
			\begin{align}
				\label{choice::o1cut}
				({v}^*,\,{q}^*)
				&\in 
				\argopt_{(v,q)\in \mathbb{D}(I)}\left\{F(I\circ (v\rightarrow S_q))\right\},
				%					\left\{
				%					\begin{aligned}
					%						&\argopt_{(v,q)\in \mathbb{D}(I)}\left\{F(I\circ (v\rightarrow S_q))\right\}, &
					%						\text{if }\opt\, F\in\text{MaxGCP},\\
					%						&\argopt_{(v,q)\in\mathbb{D}(I)}\left\{F(I\circ (v\rightarrow S_q))\right\}, &
					%						\text{if }\opt\, F\in\text{MinGCP},
					%					\end{aligned}
				%					\right. 
				\\
				\label{eq::d_vq}
				\mathbb{D}(I)
				&=\left\{
				\begin{aligned}
					\bigcup\limits_{t\in [k]\backslash\{\labeling(v, I)\}}\argmax_{v\in V,q=t}\{\mathcal{D}_{vq}\}, & \text{~if~}\opt\, F\in\text{MaxGCP},\\
					\bigcup\limits_{t\in [k]\backslash\{\labeling(v, I)\}}\argmin_{v\in V,q=t}\{\mathcal{D}_{vq}\},  & \text{~if~}\opt\, F\in\text{MinGCP}.
				\end{aligned}
				\right.
			\end{align}
			It can be readily verified that the search area $\mathbb{D}(I)$ contains the simplest moves leading to the largest improvement in the cut value or boundary size, thereby indicating that $\mathbb{D}(I)$ keeps pace with the basic needs of $\opt\,F$ and is thus the promising search area of improved solutions for $F$. In a word, $\widetilde{O}_1$ or $\widehat{O}_1$ selects the best simplest move within $\mathbb{D}(I)$, a subset of all possible moves in the first order neighborhood of the current solution $I$.

			%\item \textbf{The $\widehat{O}_1$ search operator} \blue{moves the vertex $\widehat{v}^*$ to $ S_{\widehat{q}^*}$ that satisfies}
			%%will select the vertex $\widehat{v}^*$ and transfer it into the subset $ S_{\widehat{q}^*}$ using the following procedure:
			%\begin{equation}
			%\label{choice::o1partial}
			%(\widehat{v}^*,\,\widehat{q}^*)
			%\in 
			%\left\{
			%\begin{aligned}
			%&\argopt_{(v,q)\in\cup_{q\in [k]}\argmax_v\{\widehat{\mathcal{D}}_{vq}\}}\left\{f_C(I\circ v\rightarrow S_q)\right\}, &
			%\text{if }\opt\, f_C\in\text{MaxGCP},\\
			%&\argopt_{(v,q)\in\cup_{q\in [k]}\argmin_v\{\widehat{\mathcal{D}}_{vq}\}}\left\{f_C(I\circ v\rightarrow S_q)\right\}, &
			%\text{if }\opt\, f_C\in\text{MinGCP}.
			%\end{aligned}
			%\right.
			%\end{equation}

			\textbf{The double-transfer local search operator $\widetilde{O}_2$ or $\widehat{O}_2$} moves the vertices ${v}^*$ and ${u}^*$ into $ S_{{q}^*}$ and $ S_{{p}^*}$, respectively, as follows				
			\begin{equation}
				\label{choice::o2cut}
				({v}^*,\,{q}^*), ({u}^*,\,{p}^*) \in
				\argopt_{\begin{tiny}
						\begin{matrix}
							(v,q)\in\mathbb{D}(I),u\in\mathcal{N}(v),\\
							p\in[k]\backslash\{\labeling(u,I)\}
						\end{matrix}
				\end{tiny}}\left\{F(I\circ (v\rightarrow S_q)\circ (u\rightarrow S_p))\right\}.
				%					 
				%					\left\{
				%					\begin{aligned}
					%						&, &
					%						\text{if }\opt\, F\in\text{MaxGCP},\\
					%						&\argopt_{\begin{tiny}
							%								\begin{matrix}
								%									(v,q)\in\mathbb{D}(I),u\in\mathcal{N}(v),\\
								%									p\in[k]\backslash\{\labeling(u,I)\}
								%								\end{matrix}
							%						\end{tiny}}\left\{F(I\circ (v\rightarrow S_q)\circ (u\rightarrow S_p))\right\}, &
					%						\text{if }\opt\, F\in\text{MinGCP}.
					%					\end{aligned}
				%					\right.
			\end{equation}
			It can be easily observed above that both $\widetilde{O}_2$ and $\widehat{O}_2$ typically target two vertices that are endpoints of an edge. Considering a double-transfer move, $S_{p_u}\rightarrow u\rightarrow S_{q_u}$ and $S_{p_v}\rightarrow v \rightarrow S_{q_v}$, the changes in the cut value and boundary size are respectively represented by
			\begin{align}
				\widetilde{\Delta}_{(u,v)}&=\widetilde{\Delta}_{u\rightarrow S_{q_u}}+\widetilde{\Delta}_{v\rightarrow S_{q_v}}+\widetilde{\phi}w_{uv},\quad \text{when applying } \widetilde{O}_2,\label{delta::o2cut}\\
				\widehat{\Delta}_{(u,v)}^{\{i\}}&=\widehat{\Delta}_{{u\rightarrow S_{q_u}}}^{\{i\}}+\widehat{\Delta}_{{v\rightarrow S_{q_v}}}^{\{i\}}+\widehat{\phi}^{\{i\}}w_{uv},\,i\in\{p_u,\,q_u,\,p_v,\,q_v\}, \quad \text{when applying } \widehat{O}_2,	\label{delta::o2partial}
			\end{align}
			where the coefficients $\widetilde{\phi}$ and $\widehat{\phi}^{\{i\}}$ are classified into seven cases as detailed in Table~\ref{tab::d2phi}. It is evident that $\widetilde{O}_2$ (resp.~$\widehat{O}_2$) on vertices $u,v$ deviates from two independent single-transfer moves iff $\widetilde{\phi}w_{uv}\neq 0$ (resp.~$\exists\, i\in\{p_u,\,q_u,\,p_v,\,q_v\}$ such that $\widehat{\phi}^{\{i\}}w_{uv}\neq 0$).
			
			\begin{table}[htbp]
				\centering
				\caption[Coefficients of double-transfers]{Coefficients $\widetilde{\phi}$ and $\widehat{\phi}^{\{i\}}$ required respectively by Eqs.~\eqref{delta::o2cut} and \eqref{delta::o2partial}.} 
				\label{tab::d2phi}
				\begin{tabular}{|l|ccccc|}
					\hline
					&\multirow{2}*{\makecell{ $\widetilde{\phi}$}} & \multirow{2}*{\makecell{$\widehat{\phi}^{\{p_u\}}$}} &\multirow{2}*{\makecell{$\widehat{\phi}^{\{p_v\}}$}} &\multirow{2}*{\makecell{$\widehat{\phi}^{\{q_u\}}$}} &\multirow{2}*{\makecell{$\widehat{\phi}^{\{q_v\}}$}} \\
					&&&&&\\ \hline
					${p_u}={p_v}$,  ${q_u}={q_v}$ & -2                 & -2                         & -2                         & -2                         & -2                         \\ \hline
					${q_u}={p_v}$,  ${p_u}={q_v}$ & 2                  & 2                          & 2                          & 2                          & 2                          \\ \hline
					${p_u}={p_v}$,  ${q_u}\neq {q_v}$ & -1                 & -2                         & -2                         & 0                          & 0                          \\ \hline
					${p_u}={q_v}$,  ${q_u}\neq {p_v}$ & 1                  & 2                          & 0                          & 0                          & 2                          \\ \hline
					${p_u}\neq {p_v}$,  ${q_u}={q_v}$ & -1                 & 0                          & 0                          & -2                         & -2                         \\ \hline
					${p_u}\neq {q_v}$,  ${q_u}={p_v}$ & 1                  & 0                          & 2                          & 2                          & 0                          \\ \hline
					\multirow{2}*{\makecell{${p_u}\neq {p_v}$,  ${q_u}\neq {p_v}$\\${p_u}\neq {q_v}$,  ${q_u}\neq {q_v}$}} 
					&    \multirow{2}*{\makecell{0}}          &  \multirow{2}*{\makecell{0}}  &          \multirow{2}*{\makecell{0}}                &     \multirow{2}*{\makecell{0}}                        & \multirow{2}*{\makecell{0}}                \\ &&&&&\\
					\hline
				\end{tabular}
			\end{table}

			%when excluding the final case in Table~\ref{tab::d2phi} and $\{u,v\}\in E$. 
			
			Two comments follow accordingly. 
			\begin{itemize}
				\item  Utilizing the bucket sorting schemes \cite{Ma2015} to store the stable move-gain matrices, the search area $\mathbb{D}(I)$, whose size $|\mathbb{D}(I)|$ is notably smaller than $k|V|$ in general cases, is easily obtained as detailed in Appendix~\ref{app:proof}. The update of each bucket sorting scheme when acting a simplest move is efficient with the computational cost being $\mathcal{O}(kd_{\max})$ instead of $\mathcal{O}(k|V|)$. The cost analysis of $\widetilde{O}_1$, $\widehat{O}_1$, $\widetilde{O}_2$ and $\widehat{O}_2$ is delineated in Proposition~\ref{thm::o1comp}.
				\item For \textsc{Max}-$k$-\textsc{Cut} and \textsc{Min}-$k$-\textsc{Cut}, Proposition~\ref{thm::maxkcut}
				demonstrates that the best-of-part search operators $\widetilde{O}_1$ and $\widetilde{O}_2$ are nothing but the best-of-all ones in the first and second order neighborhoods, respectively. 
			\end{itemize}

			%Note that the original double-transfer search operator $O_2$ in~\cite{Ma2015} does not guarantee the identification of local optima within the second order neighborhood, as it executes random moves on a portion of randomly selected vertices. 

			%\textbf{The $\widehat{O}_2$ search operator} \blue{moves the vertices $\widehat{v}^*$ and $\widehat{u}^*$ respectively into $ S_{\widehat{q}^*}$ and $ S_{\widehat{p}^*}$ by}
			%\begin{equation}
			%	\label{choice::o2partial}
			%	\begin{matrix}
				%		(\widehat{v}^*,\,\widehat{q}^*)\\
				%		(\widehat{u}^*,\,\widehat{p}^*)
				%	\end{matrix}
			%	\in 
			%	\left\{
			%	\begin{aligned}
				%		&\argopt_{\begin{tiny}
						%				\begin{matrix}
							%					(v,q)\in\cup_{q\in [k]}\argmax_v\{\widehat{\mathcal{D}}_{vq}\}\\
							%					u\in\mathcal{N}(v),p\in[k]\backslash\{\labeling(u,I)\}
							%				\end{matrix}
						%		\end{tiny}}\left\{f_C(I\circ v\rightarrow S_q\circ u\rightarrow S_p)\right\}, &
				%		\text{if }\opt\, f_C\in\text{MaxGCP},\\
				%		&\argopt_{\begin{tiny}
						%				\begin{matrix}
							%					(v,q)\in\cup_{q\in [k]}\argmin_v\{\widehat{\mathcal{D}}_{vq}\}\\
							%					u\in\mathcal{N}(v),p\in[k]\backslash\{\labeling(u,I)\}
							%				\end{matrix}
						%		\end{tiny}}\left\{f_C(I\circ v\rightarrow S_q\circ u\rightarrow S_p)\right\}, &
				%		\text{if }\opt\, f_C\in\text{MinGCP}.
				%	\end{aligned}
			%	\right.
			%\end{equation}

			\begin{prop}\label{thm::o1comp}
				The computational costs of the local search operators are bounded as follows:
				\begin{enumerate}
					\item For $\widetilde{O}_1$ and $\widehat{O}_1$, the complexity is at most $\mathcal{O}(\max\{k |\mathbb{D}(I)|,\,k d_{\max}\})$.
					\item For $\widetilde{O}_2$ and $\widehat{O}_2$, the complexity is at most $\mathcal{O}(k d_{\max}|\mathbb{D}(I)|)$.
				\end{enumerate}
%				For the four local search operators $\widetilde{O}_1$, $\widehat{O}_1$, $\widetilde{O}_2$ and $\widehat{O}_2$, we have the following cost estimation: (1) the cost of $\widetilde{O}_1$ or $\widehat{O}_1$ is no more than $\mathcal{O}(k|\mathbb{D}(I)|+kd_{\max})$; (2) the cost of $\widetilde{O}_2$ or $\widehat{O}_2$ is no more than $\mathcal{O}(kd_{\max}+k|\mathbb{D}(I)|d_{\max})$. 			%Assuming that the cost of computing the objective function value $f_C(\cdot)$ is $F$, the computational complexity of performing a single-transfer move operation $\widetilde{O}_1$ or $\widehat{O}_1$ consists of the following two parts. 
				%\begin{enumerate}
				%\item  Let $l$ be the average length of the doubly-linked lists corresponding to the maximum or the minimum nodes for MaxGCP or MinGCP, respectively. The computational complexity of determining a single-transfer move is $O(Fk l)$.
				%\item The computational complexity of updating the two stable move-gain matrices, the bucket sorting arrays, and relabeling the minimum or the maximum nodes is $O(kd_{\max})$.
				%\end{enumerate}
			\end{prop}
			\begin{proof}
				Please see Appendix~\ref{app:proof}. 
			\end{proof}

			\begin{prop}
				\label{thm::maxkcut}
				Given any unweighted graph and ``$F=\frac{1}{2}\sum_{p=1}^k\cut(\partial S_p)$", the operators $\widetilde{O}_1$ and $\widetilde{O}_2$ are able to identify local optima in the first and second order neighborhoods, respectively.
			\end{prop}
			\begin{proof}
				\textsc{Min}-$k$-\textsc{Cut} shares the similar proof with \textsc{Max}-$k$-\textsc{Cut} that we assume ``$\opt=\max$" in the followings.
				If $\widetilde{O}_1$ can not improve the cut value for the current solution $I$, $\widetilde{D}_{vq}\leq 0$ holds for $\forall\,(v,q)\in \mathbb{D}(I)$ (see Eq.~\eqref{eq::d_vq}), thus $\widetilde{D}_{vq}\leq 0$ holds for $\forall\,v\in V$ and $\forall\,q\in[k]$, leading to $I$ being the local optimum in the first order neighborhood. Next we prove that if there exists a double-transfer move $S_{p_u}\rightarrow u\rightarrow S_{q_u}$ and $S_{p_v}\rightarrow v \rightarrow S_{q_v}$ that can improve the cut value, i.e. $\widetilde{\Delta}_{(u,v)}>0$, $\widetilde{O}_2$ will promote the cut value. Since the current solution $I$ is a local maximum in $\mathbb{E}^{(1)}(I)$, we have $\widetilde{\Delta}_{u\rightarrow S_{q_u}}\leq 0$ and $\widetilde{\Delta}_{v\rightarrow S_{q_v}}\leq 0$, indicating $\widetilde{\phi}w_{uv}=2$ to satisfy $\widetilde{\Delta}_{(u,v)}>0$ and at least one of $\widetilde{\Delta}_{u\rightarrow S_{q_u}}=0$ and $\widetilde{\Delta}_{v\rightarrow S_{q_v}}=0$ must be satisfied, implying that either $(v,q_v)$ or $(u,q_u)$ belongs to $\mathbb{D}(I)$. Therefore, this double-transfer move $S_{p_u}\rightarrow u\rightarrow S_{q_u}$ and $S_{p_v}\rightarrow v \rightarrow S_{q_v}$ is within the search area of $\widetilde{O}_2$.
			\end{proof}

			%			\begin{prop}[Double-transfer operators $\widetilde{O}_2$ and $\widehat{O}_2$]
				%				\label{thm::o2comp}
				%				\blue{The computational cost for applying a double-transfer move is no more than $\mathcal{O}(\max\{k,c|\mathbb{D}(I)|\}d_{\max})$.}
				%				%In practice, the double-transfer move operations can be considered as the combination of two single-transfer move operations, which allow them to satisfy the last property of Proposition \ref{thm::o1comp}. Regarding the first property, the cost of computing and selecting a double-transfer move is $O(Fkld_{\max})$.
				%			\end{prop} 

			\subsubsection{Tabu search operators $\widetilde{O}_3$, $\widetilde{O}_4$, $\widehat{O}_3$, $\widehat{O}_4$}
			\label{sec::tabu}
			
			When above four local search operators fail to enhance the quality of the current solution, we turn to the tabu search phase, inspired by the approach in~\cite{Ma2015}. This phase integrates the local search operators with a tabu list $L$, the $i$-th element in which represents the tabu tenure $L_i$ of the vertex $i$, implying that the vertex $i$ is prohibited from being moved in the next $L_i$ iterations.

			The single-transfer tabu search operators  $\widetilde{O}_3$ and $\widehat{O}_3$ are derived from $\widetilde{O}_1$ and $\widehat{O}_1$, respectively, while the double-transfer tabu search operators $\widetilde{O}_4$ and $\widehat{O}_4$ originate from $\widetilde{O}_2$ and $\widehat{O}_2$, respectively. Before detailing the implementation of the tabu search operators, we define the $j$-th search layer of the bucket sorting storage scheme for any solution $I$ as 
			\begin{equation}
				\label{choice::j-layer}
				\begin{aligned}
					&\mathbb{D}_j(I,L)=\\
					&\left\{
					\begin{aligned}
						&\bigcup_{q\in [k]}\{(v,q):\,L_v=0\text{ or }F(I\circ (v\rightarrow S_q))>F_{\text{local}},\, \mathcal{D}_{vq}=b_{\max}^q-j+1\}, &
						\text{if }\opt\, F\in\text{MaxGCP},\\
						&\bigcup_{q\in [k]}\{(v,q):\,L_v=0\text{ or }F(I\circ (v\rightarrow S_q))<F_{\text{local}},\,\mathcal{D}_{vq}=b_{\min}^q+j-1\}, &
						\text{if }\opt\, F\in\text{MinGCP},
					\end{aligned}
					\right.\\
				\end{aligned}
			\end{equation}
			where $\mathcal{D}\in\{\widetilde{\mathcal{D}},\,\widehat{\mathcal{D}}\}$ and its element is denoted by $\mathcal{D}_{vq}$, $b_{\max}^q=\max_{v\in V}\{\mathcal{D}_{vq}\}$ and $b_{\min}^q=\min_{v\in V}\{\mathcal{D}_{vq}\}$. Here $L_v=0$ indicates the vertex $v$ is not banned. By comparing $F(I\circ (v\rightarrow S_q))$ with $F_{\text{local}}$, we can promptly capture improved solutions irrespective of the tabu list, allowing MMH to exit the tabu search phase (see Algorithm~\ref{algorithm::combinatorial}). In the case of a all-zero tabu list $L$, $\mathbb{D}_1(I,L)=\mathbb{D}(I)$ (see Eq.~\eqref{eq::d_vq}) exactly represents the search area of the corresponding single-transfer local search operator.

			%When the four local search operators fail to enhance the quality of the current solution, it becomes imperative to overcome local optima and broaden the diversity of the search space. Inspired from the applications outlined in \cite{Ma2015}, we can combine the local search operators  $\widetilde{O}_1$, $\widehat{O}_1$, $\widetilde{O}_2$, $\widehat{O}_2$) with tabu search strategies to generate corresponding tabu search operators ($\widetilde{O}_3$, $\widehat{O}_3$, $\widetilde{O}_4$, $\widehat{O}_4$)(see Algorithm \ref{algorithm::combinatorial}, lines 28-39).
			
			%	\begin{equation}
				%		\label{choice::o3cut}
				%		({v}^*,\,{q}^*)
				%		\in 
				%		\left\{
				%		\begin{aligned}
					%			&\argopt_{(v,q)\in \mathbb{D}_j(I,L)}\left\{f_C(I\circ v\rightarrow S_q)\right\}, &
					%			\text{if }\opt\, f_C\in\text{MaxGCP},\\
					%			&\argopt_{(v,q)\in\mathbb{D}_j(I,L)}\left\{f_C(I\circ v\rightarrow S_q)\right\}, &
					%			\text{if }\opt\, f_C\in\text{MinGCP},
					%		\end{aligned}
				%		\right.
				%	\end{equation}

			\textbf{The single-transfer tabu search operator} $\widetilde{O}_3$ or $\widehat{O}_3$ determines the vertex $v^*$ and the moved in subset $S_{q^*}$ layer-by-layer in a first-of-all strategy. To begin with $j=1$, the move is randomly selected from $	({v}^*,\,{q}^*)\in  \mathbb{D}_j(I,L)$, and set $j\leftarrow j+1$ if $\mathbb{D}_j(I,L)=\emptyset$, otherwise stop the search.

			\textbf{The double-transfer tabu search operator} $\widetilde{O}_4$ or $\widehat{O}_4$  utilizes neighborhoods of $\mathbb{D}_j(I,L)$ to determine the moves ${v}^*\rightarrow S_{{q}^*}$ and ${u}^*\rightarrow S_{{p}^*}$ as follows,
			\begin{equation}
				\label{choice::o4cut}
				\begin{matrix}
					({v}^*,\,{q}^*)\\
					({u}^*,\,{p}^*)
				\end{matrix}
				\in 
				\left\{
				\begin{aligned}
					&\argopt_{\begin{tiny}
							\begin{matrix}
								(v,q)\in\mathbb{D}_j(I,L),u\in\mathcal{N}(v,L),\\
								p\in[k]\backslash\{\labeling(u,I)\}
							\end{matrix}
					\end{tiny}}\left\{F(I\circ (v\rightarrow S_q)\circ (u\rightarrow S_p))\right\}, &
					\text{if }\opt\, F\in\text{MaxGCP},\\
					&\argopt_{\begin{tiny}
							\begin{matrix}
								(v,q)\in\mathbb{D}_j(I,L),u\in\mathcal{N}(v,L),\\
								p\in[k]\backslash\{\labeling(u,I)\}
							\end{matrix}
					\end{tiny}}\left\{F(I\circ (v\rightarrow S_q)\circ (u\rightarrow S_p))\right\}, &
					\text{if }\opt\, F\in\text{MinGCP},
				\end{aligned}
				\right.
			\end{equation}
			with $\mathcal{N}(v,L)=\{u:\,L_u=0,\,\{u,v\}\in E\}$. Repeat this operation if $\mathbb{D}_j(I,L)=\emptyset$ or $\mathcal{N}(v,L)=\emptyset$, $\forall\,v\in\mathbb{D}_j(I,L)$ and we set $j\leftarrow j+1$.  %(?????)

			%until all vertices are examined. 

			%	Note that $\widetilde{O}_3$ and $\widehat{O}_3$ prioritize randomness and diversity more than $\widetilde{O}_4$ and $\widehat{O}_4$, while $\widetilde{O}_4$ and $\widehat{O}_4$ care solution improvement more than $\widetilde{O}_3$ and $\widehat{O}_3$, suggesting that the tabu search operators strike a balance between solution quality and diversity while maintaining the potential to examine all vertices. 
			
			Above four tabu search operators are designed to discover improved solutions in the vicinity of those generated by the local search operators. When these tabu search operators fail to provide enough diversity to continue the improvement, MMH switches to a strong perturbation operator.

			\subsubsection{Random perturbation operator $O_5$}
			\label{sec::perturb}

            {\textbf{The strong perturbation operator $O_5$} randomly selects $\eta$ vertices
and moves each selected vertex to a different segment, where $\eta$ is sampled
by $\randn(0.1|V|,0.3|V|)$ (Algorithm~\ref{algorithm::combinatorial},
lines~43--44). The resulting $\eta$ single-vertex moves impose a substantial
perturbation on the current partition, which enhances diversity and enables the
algorithm to explore distant regions of the search space. After $O_5$, MMH reverts to the local search phase.}
			
			% \textbf{The strong perturbation operator $O_5$} randomly selects  $\eta$ vertices and assigns them into other segments with $\eta$ determined by $\randn(0.1|V|,0.3|V|)$ (see Algorithm \ref{algorithm::combinatorial}, lines 43-44). These $\eta$ times of moving single vertices significantly increase the diversity and facilitate the exploration of improved solutions that are far from the current one. After $O_5$, MMH reverts to the local search phase.
			
			%\textbf{The strong perturbation operator $O_5$} is based on random single-transfer moves. Each move randomly selects a vertex $v\in V$ and assigns it to a randomly chosen moved-in subset$S_q$, $q\in [k]\backslash\{\labeling(v,I)\}$ for the solution $I$. $O_5$ is applied by executing random single-transfer moves for $\eta$ times, where $\eta$ is determined by $\randn(0.1|V|,0.3|V|)$ (see Algorithm \ref{algorithm::combinatorial}, lines 43-44). These random single-transfer moves significantly increase the diversity within the solution and make it more likely to be further improved. Once the execution of $O_5$ is completed, the mutation algorithm returns to the local search phase.

			\subsubsection{Category search: An auxiliary cut mutation heuristic (ACMH)}
			\label{sec::cut-combined}

            {ACMH combines the target graph cut problem with several auxiliary ones by
reusing MMH under different objective functions. Since MMH takes the objective
function as an input, the same search procedure can be applied to different
graph cut formulations without changing its main structure. ACMH first applies
MMH to the target objective, and then uses auxiliary objectives to guide the
search toward different regions of the solution space. During the auxiliary
search, each newly updated solution is still checked against the target
objective, so that any improvement for the target problem can be retained. The
detailed procedure is outlined in Algorithm~\ref{algorithm::acmh}.}

			% ACMH, combining the target and auxiliary graph cut problems, leverages the generality of MMH to simply implement a possible enhancement of solution quality, and its detailed procedure for solving $\opt\,F(\cdot)$ is outlined in Algorithm~\ref{algorithm::acmh}.

			\begin{breakablealgorithm}
				\caption{An auxiliary cut mutation heuristic (ACMH) on $G = (V, E)$.} 
				\label{algorithm::acmh}
				\begin{algorithmic}[1] \small
					\Require  
					\begin{align*}
						\opt\, F(\cdot): & \text{~target } k\text{-CUT problem}; \\
						\opt\,F_{1}(\cdot), \ldots, \opt\,F_{p}(\cdot): & ~p\text{ different auxiliary }k\text{-CUT problems};\\
						I^*=\{S_1,\ldots,S_k\}: & \text{~initial solution}; \\
						n_{\max}:& \text{~maximum number of search operations.}
					\end{align*}
					%\State $I^*\leftarrow I$.
					\While{not satsfying stopping criteria}
					\State Input $I^*$ and $\opt\,F(\cdot)$ to MMH in Algorithm \ref{algorithm::combinatorial} and then output $I^*$ when the number of search operations reach $n_{\max}$.
					\State $I\leftarrow I^*$.
					\State Randomly select $i\in[p]$.
					\State  Input $I$ and $\opt\,F_i(\cdot)$ to MMH in Algorithm \ref{algorithm::combinatorial}   with a monitor: Add the following lines \begin{align*}
						&\textbf{If } F(I)\text{ is better than }F(I^*)\textbf{ then}\\
						& \qquad  I^*\leftarrow I.\\
						&\textbf{end} \textbf{ if}
					\end{align*}
					after each update $I\leftarrow I\circ O$ in Algorithm \ref{algorithm::combinatorial} with $O\in\{\widetilde{O}_1, \widehat{O}_1, \widetilde{O}_2, \widehat{O}_2,$ $\widetilde{O}_3, \widehat{O}_3, \widetilde{O}_4, \widehat{O}_4, O_5\}$. Output $I^*$ when the number of search operations reach $n_{\max}$.
					\EndWhile
					\State \Return{$I^*$}.
				\end{algorithmic}  
			\end{breakablealgorithm}

			%			\blue{\begin{itemize}
					%					\item \textbf{Step 1:} Select $r$ different auxiliary $k$-CUT problems that are highly related to the target problem: $\opt\,F_{1}(\cdot)$, $\ldots$, $\opt\,F_{r}(\cdot)$.
					%					\item \textbf{Step 2:} Run Algorithm \ref{algorithm::combinatorial} towards $\opt\,F(\cdot)$ for $n_{\max}$ iterations.
					%					\item \textbf{Step 3:} Randomly select $i\in[r]$, and run Algorithm \ref{algorithm::combinatorial} towards $\opt\,F(\cdot)$. During each iteration, whenever the solution is updated by some search operator, check whether the objective function value $F$ is promoted. After completing $n_{\max}$ iterations, return the best solution for $\opt\,F$.
					%					\item \textbf{Step 4:} Go to \textbf{step 2} until the overall modified mutation phase meets the stopping condition.
					%			\end{itemize}}
			
			Note that the auxiliary $k$-CUT problems should belong to the same category (MaxGCP or MinGCP) as the target, although they may differ in balance. Numerical experiments in Section \ref{sec::experiment} shows that
			ACMH achieves superior solution quality compared to MMH for e.g., \textsc{Normalized}-$2$-\textsc{Cut} (see Table~\ref{tab::normalized2cut}).

			\subsection{Population selection}
			\label{sec::update}
			
			A possible solution set $\mathbb{P}=\mathbb{I}\cup\mathbb{M}$ is composed of $\mathbb{I}$, the current population with size being $|\mathbb{I}|=m$ and $\mathbb{M}$, the solutions after mutation with size being $|\mathbb{M}|=2r$ (see Algorithm \ref{algorithm::framework}, line 6). Instead of exclusively selecting the top $m$ best solutions in $\mathbb{P}$ as the new population, we select the top $m/2$ best solutions in $\mathbb{P}$ to constitute half of the population, denoted as $P_1$. Such selection ensures the quality of the new population.  The remaining half of the population, denoted by $P_2$,  
			is obtained by solving the following Maximum Diversity Problem (MDP) \cite{Duarte2006,Marti2013} that introduces diversity to the new population,  
			$$
			P_2 = \argmax_{P_2 \subset \mathbb{P} \backslash P_1, |P_2|=m/2} \text{diversity}(P_1 \cup P_2),
			$$
			where the diversity function is defined for any $p$ solutions $\left\{I_t\right\}_{t=1}^p$ as follows
			$$\text{diversity}\left(\left\{I_t\right\}_{t=1}^p\right) = \sum_{1 \leq i < j \leq p} \dist_C(I_i,I_j).$$
			Here the distance metric $\dist_C(I_i,I_j)$ counts the number of edges located in exactly one of the solutions $I_i$ and $I_j$ \cite{Marti2009}. Algorithm~\ref{algorithm::selection} displays the entire procedure for obtaining the new polulation from the set $\mathbb{P}$ where a greedy local search method is adopted to solve MDP.

			\begin{breakablealgorithm}
				\caption{Population selection from the possible solution set $\mathbb{P}$ for $\opt\, F(\cdot)$.} 
				\label{algorithm::selection}
				\begin{algorithmic}[1] \small
					\Require  $m+2r$ : number of solutions in  $\mathbb{P}$. % The population size $m$ and the number of possible solution pairs $r$.
					\State Sort $\mathbb{P}$ in ascending order of the objective function values $F$ if $\opt=\min$, while sort $\mathbb{P}$ in descending order of $F$ if $\opt=\max$.
					\State $P_1=\{I_t\}_{t=1}^{m/2}$, $P_2=\{I_t\}_{t=m/2+1}^{m}$ and $P_3=\{I_t\}_{t=m+1}^{m+2r}.$
					\While{$\text{diversity}(P_1 \cup P_2)$ can be improved}
					\State  Swap the solution $I_i\in P_2$ and $I_j\in P_3$ when $\text{diversity}(P_1\cup P_2)$ increases the most.
					\EndWhile
					\State \Return{$P_1\cup P_2$}.
				\end{algorithmic}  
			\end{breakablealgorithm}
			
			\section{Experiment settings}
			\label{sec::experiment}
			
			%			In this paper, we focus on all $9$ graph $k$-CUT problems ($k=2,3,4,5$) in Section \ref{sec::notation}. 
			
			We conduct numerical experiments for all $9$ graph $k$-CUT problems with $k=2,3,4,5$ in Section \ref{sec::notation}
			on $35$ graphs with non-negative weights from the well-known benchmark G-set.
            % \footnote{Downloaded from \href{https://web.stanford.edu/~yyye/yyye/Gset/}{https://web.stanford.edu/$\sim$yyye/yyye/Gset/}}. 
            Both MMH in Algorithm~\ref{algorithm::combinatorial} and ACMH in Algorithm~\ref{algorithm::acmh} are integrated into PEAF in Algorithm~\ref{algorithm::framework},
			and the resulting algorithms are dubbed PEAF-MMH and PEAF-ACMH, respectively. 
			We set the population size $m=8$ in Algorithm~\ref{algorithm::selection} and the number of possible solution pairs $r=4$ in Algorithm~\ref{algorithm::framework}.  
			The stopping criterion is a predefined runtime limit:
			\begin{equation}\label{eq:timelimit}\left\{
				\begin{aligned}
					&15\text{ minutes},&\text{if }\opt\,F\in\{\text{\textsc{Min}-$k$-\textsc{Cut}, \textsc{MinMax}-$k$-\textsc{Cut}}\},\\
					&\left\{
					\begin{aligned}
						&30\text{ minutes}, &\text{if }|V|< 5000,\\
						&120\text{ minutes},&\text{if }5000\leq |V|< 10000,\\
						&240\text{ minutes}, &\text{if }|V|\geq 10000,
					\end{aligned}
					\right.
					&\text{otherwise.}
				\end{aligned}
				\right.
			\end{equation}
			In view of that MOH itself is an efficient algorithm to solve \textsc{Max}-$k$-\textsc{Cut} \cite{Ma2015}, 
			we also conduct numerical experiments utilizing the mutation-only algorithms MMH and ACMH to solve generalized graph $k$-CUT problems.
			That is, we will make a performance comparison among PEAF-MMH, PEAF-ACMH, MMH and ACMH on the same footing,
			all of which are compiled using GNU G++ with compiling flags ``\texttt{-pthread -O2}'' on the computing platform of 2*Intel Xeon E5-2650-v4/2.2GHz with 128GB RAM.
			For all of them, we select the tabu tenure $L_i$ for vertex $i$ by
			\begin{equation}
				\label{tabu::tenure}
				L_i\in \randn(3, L_{\max}),\quad L_{\max}\in \randn\left(0.05 |V|, 0.15 |V|\right),
			\end{equation}
			where $\randn(a,b)$ denotes a random integer between $a$ and $b$. Other remaining parameters are detailed in Table~\ref{tab::parameter}.
			The auxiliary cuts needed in ACMH and PEAF-ACMH are listed in Table~\ref{tab::cut-combine}. 
            {We use \texttt{Gurobi} (version 10) as a reference solver under two settings,
G-H and G-0.5H. Both settings use Gurobi's NoRel heuristic, which searches for
feasible solutions before the root-node relaxation. G-H sets
\texttt{NoRelHeurTime=timeLimit}, allowing this heuristic to use the whole time
budget, while G-0.5H sets \texttt{NoRelHeurTime=0.5*timeLimit}, leaving the
remaining time for Gurobi's standard MIP search.}            
			% Meanwhile, we will adopt the results produced by \texttt{Gurobi} (version 10) as reference solutions, and two modes, dubbed G-H and G-0.5H, will be employed:
			% G-H  sets \texttt{NoRelHeurTime=timeLimit} for triggering the heuristic option to spend the entire runtime and G-0.5H \texttt{NoRelHeurTime=0.5timeLimit} for a heuristic search that spends half of the runtime before solving the root relaxation.
			We run G-H and G-0.5H on the platform of AMD TR1950X/3.4GHz with 128GB RAM under the same runtime limits in Eq.~\eqref{eq:timelimit}.

			%			to generate reference solutions for comparative analysis. To exploit the heuristic capabilities of \texttt{Gurobi}, we set \texttt{NoRelHeurTime=timeLimit}, named G-H. To achieve a balance between heuristic and deterministic methods, we set \texttt{NoRelHeurTime=0.5timeLimit}, named G-0.5H. Both are executed under the same runtime limits as our proposed algorithms.

			%			To ensure a rigorous comparison with mutation-only algorithms, both the MMH and ACMH algorithms are executed using the same $8$ initial solutions with the same time limits for each graph.
			
			%			\blue{For each $k$-CUT problem on each graph instance, we employ the widely recognized solver \texttt{Gurobi} (version 10) to generate reference solutions for comparative analysis. To exploit the heuristic capabilities of \texttt{Gurobi}, we set \texttt{NoRelHeurTime=timeLimit}, named G-H. To achieve a balance between heuristic and deterministic methods, we set \texttt{NoRelHeurTime=0.5timeLimit}, named G-0.5H. Both are executed under the same runtime limits as our proposed algorithms.}

			%			\subsection{Experimental setup}
			%			Our proposed PEA-ACMH, PEA-MMH, ACMH and MMH algorithms are compiled using GNU G++ with compiling flags ``\texttt{-pthread -O2}'' on the computing platform of 2*Intel Xeon E5-2650-v4/2.2GHz with 128GB RAM.  G-H and G-0.5H are implemented on the platform of AMD TR1950X/3.4GHz with 128GB RAM.
			
			%\subsection{Parameter settings}

\begin{table}[htbp]
	\centering
	\small
	\caption[Auxiliary cuts]{{Parameter settings of MMH, ACMH, PEAF-MMH, and
	PEAF-ACMH on G-set. All four algorithms use the runtime limit in
	Eq.~\eqref{eq:timelimit} as the overall stopping condition. The ``module limit'' gives the maximum number of search operations allowed for
	each internal MMH or ACMH call; all unspecified parameters are inherited
	from the corresponding lower-level module.}}
	\label{tab::parameter}
	\begin{tabular}{|p{0.18\linewidth}|p{0.25\linewidth}|p{0.47\linewidth}|}
		\hline
		Algorithms & Items & Parameter settings \\ \hline
		\multirow{4}{*}{MMH}
		& $\omega$ & $0.1|V|$ \\ \cline{2-3}
		& $\xi$ & MaxGCP: 1000; MinGCP: 100 \\ \cline{2-3}
		& $\eta$ & $\randn(0.1|V|,0.3|V|)$ \\ \cline{2-3}
		& $(\gamma_0,\gamma_1,\gamma_2,\gamma_3)$
		& MaxGCP: $(1,0.5,1,1)$;
		MinGCP: $(1,0.5,1,1)$ for G1--G54, and $(0,0,0,0.5)$ for the other graphs \\ \hline
		\multirow{2}{*}{ACMH}
		& Auxiliary problems for $\opt\,F$
		& Table~\ref{tab::cut-combine} \\ \cline{2-3}
		& MMH block
		& module limit $n_{\max}=10^5$; MMH parameters as above \\ \hline
		PEAF-MMH
		& \textsc{parallel\_mutation} $=$ MMH
		& module limit $2\times 10^6$; MMH parameters as above \\ \hline
		PEAF-ACMH
		& \textsc{parallel\_mutation} $=$ ACMH
		& module limit $2\times 10^6$; ACMH settings as above \\ \hline
	\end{tabular}
\end{table}

			\begin{table}[htbp]
				\centering
				\caption[Auxiliary cuts]{The auxiliary cuts  $\opt\,F_1$ and $\opt\,F_2$ for the target problem $\opt\,F$. Compared to \textsc{Cheeger}-$k$-\textsc{Cut}, 	\textsc{Cheeger}$_2$-$k$-\textsc{Cut}~\cite{lee2014multiway} replaces  $\min\{\vol(S_p),\vol(S_p^c)\}$ with $\vol(S_p)$. Similarly, 	\textsc{AntiCheeger}$_2$-$k$-\textsc{Cut} and 	\textsc{Sparsest}$_2$-$k$-\textsc{Cut} are derived from \textsc{AntiCheeger}-$k$-\textsc{Cut} and \textsc{Sparsest}-$k$-\textsc{Cut}, respectively. \textsc{MinMax}-$k$-\textsc{Cut} employs a ``max-boundary" form in place of the ``sum-boundary" form utilized in  \textsc{Min}-$k$-\textsc{Cut}. Likewise, we introduce  \textsc{MaxMin}-$k$-\textsc{Cut} in a ``min-boundary" form as a replacement for the ``sum-boundary" form of \textsc{Max}-$k$-\textsc{Cut}.}
				\label{tab::cut-combine}
\begin{tabular}{|lll|}
					\hline
					$\opt\, F$                                  &  $\opt\,F_1$ & $\opt\,F_2$                    \\ \hline
					\multirow{2}{*}{\textsc{Max}-$k$-\textsc{Cut}}          & \textsc{MaxMin}-$k$-\textsc{Cut}:                & \multirow{2}{*}{}        \\
					& $\max \min_{1\leq p\leq k} \{\cut(\partial S_p)\}$ &                          \\ \hline
					\textsc{Judicious}-$k$-\textsc{Partition}               & \textsc{Max}-$k$-\textsc{Cut}                    & \textsc{AntiCheeger}-$k$-\textsc{Cut}     \\ \hline
					\multirow{2}{*}{\textsc{AntiCheeger}-$k$-\textsc{Cut}} & \multirow{2}{*}{\textsc{Max}-$k$-\textsc{Cut}}   & 	\textsc{AntiCheeger}$_2$-$k$-\textsc{Cut}: \\
					&                                & $\max\min_{1\leq p\leq k}\left\{\frac{\cut(\partial S_p)}{\max\{\vol(S_p),\vol(S_p^c)\}}\right\}$                        \\ \hline
					\textsc{Min}-$k$-\textsc{Cut}                           & \textsc{MinMax}-$k$-\textsc{Cut}                &                          \\ \hline
					\textsc{MinMax}-$k$-\textsc{Cut}                       & \textsc{Min}-$k$-\textsc{Cut}                    &                          \\
\hline
				\end{tabular}
			\end{table}

			\section{Numerical results for MaxGCP}

			This section presents numerical results of PEAF-ACMH, PEAF-MMH, ACMH, MMH, G-H and G-0.5H in solving three $k$-CUT problems in MaxGCP: \textsc{Judicious}-$k$-\textsc{Partition} \eqref{eq::jp-k-cut} (see Section~\ref{sec::result-jp}),  \textsc{AntiCheeger}-$k$-\textsc{Cut} \eqref{eq::ah-k-cut} (see Section~\ref{sec::result-anticheeger})
			and  \textsc{Max}-$k$-\textsc{Cut} \eqref{eq::max-k-cut} (see Section~\ref{sec::result-maxcut}). For each graph, we report the average objective function value of top 4 solutions, as only half of the population give the highest-quality solutions. {For the best solution, the columns best, cut, and time report the objective
function value, the corresponding cut value, and the elapsed time in seconds
when the best solution is first found, respectively. A smaller time means that
the algorithm reaches its best result earlier.}            
            % We also provide the objective function value, cut value, and birth time of the best solution. 
            Specifically, the best objective value among all those six algorithms is highlighted in bold. Using Table~\ref{tab::judicious2partition} as an example, the aforementioned data is recorded in columns 2-17, and the last two columns give the results of G-0.5H and G-H, respectively.
			
			To further display the effect of the evolutionary framework and the auxiliary cut strategy, 
			we summarize the pairwise comparisons of PEAF-ACMH, PEAF-MMH, ACMH, MMH on the best and average objective function values (see e.g. Table~\ref{tab::jp-summary}). This results in a total of 12 groups of comparisons for each $k$-CUT problem, where each group records the number of instances in which the former outperforms, equals, or underperforms the latter. At the same time, we employ the Wilcoxon signed-rank test to statistically analyze the results, reporting the corresponding p-value. If the p-value is less than 0.05, the two algorithms are considered statistically significantly (SS) different, denoted as ``SS=yes".

			\subsection{\textsc{Judicious}-$k$-\textsc{Partition}}
			\label{sec::result-jp}
			
			Tables \ref{tab::judicious2partition}, \ref{tab::judicious3partition}, \ref{tab::judicious4partition} and \ref{tab::judicious5partition} display the numerical results for \textsc{Judicious}-$k$-\textsc{Partition} with $k=2,3,4,5$. {It can be observed that PEAF-ACMH, PEAF-MMH, ACMH, and MMH outperform
\texttt{Gurobi} on all graph instances, with the only exception of G70 when
$k=2$.}
            % It can be easily observed that PEAF-ACMH, PEAF-MMH, ACMH and MMH outperform \texttt{Gurobi} on all graph instances except G70. 
            Comparisons (see Table~\ref{tab::jp-summary}) of PEAF-MMH, MMH and ACMH reveal that integrating MMH with the parallel evolutionary framework or the auxiliary cut strategy yields improved solutions. Moreover, the comparisons involving PEAF-ACMH indicate that it outperforms the other three algorithms.

			Now we demonstrate that PEAF-ACMH achieves high-quality solutions for \textsc{Judicious}-$k$-\textsc{Partition} with the help of best-known results in~\cite{Ma2015} for \textsc{Max}-$k$-\textsc{Cut}. For each graph instance $G=(V,E)$ and $k\in\{2,3,4,5\}$, let $\text{jpo}(G,k)$ and $\text{jpc}(G,k)$ be the best objective function value (e.g. Column 3 of Table~\ref{tab::judicious2partition}) and the corresponding cut value (e.g. Column 4 of Table~\ref{tab::judicious2partition}) generated by PEAF-ACMH, respectively. Let $\text{mc}(G,k)$ denote the best cut value in~\cite{Ma2015} that serves as the reference result for \textsc{Max-}$k$\textsc{-Cut} on each graph instance in G-set. We assume that $\text{mc}(G,k)\approx\textsc{Max-}k\textsc{-Cut}(G)$, and thus have
			\begin{equation}
				\label{neq:jp}
				\begin{aligned}
					\frac{1}{k}\left(\frac{1}{2}\vol(V)-\text{mc}(G,k)\right)&\approx	\frac{1}{k}\left(\frac{1}{2}\vol(V)-\textsc{Max-}k\textsc{-Cut}(G)\right)\\
					& \leq \textsc{Judicious-}k\text{-Partition}(G)\leq \text{jpo}(G,k),\\
					\frac{1}{k}\left(\frac{1}{2}\vol(V)-\text{mc}(G,k)\right)&\approx	\frac{1}{k}\left(\frac{1}{2}\vol(V)-\textsc{Max-}k\textsc{-Cut}(G)\right)\\&\leq \frac{1}{k}\left(\frac{1}{2}\vol(V)-\text{jpc}(G,k)\right)\leq \text{jpo}(G,k),
				\end{aligned}
			\end{equation}
			leading to the following two ratios, the lower bound over the upper bound:
			\begin{align}
				\label{jp:ratio1}
				\text{ratio}_1(G,k) &= \frac{\frac{1}{k}\left(\frac{1}{2}\vol(V)-\text{mc}(G,k)\right)}{\text{jpo}(G,k)}, \\
				\text{ratio}_2(G,k) &= \frac{\text{jpc}(G,k)}{\text{mc}(G,k)}. \label{jp:ratio2}
			\end{align}
            {The two ratios are used as empirical indicators of solution quality. 
Specifically, $\text{ratio}_1(G,k)$ compares the objective function value obtained by
PEAF-ACMH with the numerical bound derived from the reference
\textsc{Max}-$k$-\textsc{Cut} value $\text{mc}(G,k)$, while
$\text{ratio}_2(G,k)$ compares the cut value of the same solution with
$\text{mc}(G,k)$. Thus, when both ratios are close to 1, the obtained
\textsc{Judicious}-$k$-\textsc{Partition} solution is close to the reference
bound and also preserves a cut value comparable to the best-known
\textsc{Max}-$k$-\textsc{Cut} result.}
			% It is easily verified that,  the closer both ratios approach 1, the better the PEAF-ACMH approximations to \textsc{Judicious-}$k$\text{-Partition} are. 
            Accordingly, we can draw the conclusion that PEAF-ACMH achieves high-quality solutions through the numerical values shown in Table~\ref{tab:jp-quality}. {It shows that the PEAF-ACMH solutions are close to the numerical
bounds induced by the best-known \textsc{Max}-$k$-\textsc{Cut} values. The
percentage of instances with $|\text{ratio}_1-1|<0.05$ exceeds 90\% for
$k=2,3,4$ and remains 65\% for $k=5$. Moreover, the worst
$\text{ratio}_2$ over all instances is at least 0.99824, indicating
that the obtained judicious partitions almost preserve the reference
\textsc{Max}-$k$-\textsc{Cut} values. When $\text{ratio}_2=1$, the same cut
value is attained while the balance-related objective is further optimized.}
			
			%, and demonstrate that \textsc{Judicious-}$k$\textsc{-Partition} is a balanced variant of \textsc{Max-}$k$\textsc{-Cut}.

			\begin{table}[htbp]
				\centering
				\caption{{Summary of the two quality ratios in
Eqs.~\eqref{jp:ratio1}--\eqref{jp:ratio2} for
\textsc{Judicious}-$k$-\textsc{Partition} on G-set. The table reports the
percentage of instances with $|\text{ratio}_1-1|<0.05$, the worst value
$\min_G \text{ratio}_2(G,k)$, and the percentage of instances with
$\text{ratio}_2=1$.}}
				\label{tab:jp-quality}
\begin{tabular}{|l|l|l|l|l|}
					\hline
					& $k=2$     & $k=3$     & $k=4$     & $k=5$     \\ \hline
					percentage of $|\text{ratio}_1-1|< 0.05$          & 100\%     & 94.2857\% & 100\%     & 65.7143\% \\ \hline
					\multirow{2}{*}{$\min_{G}\text{ratio}_2(G,k)$} & 0.99852   & 0.99824   & 0.99865   & 0.99857   \\
					& G63       & G60       & G14       & G47       \\ \hline
					percentage of $\text{ratio}_2=1$                 & 37.1429\% & 48.5714\% & 31.4286\% & 42.8571\% \\ \hline
				\end{tabular}
			\end{table}

				{
					\scriptsize
					\setlength{\tabcolsep}{4.6pt}
					\renewcommand{\arraystretch}{1.2}
					\begin{landscape}
						\begin{longtable}{lllllllllllllllllll}
							\caption[Results for \textsc{Judicious}-$2$-\textsc{Partition}]{Numerical results for \textsc{Judicious}-$2$-\textsc{Partition} on G-set.}
							\label{tab::judicious2partition}\\
							\hline
							\multirow{2}{*}{instance} & \multicolumn{4}{l}{PEAF-ACMH}               & \multicolumn{4}{l}{PEAF-MMH}                & \multicolumn{4}{l}{ACMH}                   & \multicolumn{4}{l}{MMH}                    & \multirow{2}{*}{G-0.5H} & \multirow{2}{*}{G-H} \\ \cline{2-17}
							& mean    & best           & cut   & time      & mean    & best           & cut   & time     & mean    & best           & cut   & time      & mean    & best           & cut   & time     &                         &                      \\ \hline
							G1                        & 3776    & \textbf{3776} & 11624 & 2.348     & 3776    & \textbf{3776} & 11624 & 13.566   & 3776    & \textbf{3776} & 11624 & 1.919     & 3776    & \textbf{3776} & 11624 & 7.381    & 3885                    & 3861                 \\
							G2                        & 3779    & \textbf{3779} & 11619 & 7.993     & 3779    & \textbf{3779} & 11618 & 9.347    & 3779    & \textbf{3779} & 11619 & 7.654     & 3779    & \textbf{3779} & 11618 & 55.619   & 3847                    & 3836                 \\
							G3                        & 3778    & \textbf{3778} & 11620 & 6.277     & 3778    & \textbf{3778} & 11620 & 1.615    & 3778    & \textbf{3778} & 11620 & 2.024     & 3778    & \textbf{3778} & 11620 & 6.979    & 3894                    & 3894                 \\
							G4                        & 3766    & \textbf{3766} & 11646 & 7.279     & 3766    & \textbf{3766} & 11645 & 0.972    & 3766    & \textbf{3766} & 11646 & 1.926     & 3766    & \textbf{3766} & 11645 & 4.855    & 3875                    & 3843                 \\
\hline
						\end{longtable}
					\end{landscape}
				}

				{
					\scriptsize
					\setlength{\tabcolsep}{5pt}
					\renewcommand{\arraystretch}{1.2}
					\begin{landscape}
						\begin{longtable}{lllllllllllllllllll}
							\caption[Results for \textsc{Judicious}-$3$-\textsc{Partition}]{Numerical results for \textsc{Judicious}-$3$-\textsc{Partition} on G-set.}
							\label{tab::judicious3partition}\\
							\hline
							\multirow{2}{*}{instance} & \multicolumn{4}{l}{PEAF-ACMH}             & \multicolumn{4}{l}{PEAF-MMH}                & \multicolumn{4}{l}{ACMH}                  & \multicolumn{4}{l}{MMH}                   & \multirow{2}{*}{G-0.5H} & \multirow{2}{*}{G-H} \\ \cline{2-17}
							& mean   & best           & cut   & time     & mean    & best           & cut   & time     & mean    & best           & cut   & time     & mean    & best          & cut   & time     &                         &                      \\ \hline
							G1                        & 1338   & \textbf{1338} & 15162 & 370.667  & 1338    & \textbf{1338} & 15162 & 1254.400 & 1339    & \textbf{1338} & 15162 & 1045.590 & 1339.5  & 1339         & 15161 & 305.335  & 1416                    & 1415                 \\
							G2                        & 1336   & \textbf{1336} & 15170 & 918.488  & 1336    & \textbf{1336} & 15168 & 419.259  & 1338    & \textbf{1336} & 15170 & 109.221  & 1339    & 1338         & 15164 & 351.885  & 1429                    & 1428                 \\
							G3                        & 1335   & \textbf{1335} & 15171 & 384.335  & 1336    & 1336          & 15170 & 420.689  & 1335.75 & \textbf{1335} & 15171 & 55.644   & 1337    & 1337         & 15167 & 57.369   & 1420                    & 1419                 \\
							G4                        & 1331   & \textbf{1331} & 15183 & 985.717  & 1332    & 1332          & 15181 & 638.943  & 1332    & \textbf{1331} & 15183 & 773.233  & 1334.25 & 1332         & 15180 & 683.729  & 1451                    & 1444                 \\
\hline
						\end{longtable}
					\end{landscape}
				}

				{
					\scriptsize
					\setlength{\tabcolsep}{5.8pt}
					\renewcommand{\arraystretch}{1.2}
					\begin{landscape}
						\begin{longtable}{lllllllllllllllllll}
							\caption[Results for \textsc{Judicious}-$4$-\textsc{Partition}]{Numerical results for \textsc{Judicious}-$4$-\textsc{Partition} on G-set.}
							\label{tab::judicious4partition}\\
							\hline
							\multirow{2}{*}{instance} & \multicolumn{4}{l}{PEAF-ACMH}            & \multicolumn{4}{l}{PEAF-MMH}              & \multicolumn{4}{l}{ACMH}                & \multicolumn{4}{l}{MMH}                & \multirow{2}{*}{G-0.5H} & \multirow{2}{*}{G-H} \\ \cline{2-17}
							& mean   & best          & cut   & time     & mean   & best          & cut   & time     & mean   & best          & cut   & time     & mean   & best        & cut   & time     &                         &                      \\ \hline
							G1                        & 594.5  & \textbf{594} & 16801 & 1458.920 & 595    & 595          & 16798 & 980.622  & 596.75 & 595          & 16797 & 110.277  & 597.25 & 596        & 16792 & 421.293  & 675                     & 666                  \\
							G2                        & 593.25 & \textbf{593} & 16805 & 1098.460 & 594.25 & 594          & 16802 & 1212.180 & 595.25 & 595          & 16796 & 73.363   & 597.5  & 597        & 16790 & 563.818  & 675                     & 672                  \\
							G3                        & 594    & \textbf{594} & 16800 & 1160.260 & 595    & 595          & 16799 & 421.845  & 596.5  & 596          & 16794 & 544.323  & 597.75 & 597        & 16788 & 920.448  & 658                     & 657                  \\
							G4                        & 593    & \textbf{593} & 16805 & 68.715   & 594.25 & 594          & 16801 & 1018.260 & 594.75 & 594          & 16802 & 408.021  & 595.75 & 595        & 16796 & 1344.990 & 668                     & 665                  \\
\hline
						\end{longtable}
					\end{landscape}
				}

				{
					\scriptsize
					\setlength{\tabcolsep}{5.8pt}
					\renewcommand{\arraystretch}{1.2}
					\begin{landscape}
						\begin{longtable}{lllllllllllllllllll}
							\caption[Results for \textsc{Judicious}-$5$-\textsc{Partition}]{Numerical results for \textsc{Judicious}-$5$-\textsc{Partition} on G-set.}
							\label{tab::judicious5partition}\\
							\hline
							\multirow{2}{*}{instance} & \multicolumn{4}{l}{PEAF-ACMH}            & \multicolumn{4}{l}{PEAF-MMH}              & \multicolumn{4}{l}{ACMH}                & \multicolumn{4}{l}{MMH}                  & \multirow{2}{*}{G-0.5H} & \multirow{2}{*}{G-H} \\ \cline{2-17}
							& mean   & best          & cut   & time     & mean   & best          & cut   & time     & mean   & best          & cut   & time     & mean   & best          & cut   & time     &                         &                      \\ \hline
							G1                        & 297    & \textbf{297} & 17695 & 924.977  & 297    & \textbf{297} & 17695 & 1179.630 & 299.25 & 298          & 17689 & 146.625  & 300.25 & 299          & 17682 & 472.619  & 367                     & 365                  \\
							G2                        & 296    & \textbf{296} & 17696 & 950.846  & 297    & 297          & 17692 & 1353.490 & 297.75 & \textbf{296} & 17700 & 490.666  & 299.5  & 299          & 17683 & 411.585  & 358                     & 358                  \\
							G3                        & 297.75 & \textbf{297} & 17694 & 1138.340 & 297.25 & \textbf{297} & 17691 & 1439.950 & 298.25 & 298          & 17689 & 135.651  & 300    & 298          & 17688 & 805.079  & 355                     & 355                  \\
							G4                        & 296    & \textbf{296} & 17699 & 1288.780 & 298    & 298          & 17689 & 1261.000 & 298    & \textbf{296} & 17698 & 1035.380 & 298.25 & \textbf{296} & 17699 & 294.309  & 349                     & 348                  \\
\hline
						\end{longtable}
					\end{landscape}
				}
				
				{
					\scriptsize
					\setlength{\tabcolsep}{10pt}
					\renewcommand{\arraystretch}{1.6}
					\begin{landscape}
						\begin{longtable}{|l|lllll|lllll|}
							\caption[Summary for \textsc{Judicious}-$k$-\textsc{Partition}]{Summary of the results obtained by the proposed algorithms for \textsc{Judicious}-$k$-\textsc{Partition} ($k=2,3,4,5$) on G-set.}
							\label{tab::jp-summary}\\
							\hline
							\multicolumn{1}{|c|}{\multirow{2}{*}{}} & \multicolumn{5}{c|}{$k$=2}                    & \multicolumn{5}{c|}{$k$=3}                    \\ \cline{2-11} 
							\multicolumn{1}{|c|}{}                  & wins & ties & losses & p-value          & SS  & wins & ties & losses & p-value          & SS  \\ \hline
							PEAF-ACMH versus ACMH best             & 12   & 23   & 0      & \textless{}0.001 & Yes & 19   & 16   & 0      & \textless{}0.001 & Yes \\
							PEAF-ACMH versus ACMH avg.             & 18   & 17   & 0      & \textless{}0.001 & Yes & 32   & 3    & 0      & \textless{}0.001 & Yes \\
							PEAF-ACMH versus PEAF-MMH best           & 19   & 16   & 0      & \textless{}0.001 & Yes & 26   & 9    & 0      & \textless{}0.001 & Yes \\
							PEAF-ACMH versus PEAF-MMH avg.           & 19   & 16   & 0      & \textless{}0.001 & Yes & 26   & 9    & 0      & \textless{}0.001 & Yes \\
							PEAF-ACMH versus MMH best               & 22   & 13   & 0      & \textless{}0.001 & Yes & 30   & 5    & 0      & \textless{}0.001 & Yes \\
\hline
						\end{longtable}
					\end{landscape}
				}

				\subsection{\textsc{AntiCheeger}-$k$-\textsc{Cut}}
				\label{sec::result-anticheeger}
				
				Tables \ref{tab::anticheeger2cut}, \ref{tab::anticheeger3cut}, \ref{tab::anticheeger4cut} and \ref{tab::anticheeger5cut} display the numerical results for \textsc{AntiCheeger}-$k$-\textsc{Cut} with $k=2,3,4,5$. {PEAF-ACMH, PEAF-MMH, ACMH, and MMH outperform \texttt{Gurobi} on all graph
instances, except for G70 at $k=2$.}
                % It can be easily observed that PEAF-ACMH, PEAF-MMH, ACMH and MMH outperform \texttt{Gurobi} on all graph instances except G70. 
                The pairwise comparisons involving PEAF-ACMH in Table \ref{tab::anticheeger-summary} shows that it outperforms the other three algorithms, and the comparison between PEAF-MMH and MMH gives that  the parallel evolutionary framework indeed improves the solution quality.

				Using a similar analysis in Section \ref{sec::result-jp}, we are able to show that PEAF-ACMH achieves high-quality solutions for \textsc{AntiCheeger}-$k$-\textsc{Cut} on G-set. For each graph instance $G=(V,E)$ and $k\in\{2,3,4,5\}$, let $\text{aco}(G,k)$ and $\text{acc}(G,k)$ be the best objective function value (e.g. column 3 of Table~\ref{tab::anticheeger2cut}) and the corresponding cut value (e.g. column 4 of Table~\ref{tab::anticheeger2cut}) generated by PEAF-ACMH, respectively. Thus, we have
				\begin{equation}
					\label{neq:ac}
					\begin{aligned}
						\frac{2\text{mc}(G,k)}{\vol(V)}\approx \frac{2\textsc{Max-}k\textsc{-Cut}(G)}{\vol(V)}&\geq \textsc{AntiCheeger-}k\textsc{-Cut}(G) \geq \text{aco}(G,k),\\
						\frac{2\text{mc}(G,k)}{\vol(V)}\approx \frac{2\textsc{Max-}k\textsc{-Cut}(G)}{\vol(V)}&\geq\frac{2\text{acc}(G,k)}{\vol(V)} \geq \text{aco}(G,k),
					\end{aligned}
				\end{equation}
				and using the lower bound over the upper bound leads to 
				\begin{align}
					\label{ac:ratio1}
					\text{ratio}_1(G,k) &=\frac{\vol(V)\text{aco}(G,k)}{2\text{mc}(G,k)}, \\
					\text{ratio}_2(G,k) &=\frac{\text{acc}(G,k)}{\text{mc}(G,k)}. \label{ac:ratio2}
				\end{align}
{The two ratios assess the PEAF-ACMH solution from two aspects. The first ratio
compares the obtained \textsc{AntiCheeger}-$k$-\textsc{Cut} objective function value
with the reference bound $2\text{mc}(G,k)/\vol(V)$. The second ratio compares
the cut value of the same solution with the reference
\textsc{Max}-$k$-\textsc{Cut} value $\text{mc}(G,k)$. Thus, when both ratios
are close to 1, the obtained solution is close to the reference bound in terms
of the \textsc{AntiCheeger}-$k$-\textsc{Cut} objective and also has a cut value close to the
best-known \textsc{Max}-$k$-\textsc{Cut} result. Since $\text{mc}(G,k)$ is used
as a proxy for $\textsc{Max}$-$k$-\textsc{Cut}$(G)$, these ratios are interpreted
as numerical quality indicators.}
				% The closer both ratios approach 1, the better the PEAF-ACMH approximations to \textsc{AntiCheeger}-$k$-\textsc{Cut} are. 
Accordingly, we can draw the conclusion that PEAF-ACMH achieves high-quality solutions through the numerical values presented in Table~\ref{tab:ac-quality}. {Table~\ref{tab:ac-quality} shows that the PEAF-ACMH solutions are very close to the
reference bound $2\text{mc}(G,k)/\vol(V)$ for
\textsc{AntiCheeger}-$k$-\textsc{Cut}: all instances satisfy
$|\text{ratio}_1-1|<0.05$ for every $k=2,3,4,5$. The worst value of
$\text{ratio}_2$ is at least 0.99819, indicating that the cut values of the
obtained solutions almost match the best-known \textsc{Max}-$k$-\textsc{Cut}
values. When $\text{ratio}_2=1$, the solution attains the same cut value as the
reference \textsc{Max}-$k$-\textsc{Cut} result while further optimizing the
balance-oriented objective.}
				
				%, and demonstrate that \textsc{AntiCheeger}-$k$-\textsc{Cut} is a balanced variant of \textsc{Max-}$k$\textsc{-Cut}.

				\begin{table}[htbp]
					\centering
                    \caption{{Summary of the two quality ratios in
Eqs.~\eqref{ac:ratio1}--\eqref{ac:ratio2} for
\textsc{AntiCheeger}-$k$-\textsc{Cut} on G-set. The table reports the
percentage of instances with $|\text{ratio}_1-1|<0.05$, the worst value of
$\text{ratio}_2$, and the percentage of instances with $\text{ratio}_2=1$.}}
					% \caption{Summary of two ratios in Eqs.~\eqref{ac:ratio1} and \eqref{ac:ratio2} for \textsc{AntiCheeger}-$k$-\textsc{Cut} on G-set.}
					\label{tab:ac-quality}
					\begin{tabular}{|l|l|l|l|l|}
						\hline
						& $k=2$     & $k=3$     & $k=4$     & $k=5$     \\ \hline
						percentage of $|\text{ratio}_1-1|< 0.05$          & 100\%     & 100\% & 100\%     & 100\% \\ \hline
						\multirow{2}{*}{$\min_{G}\text{ratio}_2(G,k)$} & 0.99822   & 0.99819  & 0.99878   & 0.99872   \\
						& G63       & G60       & G58       & G23       \\ \hline
						percentage of $\text{ratio}_2=1$                 & 60\% & 54.2857\% & 48.5714\% & 57.1429\% \\ \hline
					\end{tabular}
				\end{table}

				\textbf{Balance analysis} 
				It can be numerically observed that 
				\textsc{Judicious}-$k$-\textsc{Partition} behaves more balanced than \textsc{AntiCheeger}-$k$-\textsc{Cut}
				since the latter exhibits 71 solutions with larger cut values, whereas the former only yields 19 larger cut values,
				out of 140 best solutions produced by PEAF-ACMH. That is, \textsc{AntiCheeger}-$k$-\textsc{Cut} tends to have larger cut value and is closer to \textsc{Max}-$k$-\textsc{Cut}, thereby producing more unbalanced cuts, than \textsc{Judicious}-$k$-\textsc{Partition}.
				This can be also verified using a similar analysis proposed in \cite{ding2001min} for $k=2$. 
				Assume the balanced case $\vol(S)\simeq \vol(S^c)$ for both \textsc{Judicious}-$2$-\textsc{Partition} and \textsc{AntiCheeger}-$2$-\textsc{Cut} and let
				$$\cut(\partial S)=\alpha\cdot\left(\frac{1}{2}\vol(V)-\cut(\partial S)\right),$$
				thus we have
				\begin{align*}
					\textsc{Judicious-}2 \textsc{-Partition}\quad&\simeq\quad \frac{\vol(V)}{2(1+\alpha)}, \\
					\textsc{AntiCheeger-}2 \textsc{-Cut}\quad&\simeq\quad \frac{\alpha}{1+\alpha}.
				\end{align*}
				If there exists a skewed cut $\{S_1,S_1^c\}$ that satisfies $\vol(S_1)-\cut(\partial S_1)\ll \vol(S_1^c)-\cut(\partial S_1)$ and is superior to $\{S,S^c\}$ for both 2-CUT problems, then we have $\vol(S_1^c)+\cut(\partial S_1)\simeq \vol(V)$  and
				\begin{align*}
					\vol(S_1^c)&<\left(\frac{3+2\alpha}{1+2\alpha}\right)\cut(\partial S_1) \text{ for }\textsc{Judicious-}2 \textsc{-Partition}, \\
					\vol(S_1^c)&<\left(\frac{1+\alpha}{\alpha}\right)\cut(\partial S_1) \text{ for }\textsc{AntiCheeger-}2 \textsc{-Cut}, \\
				\end{align*}
				which implies \textsc{Judicious}-$2$-\textsc{Partition} requires more effort to be skewed, thus being more balanced than \textsc{AntiCheeger}-$2$-\textsc{Cut} because 	 
				$$
				\frac{1+\alpha}{\alpha}  > \frac{3+2\alpha}{1+2\alpha}
				$$
				is always true for arbitrary positive $\alpha$. %$\forall \alpha>0$. 
				From above observations, we are able to see that numerical solutions of high-quality indeed help to transparently validate the combinatorial characteristics, such as the degree of balance discussed here, which may be easily neglected by those of low-quality.

				{
					\scriptsize
					\setlength{\tabcolsep}{3pt}
					\renewcommand{\arraystretch}{1.2}
					\begin{landscape}
\begin{longtable}{lllllllllllllllllll}
							\caption[Results for \textsc{AntiCheeger}-$2$-\textsc{Cut}]{Numerical results for \textsc{AntiCheeger}-$2$-\textsc{Cut} on G-set.}
							\label{tab::anticheeger2cut}\\
							\hline
							\multirow{2}{*}{instance} & \multicolumn{4}{l}{PEAF-ACMH}            & \multicolumn{4}{l}{PEAF-MMH}              & \multicolumn{4}{l}{ACMH}                & \multicolumn{4}{l}{MMH}                  & \multirow{2}{*}{G-0.5H} & \multirow{2}{*}{G-H} \\ \cline{2-17}
							& mean   & best          & cut   & time     & mean   & best          & cut   & time     & mean   & best          & cut   & time     & mean   & best          & cut   & time     &                         &                      \\ \hline
							G1                        & 0.60617 & \textbf{0.60617} & 11624 & 6.983     & 0.60617 & \textbf{0.60617} & 11624 & 3.342     & 0.60617 & \textbf{0.60617} & 11624 & 3.340     & 0.60617 & \textbf{0.60617} & 11624 & 1.142     & 0.59796                 & 0.59825              \\
							G2                        & 0.60588 & \textbf{0.60588} & 11619 & 12.693    & 0.60588 & \textbf{0.60588} & 11619 & 47.143    & 0.60588 & \textbf{0.60588} & 11619 & 23.456    & 0.60588 & \textbf{0.60588} & 11619 & 6.911     & 0.59940                 & 0.59940              \\
							G3                        & 0.60597 & \textbf{0.60597} & 11620 & 2.408     & 0.60597 & \textbf{0.60597} & 11620 & 3.106     & 0.60597 & \textbf{0.60597} & 11620 & 6.069     & 0.60597 & \textbf{0.60597} & 11620 & 2.034     & 0.59679                 & 0.59731              \\
							G4                        & 0.60726 & \textbf{0.60726} & 11646 & 2.317     & 0.60726 & \textbf{0.60726} & 11646 & 0.185     & 0.60726 & \textbf{0.60726} & 11646 & 6.006     & 0.60726 & \textbf{0.60726} & 11646 & 2.165     & 0.59921                 & 0.60002              \\
\hline
						\end{longtable}
					\end{landscape}
				}

				{
					\scriptsize
					\setlength{\tabcolsep}{3pt}
					\renewcommand{\arraystretch}{1.2}
					\begin{landscape}
						\begin{longtable}{lllllllllllllllllll}
							\caption[Results for \textsc{AntiCheeger}-$3$-\textsc{Cut}]{Numerical results for \textsc{AntiCheeger}-$3$-\textsc{Cut} on G-set.}
							\label{tab::anticheeger3cut}\\
							\hline
							\multirow{2}{*}{instance} & \multicolumn{4}{l}{PEAF-ACMH}            & \multicolumn{4}{l}{PEAF-MMH}              & \multicolumn{4}{l}{ACMH}                & \multicolumn{4}{l}{MMH}                  & \multirow{2}{*}{G-0.5H} & \multirow{2}{*}{G-H} \\ \cline{2-17}
							& mean   & best          & cut   & time     & mean   & best          & cut   & time     & mean   & best          & cut   & time     & mean   & best          & cut   & time     &                         &                      \\ \hline
							G1                        & 0.79071 & \textbf{0.79071} & 15163 & 855.290  & 0.79068 & 0.79068          & 15163 & 1111.310 & 0.79044 & 0.79060          & 15161 & 353.705  & 0.79048 & 0.79058          & 15161 & 1529.220 & 0.76902                 & 0.77004              \\
							G2                        & 0.79108 & \textbf{0.79108} & 15170 & 520.912  & 0.79105 & 0.79105          & 15170 & 216.783  & 0.79078 & \textbf{0.79108} & 15170 & 152.398  & 0.79070 & 0.79085          & 15166 & 957.626  & 0.76978                 & 0.77028              \\
							G3                        & 0.79118 & \textbf{0.79118} & 15172 & 626.689  & 0.79118 & \textbf{0.79118} & 15172 & 802.879  & 0.79097 & 0.79107          & 15171 & 920.507  & 0.79104 & 0.79115          & 15173 & 1246.970 & 0.77067                 & 0.77331              \\
							G4                        & 0.79176 & \textbf{0.79176} & 15183 & 523.199  & 0.79151 & 0.79159          & 15182 & 1735.020 & 0.79150 & \textbf{0.79176} & 15183 & 261.763  & 0.79126 & 0.79132          & 15175 & 758.613  & 0.76607                 & 0.76852              \\
\hline
						\end{longtable}
					\end{landscape}
				}

				{
					\scriptsize
					\setlength{\tabcolsep}{3pt}
					\renewcommand{\arraystretch}{1.2}
					\begin{landscape}
						\begin{longtable}{lllllllllllllllllll}
							\caption[Results for \textsc{AntiCheeger}-$4$-\textsc{Cut}]{Numerical results for \textsc{AntiCheeger}-$4$-\textsc{Cut} on G-set.}
							\label{tab::anticheeger4cut}\\
							\hline
							\multirow{2}{*}{instance} & \multicolumn{4}{l}{PEAF-ACMH}            & \multicolumn{4}{l}{PEAF-MMH}              & \multicolumn{4}{l}{ACMH}                & \multicolumn{4}{l}{MMH}                  & \multirow{2}{*}{G-0.5H} & \multirow{2}{*}{G-H} \\ \cline{2-17}
							& mean   & best          & cut   & time     & mean   & best          & cut   & time     & mean   & best          & cut   & time     & mean   & best          & cut   & time     &                         &                      \\ \hline
							G1                        & 0.87605 & \textbf{0.87613} & 16801 & 1381.210 & 0.87586 & 0.87588          & 16798 & 1414.780 & 0.87575 & 0.87604          & 16801 & 1245.690 & 0.87573 & 0.87594          & 16798 & 1454.190 & 0.84794                 & 0.84919              \\
							G2                        & 0.87628 & \textbf{0.87630} & 16805 & 1716.480 & 0.87617 & 0.87618          & 16803 & 1737.200 & 0.87590 & 0.87625          & 16805 & 842.750  & 0.87560 & 0.87609          & 16801 & 998.282  & 0.84743                 & 0.85057              \\
							G3                        & 0.87623 & \textbf{0.87624} & 16804 & 1694.640 & 0.87604 & 0.87607          & 16801 & 1246.470 & 0.87594 & 0.87612          & 16803 & 717.675  & 0.87594 & 0.87617          & 16803 & 1660.180 & 0.84140                 & 0.85330              \\
							G4                        & 0.87654 & \textbf{0.87655} & 16810 & 1525.450 & 0.87637 & 0.87641          & 16808 & 1459.810 & 0.87594 & 0.87615          & 16802 & 1409.830 & 0.87582 & 0.87611          & 16803 & 216.404  & 0.85252                 & 0.85127              \\
\hline
						\end{longtable}
					\end{landscape}
				}

				{
					\scriptsize
					\setlength{\tabcolsep}{3pt}
					\renewcommand{\arraystretch}{1.2}
					\begin{landscape}
						\begin{longtable}{lllllllllllllllllll}
							\caption[Results for \textsc{AntiCheeger}-$5$-\textsc{Cut}]{Numerical results for \textsc{AntiCheeger}-$5$-\textsc{Cut} on G-set.}
							\label{tab::anticheeger5cut}\\
							\hline
							\multirow{2}{*}{instance} & \multicolumn{4}{l}{PEAF-ACMH}            & \multicolumn{4}{l}{PEAF-MMH}              & \multicolumn{4}{l}{ACMH}                & \multicolumn{4}{l}{MMH}                  & \multirow{2}{*}{G-0.5H} & \multirow{2}{*}{G-H} \\ \cline{2-17}
							& mean   & best          & cut   & time     & mean   & best          & cut   & time     & mean   & best          & cut   & time     & mean   & best          & cut   & time     &                         &                      \\ \hline
							G1                        & 0.92286 & \textbf{0.92289} & 17700 & 1711.240 & 0.92227 & 0.92230          & 17690 & 707.029  & 0.92216 & 0.92220          & 17686 & 607.568  & 0.92200 & 0.92214          & 17684 & 223.688  & 0.87807                 & 0.90035              \\
							G2                        & 0.92277 & \textbf{0.92285} & 17697 & 1695.610 & 0.92262 & 0.92263          & 17695 & 1068.100 & 0.92248 & 0.92262          & 17694 & 1615.320 & 0.92213 & 0.92257          & 17692 & 922.593  & 0.89124                 & 0.89881              \\
							G3                        & 0.92285 & \textbf{0.92293} & 17701 & 1259.880 & 0.92241 & 0.92242          & 17692 & 1494.640 & 0.92225 & 0.92232          & 17688 & 1019.430 & 0.92216 & 0.92238          & 17691 & 1790.120 & 0.89049                 & 0.89791              \\
							G4                        & 0.92296 & \textbf{0.92305} & 17703 & 1783.380 & 0.92291 & 0.92292          & 17699 & 1473.570 & 0.92238 & 0.92283          & 17701 & 452.456  & 0.92248 & 0.92272          & 17697 & 831.672  & 0.87786                 & 0.89597              \\
\hline
						\end{longtable}
					\end{landscape}
				}

				{
					\scriptsize
					\setlength{\tabcolsep}{10pt}
					\renewcommand{\arraystretch}{1.6}
					\begin{landscape}
						\begin{longtable}{|l|lllll|lllll|}
							\caption[Summary for \textsc{AntiCheeger}-$k$-\textsc{Cut}]{Summary of the results obtained by the proposed algorithms for \textsc{AntiCheeger}-$k$-\textsc{Cut} ($k=2,3,4,5$) on G-set.}
							\label{tab::anticheeger-summary}\\
							\hline
							\multicolumn{1}{|c|}{\multirow{2}{*}{}} & \multicolumn{5}{c|}{$k$=2}                    & \multicolumn{5}{c|}{$k$=3}                    \\ \cline{2-11} 
							\multicolumn{1}{|c|}{}                  & wins & ties & losses & p-value          & SS  & wins & ties & losses & p-value          & SS  \\ \hline
							PEAF-ACMH versus ACMH best             & 12   & 23   & 0      & \textless{}0.001 & Yes & 27   & 8    & 0      & \textless{}0.001 & Yes \\
							PEAF-ACMH versus ACMH avg.             & 20   & 15   & 0      & \textless{}0.001 & Yes & 31   & 3    & 1      & \textless{}0.001 & Yes \\
							PEAF-ACMH versus PEAF-MMH best           & 16   & 19   & 0      & \textless{}0.001 & Yes & 30   & 5    & 0      & \textless{}0.001 & Yes \\
							PEAF-ACMH versus PEAF-MMH avg.           & 16   & 19   & 0      & \textless{}0.001 & Yes & 33   & 2    & 0      & \textless{}0.001 & Yes \\
							PEAF-ACMH versus MMH best               & 17   & 18   & 0      & \textless{}0.001 & Yes & 30   & 5    & 0      & \textless{}0.001 & Yes \\
\hline
						\end{longtable}
					\end{landscape}
				}

				\subsection{\textsc{Max}-$k$-\textsc{Cut}}
				\label{sec::result-maxcut}
				
				Tables \ref{tab::max2cut}, \ref{tab::max3cut}, \ref{tab::max4cut} and \ref{tab::max5cut} respectively record the numerical results for \textsc{Max}-$k$-\textsc{Cut} with $k=2, 3, 4, 5$. PEAF-ACMH, PEAF-MMH, ACMH, and MMH outperform \texttt{Gurobi} on all graph
instances, except for G70 at $k=2$. {G70 is extremely sparse, with
10000 vertices and 9999 edges. For $k=3,4,5$, both \texttt{Gurobi} and our
algorithms find partitions with no internal edges for
\textsc{Max}-$k$-\textsc{Cut}, \textsc{Judicious}-$k$-\textsc{Partition}, and
\textsc{AntiCheeger}-$k$-\textsc{Cut}. Hence, every edge is cut and the
partition induces $k$ independent sets. The result on G70 may therefore reflect
the special sparsity of this instance, rather than a general advantage of
\texttt{Gurobi}.}
                % It can be readily observed there that PEAF-ACMH, PEAF-MMH, ACMH and MMH outperform \texttt{Gurobi} on all graph instances except G70. 
				In Table \ref{tab::maxcut-summary}, we compare the results of ACMH and MMH, observing that MMH consistently outperforms ACMH, particularly for $k>2$. This may be attributed to the fact that the move-gain matrix $\widetilde{\mathcal{D}}$ directly stores the changes in the objective function values for \textsc{Max}-$k$-\textsc{Cut}, obviating the necessity for the auxiliary cut.
				
				%Now we prove that the best solutions obtained by the proposed four algorithms are high-quality solutions for \textsc{Max}-$k$-\textsc{Cut} on G-set. 
				
				For each graph instance $G=(V,E)$ and $k\in\{2,3,4,5\}$, let  $\text{mc}_1(G,k)$ be the best cut value (e.g. the best of Columns 3, 7, 11 and 15 in Table~\ref{tab::max2cut}), 
				and we calculate
				\begin{equation}
					\label{mc:ratio}
					\text{ratio}_0(G,k)=\frac{\text{mc}_1(G,k)}{\text{mc}(G,k)},  %\text{mc}_1(G,k)
				\end{equation}
				where $\text{mc}(G,k)$ denotes the best cut value in~\cite{Ma2015}. It is evident that the larger $\text{ratio}_0(G,k)$ is, the better the approximation of \textsc{Max}-$k$-\textsc{Cut} is. From Table~\ref{tab:mc-quality}, we know that our proposed algorithms produce approximation solutions that surpass the reported ones in \cite{Ma2015} especially for $k \in\{3, 4, 5\}$. % is set to 3, 4, 5.  
				Specifically, we list the graph instances and the corresponding cut values that satisfy $\text{ratio}_0>1$  in Table~\ref{tab::maxcut-better}.

				\begin{table}[htbp]
					\centering
					\caption{Summary of $\text{ratio}_0$ in Eq.~\eqref{mc:ratio}  for \textsc{Max}-$k$-\textsc{Cut} on G-set.}
					\label{tab:mc-quality}
\begin{tabular}{|l|l|l|l|l|}
						\hline
						& $k=2$     & $k=3$     & $k=4$     & $k=5$     \\ \hline
						percentage of $\text{ratio}_0\geq 1$         & 88.5714\%     & 97.1429\% & 100\%     & 82.8571\% \\ \hline
						percentage of $\text{ratio}_0>1$                 &2.8571\%&48.5714\% & 62.8571\% & 57.1429\%  \\ \hline
					\end{tabular}
				\end{table}
				
				{\scriptsize
					\setlength{\tabcolsep}{4.5pt}
					\renewcommand{\arraystretch}{1.5}
                    \begin{landscape}
					\begin{longtable}{|lll|lll|lll|lll|}
						\caption[Results for better maxcut]{Better \textsc{Max}-$k$-\textsc{Cut} approximations than \cite{Ma2015} on G-set: $\text{mc}_1(G,k)>\text{mc}(G,k)$. }
						\label{tab::maxcut-better}\\
						\hline
						instance & $\text{mc}_1(G,2)$ & $\text{mc}(G,2)$ & instance & $\text{mc}_1(G,3)$ & $\text{mc}(G,3)$ & instance & $\text{mc}_1(G,4)$ & $\text{mc}(G,4)$ & instance & $\text{mc}_1(G,5)$ & $\text{mc}(G,5)$\\ \hline
						G70      & 9556     & 9544     & G16      & 3992     & 3991     & G1       & 16804    & 16803    & G1       & 17705    & 17703    \\
						&          &          & G17      & 3986     & 3983     & G2       & 16810    & 16809    & G2       & 17707    & 17706    \\
						&          &          & G22      & 17169    & 17167    & G15      & 4407     & 4406     & G3       & 17703    & 17701    \\
						&          &          & G23      & 17172    & 17168    & G16      & 4418     & 4415     & G5       & 17711    & 17710    \\
						&          &          & G24      & 17163    & 17162    & G17      & 4415     & 4411     & G14      & 4642     & 4639     \\
						&          &          & G25      & 17166    & 17163    & G23      & 18779    & 18777    & G15      & 4607     & 4606     \\
\hline
					\end{longtable}
                    					\end{landscape}

				}
				
				%				Specifically, 
				%				the best solutions obtained by 
				%				
				%				
				%				these four algorithms are at least as good as the best solutions reported in \cite{Ma2015} for 88.5714\%, 97.1429\%, 100\% and 82.8571\% of the instances, respectively. Moreover, when $k$ is set to 3, 4 and 5, the results surpass the best known solutions for 48.5714\%, 62.8571\% and 57.1429\% of the instances, respectively.

				%			demonstrate that all proposed algorithms produce high-quality solutions for \textsc{Max}-$k$-\textsc{Cut} ($k=2, 3, 4, 5$), with the auxiliary cut being \textsc{MaxMin}-$k$-\textsc{Cut} (see Table \ref{tab::cut-combine}). 

				%				These algorithms outperform Gurobi on all graph instances except G70. 			
				%				
				%				\blue{In Table \ref{tab::maxcut-summary}, we compare the results of ACMH and MMH, observing that MMH consistently outperforms ACMH, particularly for $k>2$. This may be attributed to the fact that the move-gain matrix $\widetilde{\mathcal{D}}$ directly stores the changes in the objective function values for \textsc{Max}-$k$-\textsc{Cut}, obviating the necessity for the auxiliary cut. When comparing PEAF-MMH with MMH, there is no statistically significant difference in terms of the best solution contests. However, PEAF-MMH demonstrates superior performance over MMH with respect to average objective function values.}

				{
					\scriptsize
					\setlength{\tabcolsep}{3.5pt}
					\renewcommand{\arraystretch}{1.2}
					\begin{landscape}
\begin{longtable}{lllllllllllllllllll}
							\caption[Results for \textsc{Max}-$2$-\textsc{Cut}]{Numerical results for \textsc{Max}-$2$-\textsc{Cut} on G-set.}
							\label{tab::max2cut}\\
							\hline
							\multirow{2}{*}{instance} & \multicolumn{4}{l}{PEAF-ACMH}                 & \multicolumn{4}{l}{PEAF-MMH}                  & \multicolumn{4}{l}{ACMH}                    & \multicolumn{4}{l}{MMH}                       & \multirow{2}{*}{G-0.5H} & \multirow{2}{*}{G-H} \\ \cline{2-17}
							& mean     & best            & cut   & time      & mean     & best            & cut   & time     & mean     & best            & cut   & time     & mean     & best            & cut   & time      &                         &                      \\ \hline
							G1                        & 11624    & \textbf{11624} & 11624 & 0.748     & 11624    & \textbf{11624} & 11624 & 0.361    & 11624    & \textbf{11624} & 11624 & 0.242    & 11624    & \textbf{11624} & 11624 & 0.137     & 11602                   & 11515                \\
							G2                        & 11620    & \textbf{11620} & 11620 & 0.508     & 11620    & \textbf{11620} & 11620 & 0.149    & 11620    & \textbf{11620} & 11620 & 0.318    & 11620    & \textbf{11620} & 11620 & 0.352     & 11561                   & 11561                \\
							G3                        & 11622    & \textbf{11622} & 11622 & 0.774     & 11622    & \textbf{11622} & 11622 & 0.316    & 11622    & \textbf{11622} & 11622 & 0.402    & 11622    & \textbf{11622} & 11622 & 0.138     & 11612                   & 11612                \\
							G4                        & 11646    & \textbf{11646} & 11646 & 9.354     & 11646    & \textbf{11646} & 11646 & 1.225    & 11646    & \textbf{11646} & 11646 & 0.312    & 11646    & \textbf{11646} & 11646 & 0.108     & 11641                   & 11610                \\
\hline
						\end{longtable}
					\end{landscape}
				}
				
				{
					\scriptsize
					\setlength{\tabcolsep}{3.5pt}
					\renewcommand{\arraystretch}{1.2}
					\begin{landscape}
						\begin{longtable}{lllllllllllllllllll}
							\caption[Results for \textsc{Max}-$3$-\textsc{Cut}]{Numerical results for \textsc{Max}-$3$-\textsc{Cut} on G-set.}
							\label{tab::max3cut}\\
							\hline
							\multirow{2}{*}{instance} & \multicolumn{4}{l}{PEAF-ACMH}                & \multicolumn{4}{l}{PEAF-MMH}                  & \multicolumn{4}{l}{ACMH}                    & \multicolumn{4}{l}{MMH}                      & \multirow{2}{*}{G-0.5H} & \multirow{2}{*}{G-H} \\ \cline{2-17}
							& mean     & best            & cut   & time     & mean     & best            & cut   & time     & mean     & best            & cut   & time     & mean     & best            & cut   & time     &                         &                      \\ \hline
							G1                        & 15165    & \textbf{15165} & 15165 & 12.153   & 15165    & \textbf{15165} & 15165 & 20.242   & 15165    & \textbf{15165} & 15165 & 62.918   & 15165    & \textbf{15165} & 15165 & 28.137   & 14916                   & 14939                \\
							G2                        & 15172    & \textbf{15172} & 15172 & 783.469  & 15172    & \textbf{15172} & 15172 & 118.982  & 15168    & \textbf{15172} & 15172 & 192.812  & 15170.5  & \textbf{15172} & 15172 & 2.032    & 14934                   & 14966                \\
							G3                        & 15173    & \textbf{15173} & 15173 & 9.389    & 15173    & \textbf{15173} & 15173 & 16.397   & 15173    & \textbf{15173} & 15173 & 68.874   & 15173    & \textbf{15173} & 15173 & 1.568    & 15019                   & 15019                \\
							G4                        & 15184    & \textbf{15184} & 15184 & 552.603  & 15184    & \textbf{15184} & 15184 & 312.650  & 15181    & \textbf{15184} & 15184 & 217.978  & 15183    & \textbf{15184} & 15184 & 33.440   & 14982                   & 15013                \\
\hline
						\end{longtable}
					\end{landscape}
				}
				
				{
					\scriptsize
					\setlength{\tabcolsep}{3.5pt}
					\renewcommand{\arraystretch}{1.2}
					\begin{landscape}
						\begin{longtable}{lllllllllllllllllll}
							\caption[Results for \textsc{Max}-$4$-\textsc{Cut}]{Numerical results for \textsc{Max}-$4$-\textsc{Cut} on G-set.}
							\label{tab::max4cut}\\
							\hline
							\multicolumn{1}{c}{\multirow{2}{*}{instance}} & \multicolumn{4}{c}{PEAF-ACMH}                & \multicolumn{4}{c}{PEAF-MMH}                  & \multicolumn{4}{c}{ACMH}                    & \multicolumn{4}{c}{MMH}                      & \multicolumn{1}{c}{\multirow{2}{*}{G-0.5H}} & \multicolumn{1}{c}{\multirow{2}{*}{G-H}} \\ \cline{2-17}
							\multicolumn{1}{c}{}                          & mean     & best            & cut   & time     & mean     & best            & cut   & time     & mean     & best            & cut   & time     & mean     & best            & cut   & time     & \multicolumn{1}{c}{}                        & \multicolumn{1}{c}{}                     \\ \hline
							G1                                            & 16804    & \textbf{16804} & 16804 & 936.547  & 16804    & \textbf{16804} & 16804 & 294.659  & 16795.75 & 16803          & 16803 & 1455.530 & 16799.25 & 16803          & 16803 & 172.860  & 16559                                       & 16566                                    \\
							G2                                            & 16809.25 & \textbf{16810} & 16810 & 1632.600 & 16810    & \textbf{16810} & 16810 & 50.218   & 16796.5  & 16809          & 16809 & 1242.040 & 16807    & \textbf{16810} & 16810 & 402.107  & 16548                                       & 16579                                    \\
							G3                                            & 16806    & \textbf{16806} & 16806 & 456.166  & 16806    & \textbf{16806} & 16806 & 50.512   & 16799    & \textbf{16806} & 16806 & 943.991  & 16797.25 & \textbf{16806} & 16806 & 275.509  & 16583                                       & 16635                                    \\
\hline
						\end{longtable}
					\end{landscape}
				}

				{
					\scriptsize
					\setlength{\tabcolsep}{3.5pt}
					\renewcommand{\arraystretch}{1.2}
					\begin{landscape}
						\begin{longtable}{lllllllllllllllllll}
							\caption[Results for \textsc{Max}-$5$-\textsc{Cut}]{Numerical results for \textsc{Max}-$5$-\textsc{Cut} on G-set.}
							\label{tab::max5cut}\\
							\hline
							\multirow{2}{*}{instance} & \multicolumn{4}{l}{PEAF-ACMH}                & \multicolumn{4}{l}{PEAF-MMH}                  & \multicolumn{4}{l}{ACMH}                    & \multicolumn{4}{l}{MMH}                      & \multirow{2}{*}{G-0.5H} & \multirow{2}{*}{G-H} \\ \cline{2-17}
							& mean     & best            & cut   & time     & mean     & best            & cut   & time     & mean     & best            & cut   & time     & mean     & best            & cut   & time     &                         &                      \\ \hline
							G1                        & 17705    & \textbf{17705} & 17705 & 1121.110 & 17705    & \textbf{17705} & 17705 & 1479.070 & 17690.5  & 17695          & 17695 & 682.824  & 17698.75 & \textbf{17705} & 17705 & 40.315   & 17471                   & 17506                \\
							G2                        & 17706.5  & \textbf{17707} & 17707 & 1407.450 & 17707    & \textbf{17707} & 17707 & 776.489  & 17690.75 & 17696          & 17696 & 82.146   & 17700.25 & 17706          & 17706 & 195.927  & 17515                   & 17524                \\
							G3                        & 17697.75 & \textbf{17703} & 17703 & 1737.050 & 17701    & 17701          & 17701 & 1584.950 & 17694.25 & 17701          & 17701 & 35.139   & 17697    & 17702          & 17702 & 1208.390 & 17473                   & 17486                \\
							G4                        & 17699    & 17699          & 17699 & 359.343  & 17709    & \textbf{17709} & 17709 & 901.897  & 17695    & 17698          & 17698 & 1010.710 & 17703.75 & \textbf{17709} & 17709 & 481.708  & 17491                   & 17517                \\
\hline
						\end{longtable}
					\end{landscape}
				}

				{
					\scriptsize
					\setlength{\tabcolsep}{10pt}
					\renewcommand{\arraystretch}{1.6}
					\begin{landscape}
						\begin{longtable}{|l|lllll|lllll|}
							\caption[Summary for \textsc{Max}-$k$-\textsc{Cut}]{Summary of the numerical results for \textsc{Max}-$k$-\textsc{Cut} ($k=2,3,4,5$) on G-set.}
							\label{tab::maxcut-summary}\\
							\hline
							\multicolumn{1}{|c|}{\multirow{2}{*}{}} & \multicolumn{5}{c|}{$k$=2}                    & \multicolumn{5}{c|}{$k$=3}                    \\ \cline{2-11} 
							\multicolumn{1}{|c|}{}                  & wins & ties & losses & p-value          & SS  & wins & ties & losses & p-value          & SS  \\ \hline
							PEAF-ACMH versus ACMH best             & 7    & 27   & 1      & 0.01563          & Yes & 13   & 20   & 2      & 0.00568          & Yes \\
							PEAF-ACMH versus ACMH avg.             & 18   & 17   & 0      & \textless{}0.001 & Yes & 25   & 6    & 4      & \textless{}0.001 & Yes \\
							PEAF-ACMH versus PEAF-MMH best           & 1    & 30   & 4      & 0.18750          & No  & 5    & 22   & 8      & 0.21289          & No  \\
							PEAF-ACMH versus PEAF-MMH avg.           & 4    & 26   & 5      & 1.00000          & No  & 9    & 18   & 8      & 0.30777          & No  \\
							PEAF-ACMH versus MMH best               & 4    & 28   & 3      & 0.56250          & No  & 4    & 22   & 9      & 0.11328          & No  \\
\hline
						\end{longtable}
					\end{landscape}
				}

				\section{Numerical results for MinGCP}
				\label{sec::minGCP}
				
				This section presents numerical results of PEAF-ACMH, PEAF-MMH, ACMH, MMH, G-H and G-0.5H in solving six $k$-CUT problems $(k=2,3,4,5)$ in MinGCP: \textsc{Normalized}-$k$-\textsc{Cut} \eqref{eq::normalized-k-cut},  \textsc{Cheeger}-$k$-\textsc{Cut} \eqref{eq::cheeger}, \textsc{Min}-$k$-\textsc{Cut} \eqref{eq::min-k-cut} \textsc{MinMax}-$k$-\textsc{Cut} \eqref{eq::minmax-k-cut}, \textsc{Ratio}-$k$-\textsc{Cut} \eqref{eq::ratio-k-cut} and \textsc{Sparsest}-$k$-\textsc{Cut} \eqref{eq::sparsest}. Similar to the discussions in the last section for MaxGCP, we present four tables of numerical results and a summary table for each problem. % in MinGCP.

			\subsection{\textsc{Normalized}-$k$-\textsc{Cut}}

			%				\blue{The outcomes presented in Tables \ref{tab::normalized2cut}, \ref{tab::normalized3cut}, \ref{tab::normalized4cut} and 
				%					\ref{tab::normalized5cut} illustrate the performance of the algorithms PEAF-ACMH, PEAF-MMH, ACMH, MOH, G-0.5H and G-H, utilizing \textsc{Ratio}-$k$-\textsc{Cut} and 	\textsc{Cheeger}$_2$-$k$-\textsc{Cut} as auxiliary cuts. Across all graph instances except for G70, the four proposed algorithms exhibit performance that is at least on par with Gurobi. Detailed pairwise comparisons in Table \ref{tab::normalized-summary} among PEAF-MMH, MOH and ACMH reveal that incorporating the parallel evolutionary framework or the cut-combined strategy to MOH enhances solution quality. For $k=5$, ACMH produces notably superior solutions compared to PEAF-MMH, while it is difficult to discern which method performs better for $k<5$. The comparisons involving PEAF-ACMH in Table \ref{tab::normalized-summary} demonstrate that PEAF-ACMH substantially outperforms the other three algorithms, indicating that the integration of both improved methods yields the best solutions.}

			Tables \ref{tab::normalized2cut}, \ref{tab::normalized3cut}, \ref{tab::normalized4cut} and 
			\ref{tab::normalized5cut} display the numerical results for \textsc{Normalized}-$k$-\textsc{Cut} with $k=2,3,4,5$. It can be easily observed that PEAF-ACMH, PEAF-MMH, ACMH and MMH outperform \texttt{Gurobi} on all graph instances except G70. Pairwise comparisons (see Table~\ref{tab::normalized-summary}) of PEAF-MMH, MMH and ACMH reveal that integrating MMH with the parallel evolutionary framework or the auxiliary cut strategy yields improved solutions. Moreover, the comparisons involving PEAF-ACMH indicate that it outperforms the other three algorithms.

			{
				\scriptsize
				\setlength{\tabcolsep}{3pt}
				\renewcommand{\arraystretch}{1.2}
				\begin{landscape}
\begin{longtable}{lllllllllllllllllll}
						\caption[Results for 	\textsc{Normalized}-$2$-\textsc{Cut}]{Numerical results for \textsc{Normalized}-$2$-\textsc{Cut} on G-set.}
						\label{tab::normalized2cut}\\
						\hline
						\multirow{2}{*}{instance} & \multicolumn{4}{l}{PEAF-ACMH}               & \multicolumn{4}{l}{PEAF-MMH}                & \multicolumn{4}{l}{ACMH}                   & \multicolumn{4}{l}{MMH}                    & \multirow{2}{*}{G-0.5H} & \multirow{2}{*}{G-H} \\ \cline{2-17}
						& mean    & best           & cut   & time      & mean    & best           & cut   & time     & mean    & best           & cut   & time      & mean    & best           & cut   & time     &                         &                      \\ \hline
						G1                        & 0.79116 & \textbf{0.79116} & 7540 & 244.332  & 0.79617 & 0.79617          & 7633 & 951.157   & 0.79147 & 0.79131          & 7559 & 1574.260 & 0.81091 & 0.80852          & 6789 & 269.728  & 0.81058                 & 0.81033              \\
						G2                        & 0.79030 & \textbf{0.79030} & 7568 & 186.448  & 0.79282 & 0.79277          & 7544 & 1666.400  & 0.79396 & \textbf{0.79030} & 7568 & 164.436  & 0.81118 & 0.80706          & 7303 & 281.910  & 0.80379                 & 0.80178              \\
						G3                        & 0.79025 & \textbf{0.79025} & 7570 & 74.315   & 0.79059 & 0.79055          & 7509 & 1777.350  & 0.79041 & \textbf{0.79025} & 7570 & 64.513   & 0.80695 & 0.79627          & 7580 & 454.286  & 0.80434                 & 0.80290              \\
						G4                        & 0.79141 & \textbf{0.79141} & 7588 & 338.800  & 0.79398 & 0.79394          & 7546 & 1500.410  & 0.79157 & 0.79149          & 7583 & 1437.450 & 0.81730 & 0.81482          & 6604 & 1745.810 & 0.80016                 & 0.79879              \\
\hline
					\end{longtable}
				\end{landscape}
			}
			
			{
				\scriptsize
				\setlength{\tabcolsep}{3pt}
				\renewcommand{\arraystretch}{1.2}
				\begin{landscape}
					\begin{longtable}{lllllllllllllllllll}
						\caption[Results for 	\textsc{Normalized}-$3$-\textsc{Cut}]{Numerical results for \textsc{Normalized}-$3$-\textsc{Cut} on G-set.}
						\label{tab::normalized3cut}\\
						\hline
						\multirow{2}{*}{instance} & \multicolumn{4}{l}{PEAF-ACMH}               & \multicolumn{4}{l}{PEAF-MMH}                & \multicolumn{4}{l}{ACMH}                   & \multicolumn{4}{l}{MMH}                    & \multirow{2}{*}{G-0.5H} & \multirow{2}{*}{G-H} \\ \cline{2-17}
						& mean    & best           & cut   & time      & mean    & best           & cut   & time     & mean    & best           & cut   & time      & mean    & best           & cut   & time     &                         &                      \\ \hline
						G1                        & 1.61212 & \textbf{1.61210} & 10283 & 1256.560  & 1.62343 & 1.62291          & 10277 & 1420.750  & 1.61274 & 1.61238          & 10294 & 124.589   & 1.66205 & 1.65851          & 10112 & 406.652   & 1.67685                 & 1.66421              \\
						G2                        & 1.60954 & \textbf{1.60954} & 10275 & 618.084   & 1.63087 & 1.63087          & 10280 & 604.242   & 1.61078 & 1.60999          & 10288 & 586.313   & 1.66266 & 1.65357          & 10352 & 1274.580  & 1.67127                 & 1.65999              \\
						G3                        & 1.60939 & \textbf{1.60939} & 10279 & 725.535   & 1.63032 & 1.63032          & 10339 & 1363.060  & 1.61235 & 1.61004          & 10291 & 164.905   & 1.65926 & 1.65367          & 9933  & 308.433   & 1.72803                 & 1.70345              \\
						G4                        & 1.61179 & \textbf{1.61177} & 10287 & 1471.640  & 1.63770 & 1.63770          & 10306 & 1206.810  & 1.61269 & 1.61241          & 10299 & 509.208   & 1.66863 & 1.65410          & 10266 & 390.046   & 1.71381                 & 1.68991              \\
\hline
					\end{longtable}
				\end{landscape}
			}

			{
				\scriptsize
				\setlength{\tabcolsep}{3pt}
				\renewcommand{\arraystretch}{1.2}
				\begin{landscape}
					\begin{longtable}{lllllllllllllllllll}
						\caption[Results for 	\textsc{Normalized}-$4$-\textsc{Cut}]{Numerical results for \textsc{Normalized}-$4$-\textsc{Cut} on G-set.}
						\label{tab::normalized4cut}\\
						\hline
						\multirow{2}{*}{instance} & \multicolumn{4}{l}{PEAF-ACMH}               & \multicolumn{4}{l}{PEAF-MMH}                & \multicolumn{4}{l}{ACMH}                   & \multicolumn{4}{l}{MMH}                    & \multirow{2}{*}{G-0.5H} & \multirow{2}{*}{G-H} \\ \cline{2-17}
						& mean    & best           & cut   & time      & mean    & best           & cut   & time     & mean    & best           & cut   & time      & mean    & best           & cut   & time     &                         &                      \\ \hline
						G1                        & 2.45135 & \textbf{2.45108} & 11746 & 1166.440  & 2.48230 & 2.48053          & 11470 & 1704.480  & 2.45335 & 2.45270          & 11754 & 592.872   & 2.51522 & 2.51483          & 10412 & 1755.480  & 2.62893                 & 2.55273              \\
						G2                        & 2.44799 & \textbf{2.44786} & 11730 & 1788.990  & 2.48044 & 2.48037          & 11565 & 1175.620  & 2.44941 & \textbf{2.44786} & 11730 & 231.054   & 2.51026 & 2.50814          & 11679 & 122.942   & 2.55167                 & 2.51194              \\
						G3                        & 2.44854 & \textbf{2.44854} & 11730 & 1402.480  & 2.50307 & 2.50306          & 10176 & 1123.740  & 2.45049 & 2.44890          & 11731 & 1537.070  & 2.50936 & 2.50699          & 10272 & 1344.550  & 2.70125                 & 2.55914              \\
						G4                        & 2.44924 & \textbf{2.44924} & 11737 & 377.537   & 2.49317 & 2.49317          & 11755 & 535.653   & 2.45388 & 2.45232          & 11754 & 22.392    & 2.51346 & 2.51213          & 10325 & 1474.560  & 2.68468                 & 2.63779              \\
\hline
					\end{longtable}
				\end{landscape}
			}

			{
				\scriptsize
				\setlength{\tabcolsep}{3pt}
				\renewcommand{\arraystretch}{1.2}
				\begin{landscape}
					\begin{longtable}{lllllllllllllllllll}
						\caption[Results for 	\textsc{Normalized}-$5$-\textsc{Cut}]{Numerical results for \textsc{Normalized}-$5$-\textsc{Cut} on G-set.}
						\label{tab::normalized5cut}\\
						\hline
						\multirow{2}{*}{instance} & \multicolumn{4}{l}{PEAF-ACMH}               & \multicolumn{4}{l}{PEAF-MMH}                & \multicolumn{4}{l}{ACMH}                   & \multicolumn{4}{l}{MMH}                    & \multirow{2}{*}{G-0.5H} & \multirow{2}{*}{G-H} \\ \cline{2-17}
						& mean    & best           & cut   & time      & mean    & best           & cut   & time     & mean    & best           & cut   & time      & mean    & best           & cut   & time     &                         &                      \\ \hline
						G1                        & 3.30231 & \textbf{3.30230} & 12658 & 1291.230  & 3.37679 & 3.37554          & 12767 & 604.742   & 3.30386 & 3.30325          & 12662 & 593.778   & 3.37680 & 3.36480          & 12864 & 1478.130 & 3.47614                 & 3.42777              \\
						G2                        & 3.30112 & \textbf{3.30112} & 12658 & 360.183   & 3.35063 & 3.34994          & 12717 & 1397.540  & 3.30366 & 3.30221          & 12657 & 855.164   & 3.37843 & 3.37657          & 11616 & 1115.560 & 3.51341                 & 3.42302              \\
						G3                        & 3.30042 & \textbf{3.30033} & 12653 & 1772.160  & 3.36980 & 3.36375          & 12650 & 294.805   & 3.30285 & 3.30087          & 12659 & 90.696    & 3.37280 & 3.36640          & 12686 & 447.282  & 3.53472                 & 3.44436              \\
						G4                        & 3.30079 & \textbf{3.30072} & 12656 & 1479.890  & 3.35443 & 3.35164          & 12765 & 1766.370  & 3.30547 & 3.30214          & 12657 & 152.172   & 3.38128 & 3.37928          & 11453 & 882.769  & 3.67565                 & 3.59908              \\
\hline
					\end{longtable}
				\end{landscape}
			}

			{
				\scriptsize
				\setlength{\tabcolsep}{10pt}
				\renewcommand{\arraystretch}{1.6}
				\begin{landscape}
					\begin{longtable}{|l|lllll|lllll|}
						\caption[Summary for \textsc{Normalized}-$k$-\textsc{Cut}]{Summary of the numerical results for \textsc{Normalized}-$k$-\textsc{Cut} ($k=2,3,4,5$) on G-set.}
						\label{tab::normalized-summary}\\
						\hline
						\multicolumn{1}{|c|}{\multirow{2}{*}{}} & \multicolumn{5}{c|}{$k$=2}                    & \multicolumn{5}{c|}{$k$=3}                    \\ \cline{2-11} 
						\multicolumn{1}{|c|}{}                  & wins & ties & losses & p-value          & SS  & wins & ties & losses & p-value          & SS  \\ \hline
						PEAF-ACMH versus ACMH best             & 15   & 20   & 0      & \textless{}0.001 & Yes & 28   & 7    & 0      & \textless{}0.001 & Yes \\
						PEAF-ACMH versus ACMH avg.             & 23   & 12   & 0      & \textless{}0.001 & Yes & 29   & 6    & 0      & \textless{}0.001 & Yes \\
						PEAF-ACMH versus PEAF-MMH best           & 19   & 16   & 0      & \textless{}0.001 & Yes & 24   & 8    & 3      & \textless{}0.001 & Yes \\
						PEAF-ACMH versus PEAF-MMH avg.           & 19   & 15   & 1      & \textless{}0.001 & Yes & 25   & 6    & 4      & \textless{}0.001 & Yes \\
						PEAF-ACMH versus MMH best               & 22   & 13   & 0      & \textless{}0.001 & Yes & 28   & 6    & 1      & \textless{}0.001 & Yes \\
\hline
					\end{longtable}
				\end{landscape}
			}

			\subsection{\textsc{Cheeger}-$k$-\textsc{Cut}}
			\label{sec::cheeger-k-cut}
			
			Tables \ref{tab::cheeger2cut}, \ref{tab::cheeger3cut}, \ref{tab::cheeger4cut} and \ref{tab::cheeger5cut} present the numerical results for \textsc{Cheeger}-$k$-\textsc{Cut} with $k \in \{2, 3, 4, 5\}$. It is evident that our four proposed algorithms---PEAF-ACMH, PEAF-MMH, ACMH, and MMH---consistently outperform \texttt{Gurobi} on all instances except for G70. As summarized in Table~\ref{tab::cheeger-summary}, PEAF-ACMH achieves the overall best performance among the four. Notably, the comparison between PEAF-MMH and MMH validates that the parallel evolutionary framework significantly enhances solution quality.
			
			\textbf{Balance Analysis} 
			Numerical observations indicate that \textsc{Cheeger}-$k$-\textsc{Cut} promotes more balanced partitions compared to \textsc{Normalized}-$k$-\textsc{Cut}. Among the 140 best solutions identified by PEAF-ACMH, \textsc{Normalized}-$k$-\textsc{Cut} yields 124 instances with smaller cut values, whereas \textsc{Cheeger}-$k$-\textsc{Cut} does so in only 3 instances. This suggests that \textsc{Normalized}-$k$-\textsc{Cut} tends to prioritize minimizing the cut value (aligning more closely with \textsc{Min}-$k$-\textsc{Cut}), thereby resulting in more skewed partitions than \textsc{Cheeger}-$k$-\textsc{Cut}.
			
			This behavior can be theoretically justified following the analysis in Section~\ref{sec::result-anticheeger} for $k=2$. Consider a balanced partition $\{S, S^c\}$ where $\vol(S) \simeq \vol(S^c)$. Let $$\cut(\partial S)=\alpha\cdot\left(\frac{1}{2}\vol(V)-\cut(\partial S)\right),$$ which implies 	\begin{align*}
				\textsc{Cheeger-}2\textsc{-Cut}\quad&\simeq\quad \frac{\alpha}{1+\alpha}, \\
				\textsc{Normalized-}2\textsc{-Cut}\quad&\simeq\quad \frac{2\alpha}{1+\alpha}.
			\end{align*} If there exists a skewed cut $\{S_1, S_1^c\}$ satisfying $\vol(S_1) - \cut(\partial S_1) \ll \vol(S_1^c) - \cut(\partial S_1)$ that is superior to $\{S, S^c\}$ for both 2-CUT problems, then $\cut(\partial S_1) < \vol(S_1) \ll \vol(S_1^c)$ and \begin{align*}
			\vol(S_1)&>\frac{1+\alpha}{\alpha}\cut(\partial S_1) \text{ for }\textsc{Cheeger-}2\textsc{-Cut},\\ 
			\vol(S_1)&>\frac{1+\alpha}{2\alpha}\cut(\partial S_1) \text{ for }\textsc{Normalized-}2\textsc{-Cut}. \\
			\end{align*} Given that$$
			\frac{1+\alpha}{\alpha}  > \frac{1+\alpha}{2\alpha}
			$$ holds for any positive $\alpha$, this indicates that \textsc{Cheeger}-2-\textsc{Cut} requires a higher threshold to favor a skewed cut over a balanced one, thus demonstrating intrinsically better balance than \textsc{Normalized}-2-\textsc{Cut}.

			{
				\scriptsize
				\setlength{\tabcolsep}{3pt}
				\renewcommand{\arraystretch}{1.2}
				\begin{landscape}
\begin{longtable}{lllllllllllllllllll}
						\caption[Results for \textsc{Cheeger}-$2$-\textsc{Cut}]{Numerical results for \textsc{Cheeger}-$2$-\textsc{Cut} on G-set.}
						\label{tab::cheeger2cut}\\
						\hline
						\multirow{2}{*}{instance} & \multicolumn{4}{l}{PEAF-ACMH}               & \multicolumn{4}{l}{PEAF-MMH}                & \multicolumn{4}{l}{ACMH}                   & \multicolumn{4}{l}{MMH}                    & \multirow{2}{*}{G-0.5H} & \multirow{2}{*}{G-H} \\ \cline{2-17}
						& mean    & best           & cut   & time      & mean    & best           & cut   & time     & mean    & best           & cut   & time      & mean    & best           & cut   & time     &                         &                      \\ \hline
						G1                        & 0.39581 & \textbf{0.39581} & 7590 & 12.494    & 0.39581 & \textbf{0.39581} & 7590 & 79.560    & 0.39583 & \textbf{0.39581} & 7590 & 130.687   & 0.39581 & \textbf{0.39581} & 7590 & 246.427  & 0.40259                 & 0.40259              \\
						G2                        & 0.39546 & \textbf{0.39546} & 7583 & 70.744    & 0.39546 & \textbf{0.39546} & 7583 & 130.256   & 0.39546 & \textbf{0.39546} & 7583 & 28.726    & 0.39547 & \textbf{0.39546} & 7583 & 54.921   & 0.40172                 & 0.40169              \\
						G3                        & 0.39525 & \textbf{0.39525} & 7579 & 70.160    & 0.39525 & \textbf{0.39525} & 7579 & 14.816    & 0.39531 & \textbf{0.39525} & 7579 & 18.945    & 0.39531 & \textbf{0.39525} & 7579 & 41.521   & 0.40868                 & 0.40816              \\
						G4                        & 0.39586 & \textbf{0.39581} & 7590 & 1697.220  & 0.39581 & \textbf{0.39581} & 7590 & 1094.740  & 0.39587 & 0.39585          & 7590 & 715.504   & 0.39586 & \textbf{0.39581} & 7590 & 1614.210 & 0.40311                 & 0.40311              \\
\hline
					\end{longtable}
				\end{landscape}
			}
			
			{
				\scriptsize
				\setlength{\tabcolsep}{3pt}
				\renewcommand{\arraystretch}{1.2}
				\begin{landscape}
					\begin{longtable}{lllllllllllllllllll}
						\caption[Results for \textsc{Cheeger}-$3$-\textsc{Cut}]{Numerical results for \textsc{Cheeger}-$3$-\textsc{Cut} on G-set.}
						\label{tab::cheeger3cut}\\
						\hline
						\multirow{2}{*}{instance} & \multicolumn{4}{l}{PEAF-ACMH}               & \multicolumn{4}{l}{PEAF-MMH}                & \multicolumn{4}{l}{ACMH}                   & \multicolumn{4}{l}{MMH}                    & \multirow{2}{*}{G-0.5H} & \multirow{2}{*}{G-H} \\ \cline{2-17}
						& mean    & best           & cut   & time      & mean    & best           & cut   & time     & mean    & best           & cut   & time      & mean    & best           & cut   & time     &                         &                      \\ \hline
						G1                        & 0.53762 & \textbf{0.53762} & 10309 & 737.104   & 0.53769 & 0.53764          & 10309 & 90.445   & 0.53802 & 0.53791          & 10314 & 610.905   & 0.53779 & 0.53768          & 10310 & 1662.870 & 0.57836                 & 0.55656              \\
						G2                        & 0.53687 & \textbf{0.53687} & 10294 & 1024.250  & 0.53691 & 0.53690          & 10295 & 1459.120 & 0.53703 & 0.53692          & 10295 & 1222.670  & 0.53717 & 0.53694          & 10295 & 361.673  & 0.57761                 & 0.55578              \\
						G3                        & 0.53687 & \textbf{0.53681} & 10293 & 1326.610  & 0.53683 & \textbf{0.53681} & 10293 & 1317.500 & 0.53725 & 0.53691          & 10294 & 331.025   & 0.53731 & 0.53684          & 10294 & 1245.690 & 0.57836                 & 0.55559              \\
						G4                        & 0.53745 & \textbf{0.53745} & 10306 & 860.750   & 0.53745 & \textbf{0.53745} & 10306 & 998.790  & 0.53772 & 0.53754          & 10306 & 749.040   & 0.53785 & 0.53782          & 10312 & 841.343  & 0.58831                 & 0.55342              \\
\hline
					\end{longtable}
				\end{landscape}
			}
			
			{
				\scriptsize
				\setlength{\tabcolsep}{3pt}
				\renewcommand{\arraystretch}{1.2}
				\begin{landscape}
					\begin{longtable}{lllllllllllllllllll}
						\caption[Results for \textsc{Cheeger}-$4$-\textsc{Cut}]{Numerical results for \textsc{Cheeger}-$4$-\textsc{Cut} on G-set.}
						\label{tab::cheeger4cut}\\
						\hline
						\multirow{2}{*}{instance} & \multicolumn{4}{l}{PEAF-ACMH}               & \multicolumn{4}{l}{PEAF-MMH}                & \multicolumn{4}{l}{ACMH}                   & \multicolumn{4}{l}{MMH}                    & \multirow{2}{*}{G-0.5H} & \multirow{2}{*}{G-H} \\ \cline{2-17}
						& mean    & best           & cut   & time      & mean    & best           & cut   & time     & mean    & best           & cut   & time      & mean    & best           & cut   & time     &                         &                      \\ \hline
						G1                        & 0.61333 & \textbf{0.61318} & 11755 & 1101.800  & 0.61326 & 0.61322          & 11758 & 1142.620  & 0.61343 & 0.61339          & 11762 & 1005.710 & 0.61361 & 0.61327          & 11759 & 1716.700  & 0.63501                 & 0.64605              \\
						G2                        & 0.61237 & \textbf{0.61235} & 11738 & 1027.170  & 0.61241 & 0.61236          & 11740 & 1409.380  & 0.61313 & 0.61281          & 11749 & 1400.470 & 0.61329 & 0.61307          & 11754 & 1542.080  & 0.63451                 & 0.64323              \\
						G3                        & 0.61266 & \textbf{0.61259} & 11746 & 1022.370  & 0.61290 & 0.61287          & 11751 & 805.526   & 0.61308 & 0.61289          & 11751 & 1223.630 & 0.61287 & 0.61268          & 11743 & 735.022   & 0.64173                 & 0.64437              \\
						G4                        & 0.61289 & \textbf{0.61284} & 11750 & 1642.030  & 0.61289 & 0.61288          & 11751 & 935.186   & 0.61374 & 0.61365          & 11767 & 1005.100 & 0.61335 & 0.61289          & 11748 & 445.365   & 0.66354                 & 0.66962              \\
\hline
					\end{longtable}
				\end{landscape}
			}

			{
				\scriptsize
				\setlength{\tabcolsep}{3pt}
				\renewcommand{\arraystretch}{1.2}
				\begin{landscape}
					\begin{longtable}{lllllllllllllllllll}
						\caption[Results for \textsc{Cheeger}-$5$-\textsc{Cut}]{Numerical results for \textsc{Cheeger}-$5$-\textsc{Cut} on G-set.}
						\label{tab::cheeger5cut}\\
						\hline
						\multirow{2}{*}{instance} & \multicolumn{4}{l}{PEAF-ACMH}               & \multicolumn{4}{l}{PEAF-MMH}                & \multicolumn{4}{l}{ACMH}                   & \multicolumn{4}{l}{MMH}                    & \multirow{2}{*}{G-0.5H} & \multirow{2}{*}{G-H} \\ \cline{2-17}
						& mean    & best           & cut   & time      & mean    & best           & cut   & time     & mean    & best           & cut   & time      & mean    & best           & cut   & time     &                         &                      \\ \hline
						G1                        & 0.66058 & \textbf{0.66031} & 12659 & 1437.020  & 0.66083 & 0.66060          & 12665 & 1098.680  & 0.66134 & 0.66088          & 12671 & 171.333   & 0.66149 & 0.66137          & 12681 & 401.009  & 0.71020                 & 0.70249              \\
						G2                        & 0.66038 & \textbf{0.66023} & 12659 & 1341.630  & 0.66040 & 0.66032          & 12659 & 1775.490  & 0.66114 & 0.66104          & 12674 & 748.843   & 0.66089 & 0.66046          & 12660 & 1310.320 & 0.69390                 & 0.70616              \\
						G3                        & 0.66050 & \textbf{0.66045} & 12663 & 1676.870  & 0.66066 & 0.66058          & 12664 & 1018.950  & 0.66099 & 0.66073          & 12669 & 1536.900  & 0.66090 & 0.66046          & 12662 & 1342.960 & 0.68907                 & 0.69768              \\
						G4                        & 0.66091 & \textbf{0.66090} & 12669 & 1664.880  & 0.66112 & 0.66097          & 12671 & 1633.340  & 0.66157 & 0.66126          & 12678 & 669.520   & 0.66135 & 0.66106          & 12673 & 1298.130 & 0.69912                 & 0.72293              \\
\hline
					\end{longtable}%
				\end{landscape}
			}

			\subsection{\textsc{Min}-$k$-\textsc{Cut}}

			%				\blue{The results presented in Tables \ref{tab::min2cut}, \ref{tab::min3cut}, \ref{tab::min4cut} and \ref{tab::min5cut} clearly establish that all the proposed algorithms produce the high-quality solutions for \textsc{Min}-$k$-\textsc{Cut} ($k=2, 3, 4, 5$). The auxiliary cut employed is \textsc{MinMax}-$k$-\textsc{Cut} (see Table \ref{tab::cut-combine}), and all six algorithms yield identical results. Given that \textsc{Min}-$k$-\textsc{Cut} can be solved in polynomial time for a fixed $k$, our proposed algorithms consistently achieve the optimal solution within 20 seconds. Pairwise comparisons (see Table \ref{tab::mincut-summary}) indicate that the performance differences among these algorithms on G-set are statistically insignificant.}

			Tables \ref{tab::min2cut}, \ref{tab::min3cut}, \ref{tab::min4cut} and \ref{tab::min5cut} present the numerical results for \textsc{Min}-$k$-\textsc{Cut} with $k \in \{2, 3, 4, 5\}$. Remarkably, all six algorithms yield identical solutions across all tested instances. Pairwise comparisons (see Table~\ref{tab::mincut-summary}) further confirm that the performance differences between these algorithms are statistically negligible on G-set. Our four proposed algorithms consistently identify the optimal solutions within 20 seconds. The high efficiency observed here is consistent with the theoretical tractability of \textsc{Min}-$k$-\textsc{Cut} for a fixed $k$. The fact that our heuristic framework consistently identifies optimal solutions with such rapidity not only mirrors the problem's polynomial-time nature but also serves as a testament to the high-fidelity search capability of our algorithms.

%			\blue{Tables \ref{tab::min2cut}, \ref{tab::min3cut}, \ref{tab::min4cut} and \ref{tab::min5cut} respectively record the numerical results for \textsc{Min}-$k$-\textsc{Cut} with $k=2, 3, 4, 5$. All six algorithms yield identical results, and pairwise comparisons (see Table \ref{tab::mincut-summary}) also reveal that the performances of these algorithms on G-set are statistically insignificant. Our proposed four algorithms consistently achieve the optimal solution within 20 seconds, and its efficiency is partly attributed to the property that \textsc{Min}-$k$-\textsc{Cut} can be solved in polynomial time for a fixed $k$.}
			
			{
				\scriptsize
				\setlength{\tabcolsep}{8.5pt}
				\renewcommand{\arraystretch}{1.2}
				\begin{landscape}

				\end{landscape}
			}

			\subsection{\textsc{MinMax}-$k$-\textsc{Cut}}
			\label{sec::min-max-k-cut}
			
				Tables \ref{tab::min-max2cut}, \ref{tab::min-max3cut}, \ref{tab::min-max4cut} and \ref{tab::min-max5cut} display the numerical results for \textsc{MinMax}-$k$-\textsc{Cut} with $k \in \{2, 3, 4, 5\}$. It is evident that our four proposed algorithms---PEAF-ACMH, PEAF-MMH, ACMH, and MMH---consistently match or exceed the performance of \texttt{Gurobi} across all instances. In the majority of cases, our framework identifies solutions identical to those of the exact solver, with PEAF-ACMH and ACMH yielding even better results for G58 and G63 at $k=5$. Furthermore, pairwise comparisons (see Table~\ref{tab::min-max-summary}) indicate that the performance differences among the four variants are statistically negligible on G-set.
			
			A comparative analysis between PEAF-ACMH results for \textsc{Min}-$k$-\textsc{Cut} and \textsc{MinMax}-$k$-\textsc{Cut} reveals a striking similarity: the absolute difference in their cut values is at most 4 (e.g., G52 in Tables~\ref{tab::min5cut} and \ref{tab::min-max5cut}). This minimal discrepancy, combined with the exceptionally low computational cost required to reach these solutions, provides compelling numerical evidence of the underlying structural affinity between the two problems. While the polynomial-time solvability of \textsc{Min}-$k$-\textsc{Cut} was established decades ago \cite{goldschmidt1994polynomial}, the corresponding complexity for \textsc{MinMax}-$k$-\textsc{Cut} was only recently confirmed \cite{chandrasekaran2023fixed}. Our results are entirely consistent with this theoretical landscape, demonstrating that an efficient, high-fidelity heuristic framework can serve as a powerful tool for uncovering and validating the theoretical properties of combinatorial $k$-CUT structures.

			{
				\scriptsize
				\setlength{\tabcolsep}{8.5pt}
				\renewcommand{\arraystretch}{1.2}
				\begin{landscape}

				\end{landscape}
			}

			\subsection{\textsc{Ratio}-$k$-\textsc{Cut}}
			\label{sec::ratio-k-cut}

			%				\blue{Tables \ref{tab::ratio2cut}, \ref{tab::ratio3cut}, \ref{tab::ratio4cut} and \ref{tab::ratio5cut} present the outcomes produced by the algorithms PEAF-ACMH, PEAF-MMH, ACMH, MOH, G-0.5H and G-H, utilizing \textsc{Min}-$k$-\textsc{Cut} and \textsc{MinMax}-$k$-\textsc{Cut} as auxiliary cuts, indicating that across all graph instances, the four proposed algorithms perform at least as well as Gurobi. Pairwise comparisons in Table \ref{tab::ratio-summary} reveal that the performance of the four algorithms is generally similar when $k=2,3$. However, for $k=4,5$, PEAF-ACMH demonstrates superior performance compared to ACMH in terms of both the best and average objective function values, while PEAF-MMH outperforms MOH regarding the average objective function values. This suggests that the integration of the parallel evolutionary framework can enhance solution quality to a notable extent.}
			
			Tables \ref{tab::ratio2cut}, \ref{tab::ratio3cut}, \ref{tab::ratio4cut} and \ref{tab::ratio5cut} present the numerical results for \textsc{Ratio}-$k$-\textsc{Cut} with $k \in \{2, 3, 4, 5\}$. Notably, our proposed four algorithms---PEAF-ACMH, PEAF-MMH, ACMH, and MMH---consistently match or exceed the performance of \texttt{Gurobi} across all tested instances. Pairwise comparisons in Table~\ref{tab::ratio-summary} indicate that while the four variants perform similarly for $k=2$ and $3$, a clear distinction emerges as $k$ increases. For $k=4$ and $5$, PEAF-ACMH significantly outperforms its non-evolutionary counterpart, ACMH, thereby validating the effectiveness of the parallel evolutionary framework in further enhancing solution quality.
			
			Furthermore, a comparative analysis reveals that for certain graph instances (e.g., G22--G26 and G41--G47), the best cut values for \textsc{Ratio}-$k$-\textsc{Cut} are remarkably close to those of \textsc{Min}-$k$-\textsc{Cut}. These high-quality solutions are identified within a very short timeframe, suggesting a strong structural affinity between \textsc{Ratio}-$k$-\textsc{Cut} and \textsc{Min}-$k$-\textsc{Cut} on specific graph topologies.

%			Tables \ref{tab::ratio2cut}, \ref{tab::ratio3cut}, \ref{tab::ratio4cut} and \ref{tab::ratio5cut} display the numerical results for \textsc{Ratio}-$k$-\textsc{Cut} with $k=2,3,4,5$. It can be easily observed that PEAF-ACMH, PEAF-MMH, ACMH and MMH match or exceed the performance of \texttt{Gurobi} on all graph instances. From pairwise comparisons in Table~\ref{tab::ratio-summary}, the performances of the four proposed algorithms are similar for $k=2,3$. In terms of larger values of $k$, i.e. $k=4,5$, PEAF-AMCH demonstrates superior performance compared to ACMH, indicating that integrating the evolutionary framework can enhance the solution quality. 
%			
%		Additionally, a comparative analysis of the results for \textsc{Ratio}-$k$-\textsc{Cut} and \textsc{Min}-$k$-\textsc{Cut} show that, for certain graph instances (e.g., G22$\sim$G26, G41$\sim$G47), the corresponding cut values of their best solutions are remarkably close and are produced in very short time. This similarity implies that \textsc{Ratio}-$k$-\textsc{Cut} and \textsc{Min}-$k$-\textsc{Cut} may have noteworthy resemblance in their structural properties.

			{
				\scriptsize
				\setlength{\tabcolsep}{3.5pt}
				\renewcommand{\arraystretch}{1.2}
				\begin{landscape}

				\end{landscape}
			}

			\subsection{\textsc{Sparsest}-$k$-\textsc{Cut}}
			\label{sec::sparsest-k-cut}
			
				Tables \ref{tab::sparsest2cut}, \ref{tab::sparsest3cut}, \ref{tab::sparsest4cut} and \ref{tab::sparsest5cut} display the numerical results for \textsc{Sparsest}-$k$-\textsc{Cut} with $k \in \{2, 3, 4, 5\}$. Notably, our proposed  algorithms---PEAF-ACMH, PEAF-MMH, ACMH, and MMH---consistently outperform \texttt{Gurobi} across all tested instances. As indicated by the pairwise comparisons in Table~\ref{tab::cheeger-summary}, PEAF-ACMH achieves the superior performance among the four variants.
			
			\textbf{Balance Analysis}
			Numerical results suggest that \textsc{Sparsest}-$k$-\textsc{Cut} promotes more balanced partitions compared to \textsc{Ratio}-$k$-\textsc{Cut}. Among the 140 best solutions identified by PEAF-ACMH, \textsc{Ratio}-$k$-\textsc{Cut} yields smaller cut values in all instances, 112 of which are strictly smaller. This discrepancy confirms that \textsc{Ratio}-$k$-\textsc{Cut} tends to prioritize minimizing the cut value (aligning more closely with \textsc{Min}-$k$-\textsc{Cut}), thereby resulting in more skewed partitions than \textsc{Sparsest}-$k$-\textsc{Cut}.
			
			This behavior can be theoretically justified using an analysis similar to that in Section~\ref{sec::result-anticheeger} for $k=2$. Assume a balanced partition $\{S, S^c\}$ where $|S| \simeq |S^c|$ for both problems. Let $$\cut(\partial S)=\alpha\cdot\frac{1}{2}|V|,$$ 
			which leads to 	\begin{align*}
				\textsc{Sparsest-}2\textsc{-Cut}\quad&\simeq\quad 2\alpha, \\
				\textsc{Ratio-}2\textsc{-Cut}\quad&\simeq\quad 2\alpha.
			\end{align*} 
			If there exists a skewed cut $\{S_1, S_1^c\}$ satisfying $|S_1| \ll |S_1^c| \simeq |V|$ that is superior to $\{S, S^c\}$ for both variants, then \begin{align*}
				|S_1|&>\frac{1}{\alpha}\cut(\partial S_1) \text{ for }\textsc{Sparsest-}2\textsc{-Cut},\\ 
				|S_1|&>\frac{1}{2\alpha}\cut(\partial S_1) \text{ for }\textsc{Ratio-}2\textsc{-Cut}. \\
			\end{align*} This implies that \textsc{Sparsest}-$k$-\textsc{Cut} possesses a higher resistance to partition skewness, effectively maintaining better balance than \textsc{Ratio}-$k$-\textsc{Cut}, since$$
			\frac{1}{\alpha}  > \frac{1}{2\alpha}
			$$ holds for any $\alpha > 0$.

			{
				\scriptsize
				\setlength{\tabcolsep}{2pt}
				\renewcommand{\arraystretch}{1.2}
				\begin{landscape}

				\end{landscape}
			}

			\section{Conclusion and outlook}
			\label{sec::conclusion}

			We present a generalized Parallel Evolutionary Algorithm Framework (PEAF) for solving a diverse suite of graph $k$-CUT problems. The framework integrates three synergized components: (1) a crossover stage with five operators ensuring that offspring inherit essential structural traits from reference solutions; (2) a hierarchical mutation engine featuring the Multiple Mutation Heuristic (MMH) and its auxiliary cut-combined variant ACMH, which leverage nine local search operators to refine the computational complexity more efficient than $\mathcal{O}(k|V|)$; and (3) a selection mechanism that balances population quality and genetic diversity. Numerical experiments on nine widely studied $k$-CUT problems with $k \in \{2, 3, 4, 5\}$ demonstrate that our proposed algorithms consistently match or exceeds the performance of \texttt{Gurobi} across the majority of G-set instances. 
			
			The versatility of PEAF provides a robust foundation for addressing more complex graph partitioning variants. Potential extensions include:
			
		\begin{enumerate}
				\item \textbf{Non-integer Weighted Graphs:} For graphs with positive non-integer weights, the discrete bucket sorting mechanism can be adapted into an interval-based framework. By partitioning the range $[-d_{\max}, d_{\max}]$ into $N$ sub-intervals 	\[
				\begin{aligned}
					1:&\quad\left(b_{\text{bound}}-\frac{2}{N}b_{\text{bound}},b_{\text{bound}}\right],\\
					\vdots&\\
					i:&\quad\left(b_{\text{bound}}-\frac{2i}{N}b_{\text{bound}},b_{\text{bound}}-\frac{2i-2}{N}b_{\text{bound}}\right],\\
					\vdots&\\
					N:&\left[-b_{\text{bound}},-b_{\text{bound}}+\frac{2}{N}b_{\text{bound}}\right].
				\end{aligned}
				\] a doubly-linked list structure can manage vertices based on their move-gain values falling within specific intervals, maintaining efficient local search.				
				\item \textbf{Signed Graphs:} To handle signed graphs with mixed edge weights, the definitions of boundary $|\partial S_p|$ and volume $\vol(S_p)$ must be generalized to incorporate both positive ($w_{uv}^+$) and negative ($w_{uv}^-$) contributions \cite{chiang2012scalable}. Specifically, the boundary $\partial S_p$ and volume $\vol(S_p)$ can be redefined as shown in $$|\partial S_p|=\sum_{u\in S_p,\,v\in S_p^c\text{ and }\{u,v\}\in E}w_{uv}^+\,\,\,\text{ and }\,\,\,\vol(S_p)=\sum_{u\in S_p}\sum_{\{u,v\}\in E}(w_{uv}^++w_{uv}^-),$$ respectively. Since the change in cut values during a single-transfer operation remains localized to the moved vertex and its neighbors, our bucket sorting framework can be extended to maintain stable move-gain matrices for signed structures.				
				\item \textbf{Constrained $k$-Cut Problems:} For variants with additional constraints, such as size balance limits $|S_p| \leq (1+\nu)|V|/k$, $\forall\, p\in [k]$, $\nu>0$, \cite{bourse2014balanced}, the framework can incorporate a "Relax-and-Restore" strategy. This involves a feasible modification phase using a step-by-step greedy migration to restore feasibility in the crossover and mutation phase. Specifically, for any solution $I=\{S_1,\ldots,S_k\}$, let $U_I=\{p\in [k]:\,|S_p|> (1+\nu)|V|/k\}$ be the set of unsatisfied constraints, then we iteratively move the vertex by $$I^\prime=I\circ v\rightarrow S_q \text{ with } v,\,q\in\argopt_{v\in \cup_{p\in U_I} S_p,\,q\in [k]\backslash U_I}\{f(I\circ v\rightarrow S_q)\},$$
				to make $I^\prime$ feasible. Alternating between unconstrained relaxation and feasibility refinement allows the algorithm to iteratively converge toward high-quality feasible solutions.
			\end{enumerate}

			\section*{Acknowledgement}
			This work was funded by the National Key R \& D Program of China (No. 2022YFA1005102) and
				the National Natural Science Foundation of China (Nos. 12526521, 12325112, 12288101). The authors would like to express their sincere gratitude to Dr. Chen Cheng for his participation in the discussions on the Max-Cut part during the early stage of this work, when he was an undergraduate student (September 2015 -- June 2019) at Peking University.
			%\bibliography{journalname,graph}
			%\bibliographystyle{alpha}
			
			%\bibliographystyle{plain}
			
			\bibliography{comb_framework_2024}
			
			\appendix
			\section{Appendix: Proofs of Proposition \ref{thm::o1comp}}
			\label{app:proof}

		To establish the computational complexities of the local search operators, we partition the proof into three critical components: the dynamics of move-gain updates, the efficiency of the bucket sorting data structure, and the incremental evaluation of the objective function.
		
		\paragraph{1. Move-Gain Update Dynamics}
		Consider a single-transfer move $S_{p_v} \to v \to S_{q_v}$ ($p_v \neq q_v$). The update rules for the stable move-gain matrices $\widetilde{\mathcal{D}}$ and $\widehat{\mathcal{D}}$ are characterized by their temporal and spatial locality:
		\begin{enumerate}
			\item \textbf{Vertex $v$:} The gain values directly associated with the moved vertex $v$ are updated according to $$ \text{new-}\widetilde{\mathcal{D}}_{vi}=\left\{
			\begin{aligned}
				&-\widetilde{\mathcal{D}}_{vi},& \text{if }i=p_v,\\
				& 0,&\text{if }i=q_v,\\
				&\widetilde{\mathcal{D}}_{vi}-\widetilde{\mathcal{D}}_{v{q_v}},&\text{otherwise},
			\end{aligned}
			\right.
			\quad 
			\text{new-}\widehat{\mathcal{D}}_{vi}=
			\left\{
			\begin{aligned}
				&-\widehat{\mathcal{D}}_{vi},& \text{if } i=p_v,\\
				&-\widehat{\mathcal{D}}_{vi},& \text{if } i=q_v,\\
				& \widehat{\mathcal{D}}_{vi}, & \text{otherwise}.
			\end{aligned}
			\right.
			$$	
			\item \textbf{Neighbors of $v$:} For each neighbor $u \in \mathcal{N}(v)$, the gain for any single-transfer $S_{p_u} \to u \to S_{q_u}$ ($p_u \neq q_u$) is adjusted via 	\begin{align}
				\text{new-}\widetilde{\mathcal{D}}_{u{q_u}}&=\widetilde{\mathcal{D}}_{u{q_u}}+\widetilde{\phi} w_{uv},	\label{adjust::dcut}\\
				\text{new-}\widehat{\mathcal{D}}_{ui}&=\widehat{\mathcal{D}}_{ui}+\widehat{\phi}^{\{i\}} w_{uv},\,\, i\in\{p_u,\,q_u,\,p_v,\,q_v\},	\label{adjust::dpartial}
			\end{align} where the specific update terms $\widetilde{\phi}$ and $\widehat{\phi}^{\{i\}}$ are detailed in Table~\ref{tab::d2phi}.
		\end{enumerate}
		Since the update only affects $v$ and its neighborhood $\mathcal{N}(v)$, at most $k(d_v+1)$ elements are modified. Given that each element requires only $\mathcal{O}(1)$ time, the total update cost is bounded by $\mathcal{O}(kd_v) \le \mathcal{O}(kd_{\max})$.
		
		\paragraph{2. Efficient Storage and Query via Bucket Sorting}
		To facilitate the rapid identification of the primary search area $\mathbb{D}(I)$, we utilize a bucket sorting scheme (see Figures~\ref{fig::dcut} and \ref{fig::dpartial}). 
		\begin{itemize}
			\item \textbf{Structure:} Each matrix $\mathcal{D} \in \{\widetilde{\mathcal{D}}, \widehat{\mathcal{D}}\}$ is stored in $k$ chief arrays. Each array $B_q$ contains $2d_{\max}+1$ nodes, where each node $B_{qi}$ acts as the head of a doubly-linked list for vertices with gain $\mathcal{D}_{vq} = d_{\max} + 1 - i$. 
			\item \textbf{Search Area Retrieval:} By maintaining pointers $b_{\max}^q$ and $b_{\min}^q$ to the first and last non-empty nodes in each $B_q$, the search area $\mathbb{D}(I)$ is reformulated as 	\begin{equation}
				\label{eq::normalized-move}
				\mathbb{D}(I)=\left\{
				\begin{aligned}
					&\bigcup_{q\in[k]\backslash\{\labeling(v,I)\}}\{(v,q):\,\mathcal{D}_{vq}=b_{\max}^q\},&\text{if }\opt\,F\in\text{MaxGCP},\\
					&\bigcup_{q\in[k]\backslash\{\labeling(v,I)\}}\{(v,q):\,\mathcal{D}_{vq}=b_{\min}^q\},&\text{if }\opt\,F\in\text{MinGCP},
				\end{aligned}
				\right.
				\quad \mathcal{D}\in\{\widetilde{\mathcal{D}},\,\widehat{\mathcal{D}}\}.
			\end{equation}		
		\item \textbf{Implementation:} We implement these doubly-linked lists using two matrices $\mathcal{M}^1, \mathcal{M}^2 \in \mathbb{N}^{k \times |V|}$, where $\mathcal{M}_{iq}^1$ and $\mathcal{M}_{iq}^2$ store the predecessor and successor of vertex $i$ within bucket $B_q$, respectively. As illustrated in Figures~\ref{fig::dcut} and \ref{fig::dpartial}, moving a vertex within the bucket array entails a standard deletion and insertion sequence. These operations are executed through a fixed number of element reassignments in $\mathcal{M}^1$ and $\mathcal{M}^2$, resulting in a constant-time complexity of $\mathcal{O}(1)$ per move. Consequently, the total overhead for synchronizing the bucket structures is strictly dictated by the number of affected elements in the move-gain matrices, which remains bounded by $\mathcal{O}(kd_{\max})$.
		\end{itemize}
		
		\paragraph{3. Incremental Evaluation of the Objective Function}
		Let $c$ denote the marginal cost of computing the change in the objective function $F(I)$ resulting from a move. In this framework, we avoid full re-evaluation. 
		Taking \textsc{Normalized}-$3$-\textsc{Cut} as an example (Figure~\ref{fig::example}), we maintain the current boundary sizes $\text{cut}(\partial \vec{S})$ and volumes $\text{vol}(\vec{S})$. When vertex $e$ moves from $S_2$ to $S_1$, the new objective value is computed via $$F(I\circ (e\rightarrow S_1))=\frac{\cut(\partial S_1)+\widehat{\mathcal{D}}_{51}}{\vol(S_1)+d_5}+\frac{\cut(\partial S_2)+\widehat{\mathcal{D}}_{52}}{\vol(S_2)-d_5}+\frac{\cut(\partial S_3)+\widehat{\mathcal{D}}_{53}}{\vol(S_3)}.$$ This incremental approach requires only $\mathcal{O}(k)$ operations, regardless of the graph size $|V|$, thus $c = \mathcal{O}(k)$ is satisfied for all nine $k$-CUT problems in this paper.
		
		\paragraph{4. Synthesis of Total Complexity}
		Based on the above, the total complexity of the search operators is synthesized as follows:
		\begin{enumerate}
			\item \textbf{Single-transfer ($\widetilde{O}_1, \widehat{O}_1$):} The cost is the maximum of (a) searching $\mathbb{D}(I)$~\eqref{eq::normalized-move}, which is $\mathcal{O}(c|\mathbb{D}(I)|)$, and (b) updating the gain matrices and pointers, which is $\mathcal{O}(kd_{\max})$. Since $c = \mathcal{O}(k)$, the total cost is:
			$$\mathcal{O}(\max\{k|\mathbb{D}(I)|, kd_{\max}\}) = \mathcal{O}(\max\{k|\mathbb{D}(I)|, kd_{\max}\}).$$			
			\item \textbf{Double-transfer ($\widetilde{O}_2, \widehat{O}_2$):} These operators explore move combinations within $\mathbb{D}(I)$ and their neighborhoods, leading to at most $|\mathbb{D}(I)|d_{\max}$ objective evaluations. Combined with the matrix update costs, the total complexity is:
			$$\mathcal{O}(\max\{kd_{\max}, c|\mathbb{D}(I)|d_{\max}\}) = \mathcal{O}(kd_{\max}|\mathbb{D}(I)|).$$
		\end{enumerate}
	
		In sparse graphs, where $|\mathbb{D}(I)| \ll |V|$, these bounds demonstrate that the PEAF framework maintains high efficiency by confining the search and update operations to the localized boundaries of the partitions.

			\begin{figure}[htbp]
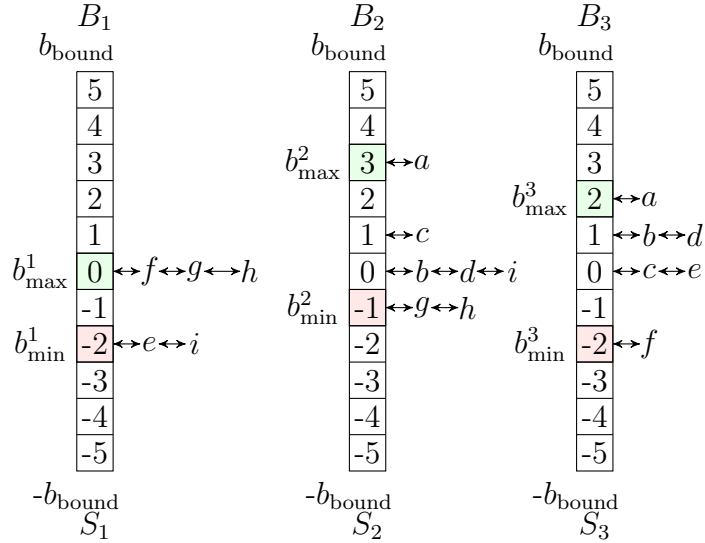

				\centering

				\caption{The update in the bucket sorting arrays of $\widetilde{\mathcal{D}}$ by applying $\widetilde{O}_1$ on the solution $I=\{S_1,S_2,S_3\}$ with $S_1=\{a,b,c,d\}$, $S_2=\{e,f\}$ and $S_3=\{g,h,i\}$. (1) 
					move $B_2:\,g$ to $B_3:\,1$; (2) move $B_1:\,g$ to $B_1:\,1$; (3) move $B_1:\,f$ to $B_1:\,1$;
					(4) move $B_3:\,f$ to $B_3:\,0$; (5) remark $b_{\max}^1=1$, $b_{\min}^3=0$.
				}
				\label{fig::dcut}
			\end{figure}
			
			\begin{figure}[htbp]
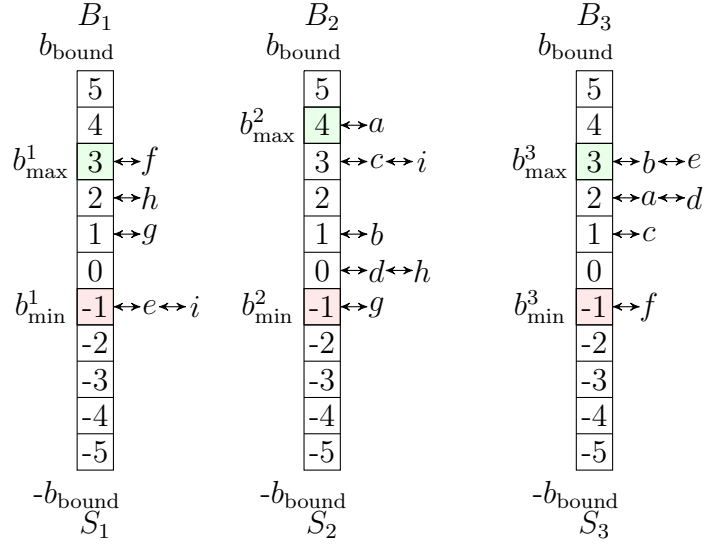

				\centering
				%
				\caption{The update in the bucket sorting arrays of $\widehat{\mathcal{D}}$ by applying $\widehat{O}_1$ on the solution $I=\{S_1,S_2,S_3\}$ with $S_1=\{a,b,c,d\}$, $S_2=\{e,f\}$ and $S_3=\{g,h,i\}$. (1) 
					move $B_2:\,g$ to $B_3:\,1$;
					(2) move $B_3:\,f$ to $B_3:\,1$; (3) remark $b_{\min}^3=1$, $b_{\min}^2=0$.}
				\label{fig::dpartial}
			\end{figure}

		\end{document}